\documentclass[reqno,draft]{amsart}

\theoremstyle{plain}
\newtheorem{lemma}{Lemma}
\newtheorem{theorem}[lemma]{Theorem}

\newtheorem{proposition}[lemma]{Proposition}
\newtheorem{definition}[lemma]{Definition}

\theoremstyle{remark}
\newtheorem{remark}{Remark}

\newcommand*  {\N} {{\mathbb N}}

\makeatletter
\newcommand*{\rom}[1]{\expandafter\@slowromancap\romannumeral #1@}
\makeatother

\def\ee{\epsilon}

\def\Dd{\Delta}

\def\Om{\Omega}

\def\pp{\partial}

\begin{document}

\vskip 0.125in

\title[Primitive Geostrophic Adjustment Model]
{On the Well--posedness of Reduced $3D$ Primitive Geostrophic Adjustment Model With Weak Dissipation}

\date{March 24, 2019}

\author[C. Cao]{Chongsheng Cao}
\address[C. Cao]
{Department of Mathematics  \\
Florida International University  \\
University Park  \\
Miami, FL 33199, USA.} \email{caoc@fiu.edu}

\author[Q. Lin]{Quyuan Lin}
\address[Q. Lin]
{Department of Mathematics  \\
	Texas A\&M University  \\
	College Station  \\
	Texas, TX 77840, USA.} \email{abellyn@math.tamu.edu}

\author[E.S. Titi]{Edriss S. Titi}
\address[E.S. Titi]
{Department of Mathematics  \\
	Texas A\&M University  \\
	College Station  \\
	Texas, TX 77840, USA.  Department of Applied Mathematics and Theoretical Physics\\ University of Cambridge\\
Wilberforce Road, Cambridge CB3 0WA, UK.
 Department of Computer Science and Applied Mathematics \\
Weizmann Institute of Science  \\
Rehovot 76100, Israel.} \email{titi@math.tamu.edu} \email{Edriss.Titi@damtp.cam.ac.uk}
\email{edriss.titi@weizmann.ac.il}

\begin{abstract}
In this paper we prove the local well-posedness and global well-posedness with small initial data of the strong solution to the reduced $3D$ primitive geostrophic adjustment model with weak dissipation. The term reduced model stems from the fact that the relevant physical quantities depends only on two spatial variables. The additional weak dissipation helps us overcome the ill-posedness of original model. We also prove the global well-posedness of the strong solution to the Voigt $\alpha$-regularization of this model, and establish the convergence of the strong solution of the Voigt $\alpha$-regularized model to the corresponding solution of original model. Furthermore, we derive a criterion for finite-time blow-up of reduced $3D$ primitive geostrophic adjustment model with weak dissipation based on Voigt $\alpha$-regularization.
\end{abstract}

\maketitle



\section{Introduction}   \label{S-1}

It is commonly believed that the dynamics of ocean and atmosphere
adjusts itself toward a geostrophic balance. The following reduced $3D$
primitive geostrophic adjustment model is a main type diagnostic
model for studying geostrophic adjustment (cf. e.g., \cite{GI82}, \cite{HO04}, \cite{PZ05}):
\begin{eqnarray}
&&\hskip-.8in u_t + u\, u_x + w u_z -f_0 \, v+
p_x = 0 , \label{EQ-1-2}  \\
&&\hskip-.8in v_t + u\, v_x + w
v_z +f_0\, u = 0 , \label{EQ-2-2}  \\
&&\hskip-.8in
p_z  +  T =0 ,   \label{EQ-3-2}  \\
&&\hskip-.8in
u_x + w_z =0 ,  \label{EQ-4-2} \\
&&\hskip-.8in T_t  + u \, T_x + w \, T_z =  0, \label{EQ-5-2}
\end{eqnarray}
where the velocity field $(u, v, w)$, the temperature $T$ and the pressure $p$ are the unknown functions of the variable $(x, z,t)$, i.e., they depend only on two spatial variables, and $f_0$ is the Coriolis parameter. System (\ref{EQ-1-2})--(\ref{EQ-5-2}) is reduced from the $3D$ inviscid primitive equations model by assuming that the flow is independent of the $y$ variable. This system has been a standard framework to study geostrophic adjustment of frontal anomalies in a rotating continuously stratified fluid of strictly rectilinear fronts and jets (cf. e.g., \cite{BL72}, \cite{GI76}, \cite{GI82}, \cite{HO93}, \cite{HO04}, \cite{KP97}, \cite{PZ05}, \cite{RO38} and references therein).

The first systematically mathematical studies of the viscous primitive equations (PEs) were carried out in the 1990s by Lions--Temam--Wang \cite{LTW92a,LTW92b,LTW95}, where they considered the PEs
with both full viscosities and full diffusivity, and established the global existence of weak solutions. The uniqueness of weak solutions in the $3D$ viscous case is still an open problem, while the weak solutions in $2D$ turn out to be unique, see Bresch, Guill\'en-Gonz\'alez, Masmoudi and Rodr\'iguez-Bellido \cite{BGMR03}. Concerning the strong solutions for the $2D$ case, the local existence result was established by Guill\'en-Gonz\'alez, Masmoudi and Rodr\'iguez-Bellido \cite{GMR01}, while the global existence for $2D$ case was achieved by Bresch, Kazhikhov
and Lemoine in \cite{BKL04}, and by Temam and Ziane in \cite{TZ03}. The global existence of strong solutions for $3D$ case was established by Cao and Titi in \cite{CT07} and later by Kobelkov in \cite{K06}, see also the subsequent articles of Kukavica and Ziane \cite{KZ07,KZ072}, for different boundary conditions,  as well as Hieber and Kashiwabara \cite{Hieber-Kashiwabara} for some progress towards relaxing the smoothness on the initial data by using the semigroup method.  On the other hand, it has already been proven that smooth solutions to the inviscid $2D$ or $3D$ PEs, with or without coupling to the temperature equation, can develop singularities in finite time, see Cao et al. \cite{CINT15} and Wong \cite{W12}. Motivated by physical considerations, it is of great interest to investigate the global well-posedness, finite-time blow-up, or even ill-posedness of the PEs with only partial viscosities or partial diffusivity. There has been several works concerning the mathematical study of these models. The global existence and uniqueness of strong solutions for the PEs with full viscosities and with either horizontal or vertical diffusivity have been established by Cao--Titi \cite{CT12} and Cao--Li--Titi \cite{CLT14a,CLT14b}. Concerning partial viscosities, strong solutions are shown to be unique and exist globally in time for the PEs with only horizontal viscosity and only horizontal diffusivity for any initial data in $H^2$ by Cao--Li--Titi \cite{CLT16} (see Cao--Li--Titi \cite{CLT17} for some generalization of the result in \cite{CLT16}), and for the PEs with only horizontal viscosity and only vertical diffusivity by Cao--Li--Titi \cite{CLT17b}. On the other hand, there is no results concerning the well-posedness or finite-time blow-up for the PEs with only vertical viscosity, even for the two-dimensional case. In this paper, we are interested in this situation. More specifically, we are interested in system (\ref{EQ-1-2})--(\ref{EQ-5-2}) with only vertical viscosity and full diffusivity:
\begin{eqnarray}
&&\hskip-.8in u_t -\nu u_{zz}+ u\, u_x + w u_z -f_0 \, v+
p_x = 0 , \label{EQ-1-5}  \\
&&\hskip-.8in v_t -\nu v_{zz} + u\, v_x + w
v_z +f_0\, u = 0 , \label{EQ-2-5}  \\
&&\hskip-.8in
p_z  +  T =0 ,   \label{EQ-3-5}  \\
&&\hskip-.8in
u_x + w_z =0 ,  \label{EQ-4-5} \\
&&\hskip-.8in T_t -\kappa \Dd T + u \, T_x + w \, T_z =  0. \label{EQ-5-5}
\end{eqnarray}
The main mathematical difficulty to prove well-posedness of system (\ref{EQ-1-5})--(\ref{EQ-5-5}) is the lack of control over the horizontal derivatives. The situation with only vertical viscosity, i.e. system (\ref{EQ-1-5})--(\ref{EQ-5-5}), is worse than the situation with only horizontal viscosity, for which the global well-posedness was established. Indeed, system (\ref{EQ-1-2})--(\ref{EQ-5-2}) and system (\ref{EQ-1-5})--(\ref{EQ-5-5}) turn out to be ill-posed. By considering system (\ref{EQ-1-2})--(\ref{EQ-5-2}) and system (\ref{EQ-1-5})--(\ref{EQ-5-5}) without the Coriolis force, the velocity in $y$ direction and the temperature, i.e., by setting $f_0=0$, $v\equiv 0$ and $T\equiv 0$, we end up with the so-called $2D$ hydrostatic Euler equations:
 \begin{eqnarray}
 &&\hskip-.8in u_t + u\, u_x + w u_z +
 p_x = 0 , \label{EQ-1-3}  \\
 &&\hskip-.8in
 p_z  =0 ,   \label{EQ-2-3}  \\
 &&\hskip-.8in
 u_x + w_z =0 ,  \label{EQ-3-3}
 \end{eqnarray}
and $2D$ hydrostatic Navier-Stokes equations:
\begin{eqnarray}
&&\hskip-.8in u_t + u\, u_x + w u_z +
p_x - \nu u_{zz}= 0 , \label{EQ-1-4}  \\
&&\hskip-.8in
p_z  =0 ,   \label{EQ-2-4}  \\
&&\hskip-.8in
u_x + w_z =0 ,  \label{EQ-3-4}
\end{eqnarray}
correspondingly. The linear ill-posedness of system (\ref{EQ-1-3})--(\ref{EQ-3-3}), about certain shear-flows, has been established by Renardy in \cite{RE09}. Furthermore, the author of \cite{RE09} has also indicated, without providing details, that one should be able to show the linear ill-posedness of system (\ref{EQ-1-4})--(\ref{EQ-3-4}) in any Sobolev space by means of matched asymptotics.  On the other hand, the nonlinear ill-posedness of system (\ref{EQ-1-3})--(\ref{EQ-3-3}) has been established by Han-Kwan and Nguyen in \cite{HN16}, where they built an abstract framework to show the hydrostatic Euler equations are ill-posed in any Sobolev space. One might be able to argue that the main reason for the ill-posedness in these is again due the lack of control over the horizontal derivatives. From a mathematical perspective, system (\ref{EQ-1-4})--(\ref{EQ-3-4}) is reminiscent of the 2D Prandtl system, in the upper half space. The ill-posedness in Sobolev spaces of Prandtl system was established by G\'erard-Varet and Dormy \cite{GD10}, and by G\'erard-Varet and Nguyen \cite{GN12}. In \cite{EE97} , the author established finite-time blow-up of solutions of Prandtl system. On the other hand, the only well-posedness results of hydrostatic Euler equations and Prandtl system can be obtained by assuming either real analyticity or some special structures on the initial data \cite{BR99,BR03,GR99,KMVW14,KTVZ11,KV13,MW12,MW15,OL66}. Recently, the authors in \cite{GMV18} showed the local well-posedness for system (\ref{EQ-1-4})--(\ref{EQ-3-4}) with convex initial data in Gevrey class, which is slightly larger than analyticity class.

All the results and discussions above suggests that, in order to prove the well-posedness for system (\ref{EQ-1-5})--(\ref{EQ-5-5}) in Sobolev spaces, without assuming any special structures, some additional horizontal dissipation terms are necessary. In the derivation of system (\ref{EQ-1-4})--(\ref{EQ-3-4}) from the Navier-Stokes equations, only vertical viscosity is survived. This suggests that the strong dissipation, i.e., horizontal viscosity, is not always expected to exist. Instead, we will consider some weaker dissipation. Inspired by Samelson and Vallis \cite{SV97}, and Salmon \cite{S98}, by introducing the linear (Rayleigh-like friction) damping in both horizontal momentum equation and vertical hydrostatic approximation to system (\ref{EQ-1-5})--(\ref{EQ-5-5}), we consider the following model:
\begin{eqnarray}
&&\hskip-.8in u_t + u\, u_x + w u_z + \ee_1 u -f_0 \, v+
p_x -\nu u_{zz}= 0 , \label{EQO-1}  \\
&&\hskip-.8in v_t + u\, v_x + w
v_z + \ee_1 v +f_0\, u -\nu v_{zz} = 0 , \label{EQO-2}  \\
&&\hskip-.8in
\ee_2 w + p_z  +  T =0 ,   \label{EQO-3}  \\
&&\hskip-.8in
u_x + w_z =0 ,  \label{EQO-4} \\
&&\hskip-.8in T_t -\kappa \Dd T + u \, T_x + w \, T_z =  0, \label{EQO-5}
\end{eqnarray}
in the horizontal channel $\{(x,z): 0\leq z\leq H, x\in \mathbb{R}\}$, with the following boundary conditions:
\begin{eqnarray}
&&\hskip-.8in
(u_z, v_z, w, T)|_{z=0,H}=0, \nonumber  \\
&&\hskip-.8in
u, v, w, T \;\text{are periodic in} \; x \; \text{with period} \; 1. \label{OBC}
\end{eqnarray}
Here $\ee_1$ and $\ee_2$ are positive constants representing the linear (Rayleigh-like friction) damping coefficients, and $\nu$ is positive constant which stands for the vertical viscosity of the horizontal momentum equations. Since it is believed that the adjustment time is not very long, friction may be assumed to have negligible influence on the subsequent development of the flow. However, in this article, we show that the friction plays important roles in the well-posedness of the adjustment system. More specifically, the damping term $\ee_2 w$ in (\ref{EQO-3}) plays an important role as a weak dissipation in horizontal direction through $w_z$ and the incompressibility (\ref{EQO-4}). Unlike the case with strong horizontal dissipation, i.e., with horizontal viscosity, the global well-posedness of system (\ref{EQO-1})--(\ref{EQO-5}) is still a difficult task. However, in this paper, we are able to prove the local well-posedness and global well-posedness with small initial data. We refer the reader to \cite{CT10} where the authors established the global well-posedness of the $3D$ Salmon's planetary geostrophic oceanic dynamics model which involves similar damping term in the hydrostatic equations.

In order to study the possible finite-time blow-up of system (\ref{EQO-1})--(\ref{EQO-5}), and to give a reliable numerical model/scheme to system (\ref{EQO-1})--(\ref{EQO-5}), we also study the Voigt $\alpha$-regularization of system (\ref{EQO-1})--(\ref{EQO-5}), with the regularization only in the $z$ variable. More specifically, we consider the following model:
\begin{eqnarray}
&&\hskip-.8in
(u-\alpha^2  u_{zz})_t  -\nu u_{zz} + u\, u_x + w u_z +\ee_1 u - f_0 v + p_x = 0,  \label{EQO1-1}  \\
&&\hskip-.8in
(v-\alpha^2  v_{zz})_t  -\nu v_{zz} + u\, v_x + w v_z +\ee_1 v + f_0 u = 0,  \label{EQO1-2}  \\
&&\hskip-.8in
\ee_2 w + p_z  +  T =0,    \label{EQO1-3}  \\
&&\hskip-.8in
u_x + w_z =0,   \label{EQO1-4} \\
&&\hskip-.8in
T_t - \kappa \Dd T   + u \, T_x + w \, T_z =  0, \label{EQO1-5}
\end{eqnarray}
in the horizontal channel $\{(x,z): 0\leq z\leq H, x\in \mathbb{R}\}$, with boundary conditions (\ref{OBC}).
We will show the global well-posedness of system (\ref{EQO1-1})--(\ref{EQO1-5}) with $\nu=0$. The same results hold for $\nu>0$. Based on this, we show the convergence of the strong solution of system (\ref{EQO1-1})--(\ref{EQO1-5}) to the corresponding solution of system (\ref{EQO-1})--(\ref{EQO-5})  on the interval of existence of the latter, as $\alpha \rightarrow 0$. Furthermore, we derive a criterion for finite-time blow-up of system (\ref{EQO-1})--(\ref{EQO-5}) based on this Voigt $\alpha$-regularization. For more details of Voigt $\alpha$-regularization of the $3D$ Euler equations, we refer the reader to \cite{CLT06,LPTW18,LT10}.

The paper is organized as follows. In section 2, we introduce some notations and collect some preliminary results which will be used in the rest of this paper. In section 3, we show the local well-posedness and global well-posedness with small initial data of system (\ref{EQO-1})--(\ref{EQO-5}). In section 3.3, we will see by assuming $f_0=0$, $v\equiv 0$ and $T\equiv 0$, we are reduced to the 2D hydrostatic Navier-Stokes equations (\ref{EQ-1-4})--(\ref{EQ-3-4}) with damping, for which we can obtain local well-posedness by requiring less on the initial conditions. In section 4, we show the global well-posedness of system (\ref{EQO1-1})--(\ref{EQO1-5}). Again, we provide the similar result by requiring less on the initial conditions for the case when $f_0=0$, $v\equiv 0$ and $T\equiv 0$. In section 5, we show the convergence of the strong solution of system (\ref{EQO1-1})--(\ref{EQO1-5}) to the corresponding solution of system (\ref{EQO-1})--(\ref{EQO-5}) on the interval of existence of the latter, as $\alpha \rightarrow 0$, with $f_0=0$, $v\equiv 0$ and $T\equiv 0$. The reason by assuming $f_0=0$, $v\equiv 0$ and $T\equiv 0$ is for sake of simplicity, and we can get the same convergence result without this assumption. Finally, in section 6, we derive criterion for finite-time blow-up of system (\ref{EQO-1})--(\ref{EQO-5}), with $f_0=0$, $v\equiv 0$ and $T\equiv 0$, based on this Voigt $\alpha$-regularization.


\section{Preliminaries}
In this section, we introduce some notations and collect some preliminary results which will be used in the rest of this paper. For domain $\Om \subset \mathbb{R}^2$, we denote by $L^p(\Om)$, for $p\geq 1$, the Lebesgue space of functions $f$ satisfying $\int_{\Om} |f|^p dxdz < \infty$, and denote the corresponding norm by $\|f \|_{L^p}:=\|f\|_{L^p(\Om)} = (\int_{\Om} |f|^p dxdz)^{\frac{1}{p}}$. For the space $L^2(\Om)$, we denote its inner product by $(\cdot,\cdot)$. Similarly, for $m\geq 1$ an integer, we denote by $H^m(\Om)=W^{m,2}(\Om)$ the Sobolev space of functions $f$ satisfying $\sum\limits_{|\alpha|\leq m}\|D^\alpha f\|^2_{L^2} < \infty$, and denote the corresponding norm by $\|f \|_{H^m}:= (\sum\limits_{|\alpha|\leq m}\|D^\alpha f\|_{L^2}^2)^{\frac{1}{2}}$. Given time $\mathcal{T}>0$, we denote by $L^p(0,\mathcal{T};X)$ the space of functions $f: [0,T]\rightarrow X$ satisfying $\int_0^\mathcal{T} \|f(t)\|_X^p dt < \infty$, where $X$ is a Banach space and $\|\cdot\|_X$ represents its norm. Similarly, we denote by $C([0,\mathcal{T}];X)$ the space of continuous functions $f: [0,T]\rightarrow X$. We write $L^p(0,\mathcal{T};L^2)$ and $L^p(0,\mathcal{T};H^m)$ instead of $L^p(0,\mathcal{T};L^2(\Om))$ and $L^p(0,\mathcal{T};H^m(\Om))$, respectively, for simplicity. When $\Om = \mathbb{T}^2$ is the unit two-dimensional flat torus, we denote by $L^2(\mathbb{T}^2), H^m(\mathbb{T}^2)$ the set of all periodic functions in $x\in \mathbb{R}$ and $z\in \mathbb{R}$ with period 1, which have bounded $L^2(\mathbb{T}^2)$ norm or $H^m(\mathbb{T}^2)$ norm, respectively. We start by recalling the following:

\begin{lemma} (cf. \cite{CW11})
	Assume that $f,g,h,g_z,h_x \in L^2(\mathbb{T}^2)$. Then
	\begin{eqnarray*}
		&&\hskip-.68in \int_{\mathbb{T}^2} |fgh|dxdz \leq C \|f\|_{L^2} \|g\|_{L^2}^{\frac{1}{2}} (\|g\|_{L^2}^{\frac{1}{2}}+ \|g_z\|_{L^2}^{\frac{1}{2}}) \|h\|_{L^2}^{\frac{1}{2}} (\|h\|_{L^2}^{\frac{1}{2}}+\|h_x\|_{L^2}^{\frac{1}{2}}).
	\end{eqnarray*}

\end{lemma}

\begin{proof}
	First, recall that by one-dimensional Agmon's inequality (or Gagliardo--Nirenberg interpolation inequality), for $\phi\in H^1(0,1)$, one has
	\begin{eqnarray}
		&&\hskip-.68in
		\|\phi\|_{L^\infty(0,1)} \leq C\left( \|\phi\|_{L^{2}(0,1)} + \|\phi\|_{L^{2}(0,1)} ^\frac{1}{2} \|\phi_x\|_{L^{2}(0,1)} ^\frac{1}{2} \right). \label{Agmon}
	\end{eqnarray}
Therefore, by H\"older's inequality and Agmon's inequality (\ref{Agmon}),
	\begin{eqnarray}
	&&\hskip-.68in
	\int_{\mathbb{T}^2}|fgh|(x,z)dxdz \leq C \int_0^1 \left[\Big(\int_0^1 |f(x,z)|^2 dx\Big)^{\frac{1}{2}} \Big(\int_0^1 |g(x,z)|^2 dx\Big)^{\frac{1}{2}} \Big(\sup_{0\leq x \leq 1} |h(x,z)|\Big)\right] dz \nonumber \\
	&&\hskip.15in
	\leq C \int_0^1 \Bigg\{\Big(\int_0^1 |f(x,z)|^2 dx\Big)^{\frac{1}{2}} \Big(\int_0^1 |g(x,z)|^2 dx\Big)^{\frac{1}{2}} \nonumber\\
	&&\hskip.85in
	\left[ \Big(\int_0^1 |h(x,z)|^{2} dx\Big)^{\frac{1}{4}}  \Big(\int_0^1 |h_x(x,z)|^{2} dx\Big)^{\frac{1}{4}} + \Big(\int_0^1 |h(x,z)|^{2} dx\Big)^{\frac{1}{2}} \right]\Bigg\} dz \nonumber \\
	&&\hskip.15in
	\leq C \|f\|_{L^2} \sup_{0\leq z \leq 1} \Big(\int_0^1 |g(x,z)|^2 dx\Big)^{\frac{1}{2}}  \|h\|_{L^2}^{\frac{1}{2}} (\|h\|_{L^2}^{\frac{1}{2}}+\|h_x\|_{L^2}^{\frac{1}{2}}). \label{L2-1}
	\end{eqnarray}
	By Minkowski's inequality, Agmon's inequality (\ref{Agmon}), and H\"older inequality,
	\begin{eqnarray}
	&&\hskip-.68in
	\sup_{0\leq z \leq 1} \Big(\int_0^1 |g(x,z)|^2 dx\Big)^{\frac{1}{2}}  \leq C \Big(\int_0^1 \sup_{0\leq z \leq 1}|g(x,z)|^2  dx\Big)^{\frac{1}{2}} \nonumber \\
	&&\hskip 0.3in
	\leq C \left(\int_0^1 \left[\Big(\int_0^1 |g(x,z)|^2 dz\Big)^{\frac{1}{2}}\Big(\int_0^1 |g_z(x,z)|^2 dz\Big)^{\frac{1}{2}} + \int_0^1 |g(x,z)|^2 dz\right]  dx \right)^{\frac{1}{2}} \nonumber \\
	&&\hskip 0.3in
	\leq C \|g\|_{L^2}^{\frac{1}{2}} (\|g\|_{L^2}^{\frac{1}{2}}+ \|g_z\|_{L^2}^{\frac{1}{2}}). \label{L2-2}
	\end{eqnarray}
	Inserting (\ref{L2-2}) to (\ref{L2-1}) yields the desired inequality.
\end{proof}

Next we prove the following:

\begin{lemma}
	Assume that $f\in H^1(\mathbb{T}^2)$ and $f_{xz}\in L^2(\mathbb{T}^2)$. Then $f\in L^\infty(\mathbb{T}^2)$. Moreover, $$\|f\|_{L^\infty} \leq C \left(\|f\|_{H^1}^2 + \|f_{xz}\|_{L^2}^2\right)^{\frac{1}{2}}.$$
\end{lemma}

\begin{proof} Let $\{\hat{f_k}\}_{k\in \mathbb{Z}^2}$ be the Fourier coefficients of $f$. By Cauchy--Schwarz inequality, we have
     	\begin{eqnarray*}
     		&&\hskip-.68in \|f\|_{L^\infty} \leq \sum\limits_{k\in \mathbb{Z}^2} |\hat{f_k}| = \sum\limits_{k\in \mathbb{Z}^2} \frac{|\hat{f_k}|(1+k_1^2+k_2^2+k_1^2 k_2^2)^{\frac{1}{2}}}{(1+k_1^2+k_2^2+k_1^2 k_2^2)^{\frac{1}{2}}}\\
     		&&\hskip-.68in \leq \left(\sum\limits_{k\in \mathbb{Z}^2} |\hat{f_k}|^2(1+k_1^2+k_2^2+k_1^2 k_2^2)\right)^{\frac{1}{2}} \left(\sum\limits_{k\in \mathbb{Z}^2} \frac{1}{(1+k_1^2)(1+k_2^2)}\right)^{\frac{1}{2}}\leq C \Big(\|f\|_{H^1}^2 + \|f_{xz}\|_{L^2}^2\Big)^{\frac{1}{2}} < \infty.
     	\end{eqnarray*}
     	Therefore, $f\in L^\infty(\mathbb{T}^2)$.
\end{proof}

We also need the following Aubin-Lions theorem.

\begin{proposition} (Aubin-Lions Lemma, cf. Simon \cite{SIMON} Corollary 4)
	Assume that X, B and Y are three Banach spaces, with $X\hookrightarrow \hookrightarrow B \hookrightarrow Y$. Then it holds that\\
	
	(i) If $\mathcal{F}$ is a bounded subset of $L^p(0,\mathcal{T};X)$, where $1\leq p < \infty$, and $\mathcal{F}_t := \{\frac{\partial f}{\partial t}| f\in \mathcal{F}\}$ is bounded in $L^1(0,\mathcal{T}; Y)$, then $\mathcal{F}$ is relative compact in $L^p(0,\mathcal{T};B)$.\\
	
	(ii) If $\mathcal{F}$ is a bounded subset of $L^\infty(0,\mathcal{T};X)$ and $\mathcal{F}_t $ is bounded in $L^q(0,\mathcal{T}; Y)$, where $q>1$, then $\mathcal{F}$ is relative compact in $C([0,\mathcal{T}];B)$.
\end{proposition}
\section{Well-posedness of system (\ref{EQO-1})--(\ref{EQO-5})}   \label{S-2}

In this section we study the
system (\ref{EQO-1})--(\ref{EQO-5}) in the horizontal channel $\{(x,z): 0\leq z\leq H, x\in \mathbb{R}\}$. We complement this system with the boundary conditions (\ref{OBC})
and the initial condition
\begin{eqnarray}
&&\hskip-.8in
(u, v, T)|_{t=0}=(u_0, v_0, T_0). \label{OIC}
\end{eqnarray}
In particular, without loss of generality, we choose $H=\frac{1}{2}$. Instead of considering this physical problem, in this section, we consider another problem, that is, system (\ref{EQO-1})--(\ref{EQO-5}) in the unit two dimensional torus $\mathbb{T}^2$, subject to the following symmetric boundary conditions and initial conditions:
\begin{eqnarray}
&&\hskip-.8in
u, \;v, \; w,\; p \; \; \text{and} \; T \; \text{are periodic in} \; x \;  \text{and} \; z \;  \text{with period 1}; \label{OBC-1} \\
&&\hskip-.8in
u, \; v,\; p \; \;\text{are even in} \; z, \;\text{and}\; w, \; T \;\text{are odd in z}; \label{OBC-2} \\
&&\hskip-.8in
(u, v, T)|_{t=0}=(u_0, v_0, T_0).\label{OIC-1}
\end{eqnarray}
The periodicity and symmetry are valid due to the fact that the periodic subspace $\mathcal{H}$, give by

$$\mathcal{H}:= \{(u,v,w,p,T) \; |\; u,\;v,\;w,\;p \;\; \text{and} \; T \; \text{are spatially periodic in both} \; x \;\text{and}  \; z \; \text{variables with period 1},$$
		$$ \text{and are even, even, odd, even, odd with respect to} \; z \; \text{variable, respectively} \},$$
is invariant under the evolution system (\ref{EQO-1})--(\ref{EQO-5}). After solving this problem in the flat torus, the solution restricted on original horizontal channel $\{(x,z): 0\leq z\leq \frac{1}{2}, x\in \mathbb{R}\}$ will solve the original physical problem with corresponding boundary conditions (\ref{OBC}) and initial conditions (\ref{OIC}).
\subsection{Reformulation of The Problem.}
First, let us reformulate the system (\ref{EQO-1})--(\ref{EQO-5}) by deriving equations for $w, p_x$ and $p_z$ in terms of $u,v$ and $T$. For the sake of simplicity, we drop the argument $t$ in functions when there is no confusion. First, from (\ref{EQO-4}) and by boundary condition (\ref{OBC-2}), i.e., $w(x,0)=0$, we have
\begin{eqnarray}
	&&\hskip-.8in
	w(x,z)=-\int_0^z u_x(x,s)ds. \label{w}
\end{eqnarray}
From (\ref{EQO-3}) and (\ref{w}), we have
\begin{eqnarray}
&&\hskip-.38in
p_z(x,z)=-T(x,z)-\ee_2 w(x,z) = -T(x,z) + \ee_2 \int_0^z u_x(x,s)ds. \label{pz}
\end{eqnarray}
Next, we will derive equation for $p_x$. Notice that since $w(x,0)=w(x,1)=0$, from (\ref{w}), one has the compatibility condition
\begin{eqnarray}
&&\hskip-.38in
\int_0^1 u_x(x,z)dz=0. \label{com}
\end{eqnarray}
Let us denote by $c(t):=\int_0^1 u(x,z,t)dz$ and $d(x,t):=\int_0^1 v(x,z,t)dz$.
Integrating (\ref{EQO-1}) with respect to $z$ over $(0,1)$, using boundary condition (\ref{OBC-1}) and (\ref{OBC-2}), one has:
\begin{eqnarray*}
&&\hskip-.58in
\dot{c}(t)+\ee_1 c(t)+\int_0^1 \Big(uu_x(x,z) + wu_z(x,z) + p_x(x,z)\Big)dz
 = f_0 d(x,t).
\end{eqnarray*}
By integration by parts and using (\ref{EQO-4}), (\ref{OBC-1}) and (\ref{OBC-2}), we get
\begin{eqnarray}
	&&\hskip-.58in
\dot{c}(t)+ \ee_1 c(t) + \int_0^1 \Big((u^2)_x(x,z) + p_x(x,z)\Big) dz = f_0 d(x,t). \label{c-1}
\end{eqnarray}
Integrating (\ref{c-1}) with respect to $x$ over $(0,1)$, using compatibility condition (\ref{com}), we have
\begin{eqnarray*}
&&\hskip-.8in
\dot{c}(t) + \ee_1 c(t)+\int_0^1 \int_0^1 \Big((u^2)_x(x,z) + p_x(x,z)\Big)dxdz = f_0 \int_0^1 d(x,t) dx.
\end{eqnarray*}
Thanks to (\ref{OBC-1}), we have
\begin{eqnarray}
	&&\hskip-.8in
	\dot{c}(t) + \ee_1 c(t)= f_0 \int_0^1 d(x,t) dx. \label{c}
\end{eqnarray}
Plugging (\ref{c}) back into (\ref{c-1}) yields
\begin{eqnarray}
&&\hskip-.58in
\int_0^1 p_x(x,z)dz = f_0 \int_0^1  v(x,z) dz - f_0 \int_0^1 \int_0^1 v(x,z)  dxdz  - \int_0^1 2uu_x(x,z) dz. \label{P-1}
\end{eqnarray}
Next, from (\ref{w}) and (\ref{pz}), we have
\begin{eqnarray}
&&\hskip-.58in
p(x,z)=p_s(x)+\ee_2 \int_0^z\int_0^s u_x(x,\xi)d\xi ds-\int_0^zT(x,s)ds,  \label{P-2}
\end{eqnarray}
where $p_s(x) = p(x,0)$ is the pressure at $z=0$. By differentiating (\ref{P-2}) with respect to $x$, and integrating respect to $z$ over $(0,1)$, by virtue of (\ref{P-1}), we have
\begin{eqnarray}
&&\hskip-.8in
(p_s)_{x}(x)=\int_0^1 \Big[\int_0^{z'} T_x(x,s)ds-\ee_2 \int_0^{z'}\int_0^s u_{xx}(x,\xi)d\xi ds + f_0 v(x,z') -2uu_x(x,z')\Big]dz' \nonumber \\
&&\hskip-.28in
- f_0 \int_0^1 \int_0^1  v(x',z') dx'dz'. \label{P-3}
\end{eqnarray}
Therefore, by differentiating (\ref{P-2}) with respect to $x$, and using (\ref{P-3}), we have
\begin{eqnarray}
&&\hskip-.8in
p_{x}(x,z)= \ee_2 \int_0^z\int_0^s u_{xx}(x,\xi)d\xi ds-\int_0^zT_x(x,s)ds \nonumber \\
&&\hskip-.2in
+\int_0^1 \Big[\int_0^{z'} T_x(x,s)ds-\ee_2 \int_0^{z'}\int_0^s u_{xx}(x,\xi)d\xi ds  + f_0 v(x,z') -2uu_x(x,z')\Big]dz' \nonumber \\
&&\hskip-.2in
- f_0 \int_0^1 \int_0^1  v(x',z') dx'dz'. \label{px}
\end{eqnarray}
By virtue of (\ref{w}), (\ref{pz}) and (\ref{px}), and since $p$ is determined up to a constant, the unknowns for system (\ref{EQO-1})--(\ref{EQO-5}) are only $(u,v,T)$. Therefore, we reformulate system (\ref{EQO-1})--(\ref{EQO-5}) to the following system:
\begin{eqnarray}
&&\hskip-.8in
u_t -\nu u_{zz} + u\, u_x + w u_z +\ee_1 u - f_0 v + p_x = 0,  \label{EQ-1}  \\
&&\hskip-.8in
v_t -\nu v_{zz} + u\, v_x + w v_z +\ee_1 v + f_0 u = 0,  \label{EQ-2}  \\
&&\hskip-.8in
T_t - \kappa \Dd T   + u \, T_x + w \, T_z =  0, \label{EQ-5}
\end{eqnarray}
with $w,p_x,p_z$ defined by:
\begin{eqnarray}
&&\hskip-.8in
w(x,z):=-\int_0^z u_x(x,s)ds,  \label{EQw}  \\
&&\hskip-.8in
p_{x}(x,z):= \ee_2 \int_0^z\int_0^s u_{xx}(x,\xi)d\xi ds-\int_0^zT_x(x,s)ds \nonumber \\
&&\hskip-.2in
+\int_0^1 \Big[\int_0^{z'} T_x(x,s)ds-\ee_2 \int_0^{z'}\int_0^s u_{xx}(x,\xi)d\xi ds  + f_0 v(x,z') -2uu_x(x,z')\Big]dz' \nonumber \\
&&\hskip-.2in
- f_0 \int_0^1 \int_0^1  v(x',z') dx'dz', \label{EQpx}  \\
&&\hskip-.8in
p_z(x,z):= -T(x,z) + \ee_2 \int_0^z u_x(x,s)ds. \label{EQpz}
\end{eqnarray}
In this section, we are interested in system (\ref{EQ-1})--(\ref{EQ-5}) with (\ref{EQw})--(\ref{EQpz}) in the unit two dimensional torus $\Om=\mathbb{T}^2$, subject to the following symmetry boundary conditions and initial conditions:
\begin{eqnarray}
&&\hskip-.8in
u, \; v \; \text{and} \; T \;  \text{are periodic in} \;  x \; \text{and} \; z \; \text{with period 1}; \label{BC-1} \\
&&\hskip-.8in
u,\; v  \;  \text{are even in} \; z, \; \text{and} \; T \; \text{is odd in z}; \label{BC-2}  \\
&&\hskip-.8in
(u, v, T)|_{t=0}=(u_0, v_0, T_0).\label{IC-1}
\end{eqnarray}
It's worth mentioning again that our system (\ref{EQ-1})--(\ref{EQ-5}) with (\ref{EQw})--(\ref{EQpz}) satisfies the compatibility condition (\ref{com}).
By virtue of (\ref{EQw})--(\ref{EQpz}) and (\ref{BC-1}), (\ref{BC-2}), we obtain that $w, p$ also satisfy the symmetry conditions:
\begin{eqnarray}
&&\hskip-.8in
w \; \text{and} \;  p \;  \text{are periodic in} \; x \; \text{and} \; z \; \text{with period 1}; \label{wpBC-1} \\
&&\hskip-.8in
p \; \text{is even in} \; z,  \; \text{and} \; w \; \text{is odd in z}. \label{wpBC-2}
\end{eqnarray}
From (\ref{EQw}) and (\ref{EQpz}), and by differentiating (\ref{EQw}) with respect to $z$, we have
\begin{eqnarray}
&&\hskip-.8in
\ee_2 w + p_z  +  T =0,    \label{EQ-3}  \\
&&\hskip-.8in
u_x + w_z =0.   \label{EQ-4}
\end{eqnarray}
Therefore, we conclude system (\ref{EQ-1})--(\ref{EQ-5}) with (\ref{EQw})--(\ref{EQpz}) and subjects to (\ref{BC-1})--(\ref{IC-1}) is equivalent to original system (\ref{EQO-1})--(\ref{EQO-5}) subjects to (\ref{OBC-1})--(\ref{OIC-1}).

\subsection{Local Well-posedness}  \par
In this section, we will show the local regularity of strong solutions to the
system (\ref{EQ-1})--(\ref{EQ-5}) with (\ref{EQw})--(\ref{EQpz}), subjects to boundary and initial conditions (\ref{BC-1})--(\ref{IC-1}). First, we give the definition of strong solution to system (\ref{EQ-1})--(\ref{EQ-5}) with (\ref{EQw})--(\ref{EQpz}).

\begin{definition}
	Suppose that $u_0, v_0, T_0, \partial_x u_{0}, \partial_x  v_{0}, \partial_x  T_{0}\in H^1(\mathbb{T}^2)$ satisfy the symmetry conditions (\ref{BC-1}) and (\ref{BC-2}), with the compatibility condition $\int_0^1 \partial_x  u_{0} dz = 0$. Given time $\mathcal{T}>0$, we say $(u,v,T)$ is a strong solution to system (\ref{EQ-1})--(\ref{EQ-5}) with (\ref{EQw})--(\ref{EQpz}), subjects to (\ref{BC-1})--(\ref{IC-1}), on the time interval $[0,\mathcal{T}]$, if \\
	
	(i) u, v and T satisfy the symmetry conditions (\ref{BC-1}) and (\ref{BC-2});\\
	
	(ii) u, v and T have the regularities
	\begin{eqnarray*}
		&&\hskip-.28in
		u, v, T, u_x, v_x, T_x \in L^{\infty}(0,\mathcal{T};H^1), \;\; u_z,v_z,u_{xz},v_{xz} \in L^2(0,\mathcal{T};H^1), \;\; T,T_x \in L^2(0,\mathcal{T};H^2) , \\
		&&\hskip-.28in
		u,v,T\in L^{\infty}(0,\mathcal{T};L^\infty)\cap C([0,\mathcal{T}];L^2),\;\; \nabla u,\nabla v, \nabla T\in L^{2}(0,\mathcal{T};L^\infty),\;\; \pp_t u, \pp_t v, \pp_t T \in L^2(0,\mathcal{T};L^2);
	\end{eqnarray*}
	
	(iii) u, v and T satisfy system (\ref{EQ-1})--(\ref{EQ-5}) in the following sense:
	\begin{eqnarray*}
		&&\hskip-.68in
		\partial_t u -\nu u_{zz}+uu_x + wu_z+\ee_1 u -f_0 v + p_x = 0  \;\;  \text{in} \; \; L^2(0,\mathcal{T}; L^2);\\
		&&\hskip-.68in
		\partial_t v -\nu v_{zz}+uv_x + wv_z+\ee_1 v +f_0 u = 0  \;\;  \text{in} \; \; L^2(0,\mathcal{T}; L^2);\\
		&&\hskip-.68in
		\partial_t T -\kappa \Dd T +uT_x + wT_z= 0  \;\;  \text{in} \; \; L^2(0,\mathcal{T}; L^2),
	\end{eqnarray*}
	with $w,p_x,p_z$ defined by (\ref{EQw})--(\ref{EQpz}), and fulfill the initial condition (\ref{IC-1}).
\end{definition}
We have the following result concerning the existence and uniqueness of strong solutions to system (\ref{EQ-1})--(\ref{EQ-5}) with (\ref{EQw})--(\ref{EQpz}), subjects to (\ref{BC-1})--(\ref{IC-1}), on $\mathbb{T}^2 \times (0,\mathcal{T})$, for some positive time $\mathcal{T}$.

\begin{theorem} \label{T-MAIN}
	Suppose that $u_0, v_0, T_0, \partial_x  u_{0}, \partial_x  v_{0}, \partial_x  T_{0}\in H^1(\mathbb{T}^2)$ satisfy the symmetry conditions (\ref{BC-1}) and (\ref{BC-2}), with the compatibility condition $\int_0^1 \partial_x  u_{0} dz = 0$. Then there exists some time $\mathcal{T}>0$ such that there exists a unique strong solution $(u, v, T)$ of system
	(\ref{EQ-1})--(\ref{EQ-5}) with (\ref{EQw})--(\ref{EQpz}), subjects to (\ref{BC-1})--(\ref{IC-1}), on the interval $[0,\mathcal{T}]$. Moreover, the unique strong solution $(u, v, T)$ depends continuously on the initial data.
\end{theorem}
In section 3.2.1, we establish the existence of solutions to system (\ref{EQ-1})--(\ref{EQ-5}) with (\ref{EQw})--(\ref{EQpz}) by employing the standard Galerkin approximation procedure. In section 3.2.2, we establish formal \textit{a priori} estimates for the solutions of system (\ref{EQ-1})--(\ref{EQ-5}) with (\ref{EQw})--(\ref{EQpz}). These estimates can be justified rigorously by deriving them first to the Galerkin approximation system and then passing to the limit using the Aubin-Lions compactness theorem. In section 3.2.3, we establish the uniqueness of strong solutions, and its continuous dependence on the initial data.

\subsubsection{Galerkin approximating system.}
In this section, we employ the standard Galerkin approximation procedure. Let
\begin{equation}
\phi_k = \phi_{k_1,k_2} :=
\begin{cases}
\sqrt{2}\exp\left(2\pi ik_1 x \right)\cos(2\pi k_2 z) & \text{if} \;  k_2\neq 0\\
\exp\left(2\pi ik_1 x \right) & \text{if} \;  k_2=0,  \label{phik}
\end{cases}
\end{equation}
\begin{equation}
\psi_k = \psi_{k_1,k_2} :=\sqrt{2} \exp\left(2\pi ik_1 x \right)\sin(2\pi k_2 z), \label{psik}
\end{equation}
and
\begin{eqnarray*}
	&&\hskip-.28in
	\mathcal{E}:=  \{ \phi \in L^2(\mathbb{T}^2) \; | \; \phi= \sum\limits_{k\in \mathbb{Z}^2} a_k \phi_k, \; a_{-k_1, k_2}=a_{k_1,k_2}^{*}, \; \sum\limits_{k\in \mathbb{Z}^2} |a_k|^2 < \infty \}, \\
	&&\hskip-.28in
	\mathcal{O}:= \{ \psi \in L^2(\mathbb{T}^2) \; | \; \psi= \sum\limits_{k\in \mathbb{Z}^2} a_k \psi_k, \; a_{-k_1, k_2}=a_{k_1,k_2}^{*}, \; \sum\limits_{k\in \mathbb{Z}^2} |a_k|^2 < \infty  \}.
\end{eqnarray*}
Observe that functions in $\mathcal{E}$ and $\mathcal{O}$ are even and odd with respect to $z$ variable, respectively. Moreover, $\mathcal{E}$ and $\mathcal{O}$ are closed subspace of $L^2(\mathbb{T}^2)$, orthogonal to each other and consist of real valued functions. For any $m\in \N$, denote by
\begin{eqnarray*}
	&&\hskip-.28in
	\mathcal{E}_m:=  \{ \phi \in L^2(\mathbb{T}^2)\; | \; \phi= \sum\limits_{|k|\leq m} a_k \phi_k, \; a_{-k_1, k_2}=a_{k_1,k_2}^{*}  \}, \\
	&&\hskip-.28in
	\mathcal{O}_m:= \{ \psi \in L^2(\mathbb{T}^2)\; | \; \psi= \sum\limits_{|k|\leq m} a_k \psi_k, \; a_{-k_1, k_2}=a_{k_1,k_2}^{*}  \},
\end{eqnarray*}
the finite-dimensional subspaces of $\mathcal{E}$ and $\mathcal{O}$, respectively. For any function $f\in L^2(\mathbb{T}^2)$, we denote by $\hat{f_j}=(f,\bar{\phi_j})$ and $\tilde{f_j}=(f,\bar{\psi_j})$, and we write $P_m f := \sum_{|k|\leq m} \hat{f_k} \phi_k$ and $\Pi_m f := \sum_{|k|\leq m} \tilde{f_k} \psi_k$. Then $P_m$ and $\Pi_m$ are the orthogonal projections from $L^2(\mathbb{T}^2)$ to $\mathcal{E}_m$ and $\mathcal{O}_m$, respectively. Now let
\begin{eqnarray*}
	&&\hskip-.68in
	u_m=\sum_{|k|=0}^{m}a_k(t)\phi_k(x,z), \;\; v_m=\sum_{|k|=0}^{m}b_k(t)\phi_k(x,z), \;\; T_m=\sum_{|k|=0}^{m}c_k(t)\psi_k(x,z),
\end{eqnarray*}
and consider the following Galerkin approximation system for our model (\ref{EQ-1})--(\ref{EQ-5}), with (\ref{EQw})--(\ref{EQpz}):
\begin{eqnarray}
&&\hskip-.8in
\partial_t u_m -\nu \partial^2_{zz} u_m+ P_m[u_m \partial_x u_m+w_m \partial_z u_m]+\ee_1 u_m-f_0 v_m+\partial_{x} p_m=0, \label{EQ-13}  \\
&&\hskip-.8in
\partial_t v_m -\nu \partial^2_{zz} v_m + P_m[u_m \partial_{x} v_m+w_m \partial_{z} v_m]+\ee_1 v_m + f_0 u_m=0, \label{EQ-14} \\
&&\hskip-.8in
\partial_tT_m-\kappa \Dd T_m+\Pi_m[u_m \partial_{x} T_m + w_m \partial_{z} T_m]=0,  \label{EQ-15}
\end{eqnarray}
with $w_m, \partial_{x} p_m, \partial_{z} p_m$ defined by:
\begin{eqnarray}
&&\hskip-.5in
w_m(x,z):=-\int_0^z \partial_x u_m(x,s)ds,  \label{EQwm}  \\
&&\hskip-.5in
\partial_{x} p_m(x,z):= \ee_2 \int_0^z\int_0^s \partial^2_{xx} u_m(x,\xi)d\xi ds-\int_0^z\partial_{x} T_m(x,s)ds \nonumber \\
&&\hskip-.02in
+\int_0^1 \Big[\int_0^{z'} \partial_{x} T_m(x,s)ds-\ee_2 \int_0^{z'}\int_0^s \partial^2_{xx} u_m(x,\xi)d\xi ds  + f_0 v_m(x,z') \Big]dz' \nonumber \\
&&\hskip-.02in
-P_m \int_0^1 2u_m\partial_x u_m(x,z')dz' - f_0 \int_0^1 \int_0^1  v_m(x',z') dx'dz', \label{EQpmx}  \\
&&\hskip-.5in
\partial_{z} p_m(x,z):= -T_m(x,z) + \ee_2 \int_0^z \partial_x u_m(x,s)ds, \label{EQpmz}
\end{eqnarray}
subjects to the following initial conditions:
\begin{eqnarray}
u_{m}(0)=P_m u_0, \;\; v_m(0) = P_m v_0, \;\;  T_m(0)= \Pi_m T_0.\label{EQ-18}
\end{eqnarray}
Observe that the definitions of $w_m, \partial_{x} p_m$ and $\partial_{z} p_m$ are inspired by (\ref{EQw})--(\ref{EQpz}). Moreover, notice that $(\partial_{x} p_m)_z(x,z)= -\partial_{x} T_m(x,z) + \ee_2 \int_0^z \partial^2_{xx} u_m(x,s)ds = (\partial_{z} p_m)_x (x,z)$, and hence (\ref{EQpmx}) and (\ref{EQpmz}) are compatible. For each $m\geq 1$, the Galerkin approximation, system (\ref{EQ-13})--(\ref{EQ-15}), together with (\ref{EQwm})--(\ref{EQpmz}) corresponds to a first order system of ordinary differential equations, in the coefficients $a_k, b_k$ and $c_k$ for $0\leq |k| \leq m$, with quadratic nonlinearity. Therefore, by the theory of ordinary differential equations, there exists some $t_m > 0$ such that system (\ref{EQ-13})--(\ref{EQ-15}) together with (\ref{EQwm})--(\ref{EQpmz}) admit a unique solution $(u_m, v_m, T_m)$ on the interval $[0, t_m]$. Observe that from (\ref{EQ-18}), we have $a_k(0), b_k(0), c_k(0)\in \mathbb{C}$ satisfying $a_{-k_1, k_2}(0)=a_{k_1,k_2}^{*}(0), b_{-k_1, k_2}(0)=b_{k_1,k_2}^{*}(0)$, and $c_{-k_1, k_2}(0)=c_{k_1,k_2}^{*}(0)$. Thanks to the uniqueness of the solutions of the ODE system, we conclude that $a_{-k_1, k_2}(t)=a_{k_1,k_2}^{*}(t), b_{-k_1, k_2}(t)=b_{k_1,k_2}^{*}(t)$, and $c_{-k_1, k_2}(t)=c_{k_1,k_2}^{*}(t)$, for $t\in[0, t_m]$. Therefore, $u_m, v_m \in \mathcal{E}_m$, and $T_m \in \mathcal{O}_m$. In the next section, we perform formal \textit{a priori} estimates for the original system (\ref{EQ-1})--(\ref{EQ-5}) with (\ref{EQw})--(\ref{EQpz}). These formal \textit{a priori} estimates can be justified rigorously by establishing them first to the Galerkin approximation system and then passing to the limit using the Aubin-Lions compactness theorem.

\subsubsection{A priori estimates}
The constant $C$ appears in the following inequalities may change from line to line. By taking the $L^2$-inner product of equation  (\ref{EQ-1}) with $u, -\Dd u, \Dd u_{xx}$, equation  (\ref{EQ-2}) with $v, -\Dd v, \Dd v_{xx}$,
equation  (\ref{EQ-3}) with $w, -\Dd w, \Dd w_{xx}$ and
equation  (\ref{EQ-5}) with $T, -\Dd T, \Dd T_{xx}$, and by integration by parts, thanks to (\ref{BC-1}) and (\ref{wpBC-1}), we have
\begin{eqnarray}
	&&\hskip-.28in \frac{1}{2} \frac{d}{dt} \Big(\|u\|_{L^2}^{2} + \|\nabla u\|_{L^2}^{2} +  \|v\|_{L^2}^{2} + \|\nabla v\|_{L^2}^{2} + \|\nabla u_{x}\|_{L^2}^{2} + \|\nabla v_{x}\|_{L^2}^{2} +\|T\|_{L^2}^{2} + \|\nabla T\|_{L^2}^{2}  + \|\nabla T_x\|_{L^2}^{2}\Big)\nonumber \\
	&&\hskip-.28in
	+ \nu \Big(\|u_z\|_{L^2}^{2} + \|v_z\|_{L^2}^{2} + \|\nabla u_z\|_{L^2}^{2} + \|\nabla v_z\|_{L^2}^{2}+ \|\nabla u_{xz}\|_{L^2}^{2} + \|\nabla v_{xz}\|_{L^2}^{2}\Big)
	\nonumber \\
     &&\hskip-.28in + \ee_1 \Big(\|u\|_{L^2}^{2} + \|\nabla u\|_{L^2}^{2} +  \|v\|_{L^2}^{2} + \|\nabla v\|_{L^2}^{2} + \|\nabla u_{x}\|_{L^2}^{2} + \|\nabla v_{x}\|_{L^2}^{2}\Big)
      \nonumber \\
     &&\hskip-.28in
     + \ee_2 \Big(\|w\|_{L^2}^{2} + \|\nabla w\|_{L^2}^{2}+ \|\nabla w_{x}\|_{L^2}^{2}\Big) + \kappa \Big( \|\nabla T\|_{L^2}^{2} + \|\Dd T\|_{L^2}^{2} + \|\Dd T_x\|_{L^2}^{2} \Big) \nonumber \\
     &&\hskip-.28in
       = \int_{\mathbb{T}^2} \left[ u u_{x}+w u_{z} - f_0 v + p_x
	\right] \left(-u+ \Dd u - \Dd u_{xx}\right) + \left[ u v_{x}+w v_{z} + f_0 u
	\right] \left(-v+ \Dd v - \Dd v_{xx}\right)\nonumber \\
	 &&\hskip-.08in
	 + \left(  p_{z}+ T\right) \left(-w + \Dd w - \Dd w_{xx}\right) + \left(uT_x+wT_z\right)(-T+ \Dd T -\Dd T_{xx})\; dxdz.\label{local}
\end{eqnarray}
By integration by parts, thanks to (\ref{BC-1}), (\ref{wpBC-1}) and (\ref{EQ-4}), we have
\begin{eqnarray*}
	&&\hskip-.45in  \int_{\mathbb{T}^2} \left(-f_0v + p_x\right) \left(-u+ \Dd u - \Dd u_{xx}\right) + f_0 u \left(-v + \Dd v - \Dd v_{xx}\right) + p_z\left(-w + \Dd w - \Dd w_{xx}\right) \\
	&&\hskip-.45in
	+ \left(uu_x+wu_z\right)(-u +u_{zz} ) + (uv_x + wv_z)(-v) + (uT_x + wT_z)(-T)\; dxdz =0.
\end{eqnarray*}
Therefore, the right-hand side of (\ref{local}) becomes
\begin{eqnarray*}
&&\hskip-.28in
 \int_{\mathbb{T}^2} \left( u u_{x}+w u_{z}
\right)\left(u_{xx} - u_{xxxx}- u_{xxzz}\right) + \left( u v_{x}+w v_{z}
\right) \left(\Dd v - v_{xxxx} - v_{xxzz}\right)
 \\
&&\hskip-.28in + T \left(-w + \Dd w - \Dd w_{xx} \right) + \left(uT_x+wT_z\right)(\Dd T - \Dd T_{xx})\; dxdz
=: I_1+I_2+I_3+I_4 +I_5+I_6+I_7+I_8.
\end{eqnarray*}
Let us denote by
\begin{eqnarray}
&&\hskip-.28in
 Y:=1+\|u\|_{L^2}^{2} + \|\nabla u\|_{L^2}^{2} +  \|v\|_{L^2}^{2} + \|\nabla v\|_{L^2}^{2} + \|\nabla u_{x}\|_{L^2}^{2} + \|\nabla v_{x}\|_{L^2}^{2} +\|T\|_{L^2}^{2} + \|\nabla T\|_{L^2}^{2} +  \|\nabla T_x\|_{L^2}^{2}, \nonumber \\
&&\hskip-.28in
 F:=\|u_z\|_{L^2}^{2} + \|v_z\|_{L^2}^{2} + \|\nabla u_z\|_{L^2}^{2} + \|\nabla v_z\|_{L^2}^{2}+ \|\nabla u_{xz}\|_{L^2}^{2} + \|\nabla v_{xz}\|_{L^2}^{2}, \nonumber \\
 &&\hskip-.28in
G:= \|w\|_{L^2}^{2} + \|\nabla w\|_{L^2}^{2}+ \|\nabla w_{x}\|_{L^2}^{2}, \qquad  \qquad K :=\|\nabla T\|_{L^2}^{2} + \|\Dd T\|_{L^2}^{2} + \|\Dd T_x\|_{L^2}^{2}. \label{yfgh}
\end{eqnarray}
From (\ref{w}), by H\"older inequality and Minkowski inequality, we have
\begin{eqnarray}
&&\hskip-.68in \|w\|_{L^2}= \left(\int_0^1 \int_0^1 |\int_0^z u_x(x,s) ds|^2 dxdz\right)^{\frac{1}{2}}\leq \int_0^z \left(\int_0^1 \int_0^1 |u_x(x,s)|^2 dxdz  \right)^{\frac{1}{2}} ds \nonumber\\
&&\hskip-.68in \leq \int_0^1 \left(\int_0^1 \int_0^1 |u_x(x,s)|^2 dxdz  \right)^{\frac{1}{2}} ds \leq \left(\int_0^1 \int_0^1 \int_0^1 |u_x(x,s)|^2 dxdz   ds\right)^{\frac{1}{2}} = \|u_x\|_{L^2}. \label{ESw}
\end{eqnarray}
Similarly, one can get
\begin{eqnarray}
&&\hskip-.68in \|w_x\|_{L^2} \leq \|u_{xx}\|_{L^2}. \label{ESwx}
\end{eqnarray}
Let us estimate terms $I_1$--$I_8$. By integration by parts, using Cauchy--Schwarz inequality, Young's inequality and Lemma 1, thanks to (\ref{BC-1}), (\ref{wpBC-1}), (\ref{EQ-4}), (\ref{ESw}) and (\ref{ESwx}), we have
\begin{eqnarray*}
	&&\hskip-.45in  |I_1| = \left|\int_{\mathbb{T}^2} \left( u u_{x}+w u_{z}
	\right) u_{xx} dxdz \right|
	\leq C \|u\|_{L^2}^{\frac{1}{2}} (\|u\|_{L^2}^{\frac{1}{2}} + \|u_z\|_{L^2}^{\frac{1}{2}}) \|u_x\|_{L^2}^{\frac{1}{2}} (\|u_x\|_{L^2}^{\frac{1}{2}} + \|u_{xx}\|_{L^2}^{\frac{1}{2}})\|u_{xx}\|_{L^2} \\
	&&\hskip .35in + C \|u_z\|_{L^2}^{\frac{1}{2}} (\|u_z\|_{L^2}^{\frac{1}{2}} + \|u_{xz}\|_{L^2}^{\frac{1}{2}}) \|w\|_{L^2}^{\frac{1}{2}} (\|w\|_{L^2}^{\frac{1}{2}} + \|w_{z}\|_{L^2}^{\frac{1}{2}}) \|u_{xx}\|_{L^2}\leq  CY^{\frac{3}{2}} \leq CY^{3}, \\
	&&\hskip-.45in  |I_2| = \left|\int_{\mathbb{T}^2} \left[ u u_{x}+w u_{z}
	\right]\; u_{xxxx} \; dxdz\right|=
	\left|  \int_{\mathbb{T}^2} \left[ 3u_{x} u_{xx}+w_{xx} u_{z} + 2w_{x} u_{xz}
	\right]\; u_{xx} \; dxdz \right|  \\
	&&\hskip-.25in
	\leq    C \Big[  \|u_x\|_{L^2}^{\frac{1}{2}} (\|u_x\|_{L^2}^{\frac{1}{2}} + \|u_{xx}\|_{L^2}^{\frac{1}{2}}) \|u_{xx}\|_{L^2}^{\frac{3}{2}} (\|u_{xx}\|_{L^2}^{\frac{1}{2}} + \|u_{xxz}\|_{L^2}^{\frac{1}{2}})\\
	&&\hskip-.05in
	+ \|w_{xx} \|_{L^2}   \|u_{z}\|_{L^2}^{\frac{1}{2}} ( \|u_{z}\|_{L^2}^{\frac{1}{2}}+ \|u_{xz}\|_{L^2}^{\frac{1}{2}}) \|u_{xx}\|_{L^2}^{\frac{1}{2}} (\|u_{xx}\|_{L^2}^{\frac{1}{2}}+ \|u_{xxz}\|_{L^2}^{\frac{1}{2}})  \\
	&&\hskip-.05in
	+  \|w_{x} \|_{L^2}  \|u_{xz}\|_{L^2}^{\frac{1}{2}}(\|u_{xz}\|_{L^2}^{\frac{1}{2}} + \|u_{xxz}\|_{L^2}^{\frac{1}{2}} ) \|u_{xx}\|_{L^2}^{\frac{1}{2}} (\|u_{xx}\|_{L^2}^{\frac{1}{2}} +  \|u_{xxz}\|_{L^2}^{\frac{1}{2}}) \Big] \\
	&&\hskip-.25in
	\leq \frac{\ee_2}{6} \|w_{xx}\|_{L^2}^2 + \frac{\nu}{6}\|u_{xxz}\|_{L^2}^2 + CY^3 \leq \frac{\nu}{10}F + \frac{\ee_2}{6}G + CY^3,\\
	&&\hskip-.45in  |I_3| = \left|\int_{\mathbb{T}^2} \left[ u u_{x}+w u_{z}
	\right]\; u_{xxzz} \; dxdz\right|=
	\left|  \int_{\mathbb{T}^2} \left( u_xu_{xz}+w_x u_{zz}
	\right)\; u_{xz} \; dxdz \right|  \\
	&&\hskip-.25in
	\leq    C \Big[  \|u_x\|_{L^2}^{\frac{1}{2}} (\|u_x\|_{L^2}^{\frac{1}{2}} + \|u_{xz}\|_{L^2}^{\frac{1}{2}}) \|u_{xz}\|_{L^2}^{\frac{3}{2}} (\|u_{xz}\|_{L^2}^{\frac{1}{2}} + \|u_{xxz}\|_{L^2}^{\frac{1}{2}})\\
	&&\hskip-.05in
	+ \|u_{xz} \|_{L^2}   \|w_{x}\|_{L^2}^{\frac{1}{2}} ( \|w_{x}\|_{L^2}^{\frac{1}{2}}+ \|w_{xz}\|_{L^2}^{\frac{1}{2}}) \|u_{zz}\|_{L^2}^{\frac{1}{2}} (\|u_{zz}\|_{L^2}^{\frac{1}{2}}+ \|u_{xzz}\|_{L^2}^{\frac{1}{2}})  \\
	&&\hskip-.25in
	\leq \frac{\nu}{10}(\|u_{zz}\|_{L^2}^2 + \|u_{xxz}\|_{L^2}^2 + \|u_{xzz}\|_{L^2}^2 )+ CY^2 \leq \frac{\nu}{10}F + CY^3, \\
	&&\hskip-.45in  |I_4| = \left|\int_{\mathbb{T}^2} \left[ u v_{x}+w v_{z}
	\right]\; \Dd v \; dxdz\right|	\leq    C (\|u_{}\|_{L^2}+\|u_{z}\|_{L^2})(\|v_{x}\|_{L^2}+\|v_{xx}\|_{L^2})(\|v_{xx}\|_{L^2}+ \|v_{zz}\|_{L^2}) \\
	&&\hskip-.25in
    + C(\|w_{}\|_{L^2}+\|w_{z}\|_{L^2})(\|v_{z}\|_{L^2} + \|v_{xz}\|_{L^2})(\|v_{xx}\|_{L^2}+ \|v_{zz}\|_{L^2})
    \leq \frac{\nu}{10}\| v_{zz}\|_{L^2}^2  +CY^2 \leq \frac{\nu}{10}F  + CY^3, \\
	&&\hskip-.45in  |I_5| = |\int_{\mathbb{T}^2} \left[ u v_{x}+w v_{z}
	\right]\; v_{xxxx} dxdz | =
	|\int_{\mathbb{T}^2} \left[ u_{xx} v_{x} + w_{xx} v_{z} +2 u_x v_{xx} + 2w_{x} v_{xz} \right]\; v_{xx}  dxdz| \\
	&&\hskip-.25in
	\leq C \Big[ \|v_{xx}\|_{L^2}\|v_{x}\|_{L^2}^{\frac{1}{2}} (\|v_{x}\|_{L^2}^{\frac{1}{2}} + \|v_{xx}\|_{L^2}^{\frac{1}{2}}) \|u_{xx}\|^{\frac{1}{2}}_{L^2} (\|u_{xx}\|^{\frac{1}{2}}_{L^2} + \|u_{xxz}\|_{L^2}^{\frac{1}{2}})\\
	&&\hskip-.05in
	+ \|w_{xx}\|_{L^2}\|v_{z}\|_{L^2}^{\frac{1}{2}} (\|v_{z}\|_{L^2}^{\frac{1}{2}} + \|v_{xz}\|_{L^2}^{\frac{1}{2}}) \|v_{xx}\|_{L^2}^{\frac{1}{2}} ( \|v_{xx}\|_{L^2}^{\frac{1}{2}} + \|v_{xxz}\|_{L^2}^{\frac{1}{2}} )  \\
	&&\hskip-.05in
	+\|v_{xx}\|_{L^2}^{\frac{3}{2}}(\|v_{xx}\|_{L^2}^{\frac{1}{2}} + \|v_{xxz}\|_{L^2}^{\frac{1}{2}} )\|u_{x}\|_{L^2}^{\frac{1}{2}} ( \|u_{x}\|_{L^2}^{\frac{1}{2}} + \|u_{xx}\|_{L^2}^{\frac{1}{2}}) \\
	&&\hskip-.05in
	+ \|w_x\|_{L^2}^{\frac{1}{2}} (\|w_x\|_{L^2}^{\frac{1}{2}} + \|w_{xz}\|_{L^2}^{\frac{1}{2}})  \|v_{xz}\|_{L^2}^{\frac{1}{2}} (\|v_{xz}\|_{L^2}^{\frac{1}{2}} + \|v_{xxz}\|_{L^2}^{\frac{1}{2}}) \|v_{xx}\|_{L^2} \Big] \\
	&&\hskip-.25in
	\leq \frac{\ee_2}{6} \|w_{xx}\|_{L^2}^2 + \frac{\nu}{10}(\|u_{xxz}\|_{L^2}^2+\|v_{xxz}\|_{L^2}^2)+ CY^3\leq \frac{\nu}{10}F + \frac{\ee_2}{6}G + CY^3,\\
	&&\hskip-.45in  |I_6| = |\int_{\mathbb{T}^2} \left[ u v_{x}+w v_{z}
	\right]\; v_{xxzz} dxdz | =
	|\int_{\mathbb{T}^2} \left[ u_{xz} v_{x} + v_{xx} u_{z} - v_z u_{xx} + w_{x} v_{xz} \right]\; v_{xz}  dxdz| \\
	&&\hskip-.25in
	\leq C \Big[ \|v_{xz}\|_{L^2}\|v_{x}\|_{L^2}^{\frac{1}{2}} (\|v_{x}\|_{L^2}^{\frac{1}{2}} + \|v_{xz}\|_{L^2}^{\frac{1}{2}}) \|u_{xz}\|^{\frac{1}{2}}_{L^2} (\|u_{xz}\|^{\frac{1}{2}}_{L^2} + \|u_{xxz}\|_{L^2}^{\frac{1}{2}})\\
	&&\hskip-.05in
	+ \|v_{xz}\|_{L^2}\|u_{z}\|_{L^2}^{\frac{1}{2}} (\|u_{z}\|_{L^2}^{\frac{1}{2}} + \|u_{xz}\|_{L^2}^{\frac{1}{2}}) \|v_{xx}\|_{L^2}^{\frac{1}{2}} ( \|v_{xx}\|_{L^2}^{\frac{1}{2}} + \|v_{xxz}\|_{L^2}^{\frac{1}{2}} )  \\
	&&\hskip-.05in
	+\|v_{xz}\|_{L^2}\|v_{z}\|_{L^2}^{\frac{1}{2}} ( \|v_{z}\|_{L^2}^{\frac{1}{2}} + \|v_{xz}\|_{L^2}^{\frac{1}{2}})\|u_{xx}\|_{L^2}^{\frac{1}{2}}(\|u_{xx}\|_{L^2}^{\frac{1}{2}} + \|u_{xxz}\|_{L^2}^{\frac{1}{2}} ) \\
	&&\hskip-.05in
	+ \|v_{xz}\|_{L^2}\|w_x\|_{L^2}^{\frac{1}{2}} (\|w_x\|_{L^2}^{\frac{1}{2}} + \|w_{xz}\|_{L^2}^{\frac{1}{2}})  \|v_{xz}\|_{L^2}^{\frac{1}{2}} (\|v_{xz}\|_{L^2}^{\frac{1}{2}} + \|v_{xxz}\|_{L^2}^{\frac{1}{2}})  \Big] \\
	&&\hskip-.25in
	\leq  \frac{\nu}{10}(\|u_{xxz}\|_{L^2}^2+\|v_{xxz}\|_{L^2}^2)+ CY^2\leq \frac{\nu}{10}F  + CY^3,
\end{eqnarray*}
\begin{eqnarray*}
	&&\hskip-.45in  |I_7| = |\int_{\mathbb{T}^2} T \left(-w + \Dd w - \Dd w_{xx} \right) dxdz |
	 \\
	&&\hskip-.25in
	\leq \|T\|_{L^2}\|w\|_{L^2}+ \|\nabla T\|_{L^2} \|\nabla w\|_{L^2} +\|\nabla T_x\|_{L^2} \|\nabla w_x\|_{L^2} \leq \frac{\ee_2}{6}G +CY,\\
	&&\hskip-.45in  |I_8| = |\int_{\mathbb{T}^2} \left[ u T_{x}+w T_{z}
	\right]\; (\Dd T - \Dd T_{xx}) dxdz | \\
	&&\hskip-.25in
	\leq |\int_{\mathbb{T}^2} \left[ u T_{x}+w T_{z}
	\right]\; \Dd T  dxdz | + |\int_{\mathbb{T}^2} \left[ u_x T_{x} + uT_{xx} + wT_{xz}+w_x T_{z}
	\right]\;  \Dd T_{x} dxdz |\\
	&&\hskip-.25in
	\leq C \Big[ \|u\|_{L^2}^{\frac{1}{2}} (\|u\|_{L^2}^{\frac{1}{2}} + \|u_{x}\|_{L^2}^{\frac{1}{2}}) \|T_{x}\|^{\frac{1}{2}}_{L^2} (\|T_{x}\|^{\frac{1}{2}}_{L^2} + \|T_{xz}\|_{L^2}^{\frac{1}{2}}) \\
	&&\hskip.15in
	+ \|w\|_{L^2}^{\frac{1}{2}} (\|w\|_{L^2}^{\frac{1}{2}} + \|w_{z}\|_{L^2}^{\frac{1}{2}}) \|T_{z}\|^{\frac{1}{2}}_{L^2} (\|T_{z}\|^{\frac{1}{2}}_{L^2} + \|T_{xz}\|_{L^2}^{\frac{1}{2}}) \Big] \|\Dd T\|_{L^2}\\
	&&\hskip-.15in
	+  C\Big[ \|u_x\|_{L^2}^{\frac{1}{2}} (\|u_x\|_{L^2}^{\frac{1}{2}} + \|u_{xx}\|_{L^2}^{\frac{1}{2}}) \|T_{x}\|^{\frac{1}{2}}_{L^2} (\|T_{x}\|^{\frac{1}{2}}_{L^2} + \|T_{xz}\|_{L^2}^{\frac{1}{2}}) \\
	&&\hskip.15in
	+ \|u\|_{L^2}^{\frac{1}{2}} (\|u\|_{L^2}^{\frac{1}{2}} + \|u_{z}\|_{L^2}^{\frac{1}{2}}) \|T_{xx}\|^{\frac{1}{2}}_{L^2} (\|T_{xx}\|^{\frac{1}{2}}_{L^2} + \|T_{xxx}\|_{L^2}^{\frac{1}{2}}) \\
	&&\hskip.15in
	+ \|w\|_{L^2}^{\frac{1}{2}} (\|w\|_{L^2}^{\frac{1}{2}} + \|w_{x}\|_{L^2}^{\frac{1}{2}}) \|T_{xz}\|^{\frac{1}{2}}_{L^2} (\|T_{xz}\|^{\frac{1}{2}}_{L^2} + \|T_{xzz}\|_{L^2}^{\frac{1}{2}})\\
	&&\hskip.15in
	+ \|w_x\|_{L^2}^{\frac{1}{2}} (\|w_x\|_{L^2}^{\frac{1}{2}} + \|w_{xz}\|_{L^2}^{\frac{1}{2}}) \|T_{z}\|^{\frac{1}{2}}_{L^2} (\|T_{z}\|^{\frac{1}{2}}_{L^2} + \|T_{xz}\|_{L^2}^{\frac{1}{2}}) \Big] \|\Dd T_x\|_{L^2} \\
	&&\hskip-.25in
	\leq  \frac{\kappa}{2}(\|\Dd T\|_{L^2}^2+\|\Dd T_x\|_{L^2}^2)+ CY^3\leq \frac{\kappa}{2} K + CY^3.
\end{eqnarray*}
From the estimates above, (\ref{local}) becomes
\begin{eqnarray}
&&\hskip-.28in  \frac{dY}{dt}+\nu F+\ee_2 G + \kappa K\leq CY^3. \label{local2}
\end{eqnarray}
Therefore, we have $\frac{dY}{dt} \leq C Y^3$, and this implies that
$$Y(t)\leq \sqrt{\frac{Y(0)^2}{1-Y(0)^2Ct}}.$$
Choose $$\mathcal{T}=\frac{3}{4CY(0)^2}.$$
From above, we have $Y(t)\leq 2Y(0)$ on $[0,\mathcal{T}]$.
Plugging it in (\ref{local2}), we have
\begin{eqnarray}
	&&\hskip-.68in \frac{dY}{dt} + \nu F + \ee_2 G +\kappa K\leq 8C Y(0)^3,  \;\text{for}\; t\in[0,\mathcal{T}]. \label{localest}
\end{eqnarray}
Integrating above from 0 to $t$ for any time $t\in[0,\mathcal{T}]$, we obtain
\begin{eqnarray*}
	&&\hskip-.68in Y(t) + \int_0^t \Big(\nu F(s) + \ee_2 G(s) +\kappa K(s)\Big) \; ds \leq Y(0) + 8Ct Y(0)^3 .
\end{eqnarray*}
Therefore, we have
\begin{eqnarray}
u, v, T, u_x, v_x, T_x \in L^{\infty}(0,\mathcal{T};H^1), \;\; u_z,v_z,u_{xz},v_{xz} \in L^2(0,\mathcal{T};H^1), \;\; T,T_x \in L^2(0,\mathcal{T};H^2). 	\label{L-2x}
\end{eqnarray}
By virtue of (\ref{L-2x}) and (\ref{ESw}), we have
\begin{eqnarray}
w\in L^{\infty}(0,\mathcal{T};H^1). 	\label{L-2x2}
\end{eqnarray}
Thanks to Lemma 2, from (\ref{L-2x}), we also have
\begin{eqnarray}
u,v,T\in L^{\infty}(0,\mathcal{T};L^\infty),\;\; \nabla u,\nabla v, \nabla T\in L^{2}(0,\mathcal{T};L^\infty). 	\label{L-2x3}
\end{eqnarray}


\subsubsection{Uniqueness of solutions and its continuous dependence on the initial data}
In this section, we will show the continuous dependence on the initial data and
the uniqueness of the strong solutions. Let $(u_1, v_1, w_1, p_1, T_1)$ and
$(u_2, v_2, w_2, p_2, T_2)$ be two strong solutions of system
(\ref{EQ-1})--(\ref{EQ-5}) with (\ref{EQw})--(\ref{EQpz}), and initial data $((u_0)_1, (v_0)_1, (T_0)_1)$ and
$((u_0)_2, (v_0)_2, (T_0)_2)$, respectively. Denote by $u=u_1-u_2$, $v=v_1-v_2$, $w=w_1-w_2$, $p=p_1 -p_2$, $T=T_1-T_2.$ It is clear that
\begin{eqnarray}
&&\hskip-.8in
\pp_t  u - \nu u_{zz}   + u_1 u_x + w_1 u_z  + u (u_2)_x + w (u_2)_z + \ee_1 u -f_0 v + p_x = 0,  \label{UPP-1}  \\
&&\hskip-.8in
\pp_t  v - \nu v_{zz}   + u_1 v_x + w_1 v_z  + u (v_2)_x + w (v_2)_z + \ee_1 v +f_0 u = 0,  \label{UPP-2}  \\
&&\hskip-.8in
\ee_2 w + p_z +T =0,   \label{UPP-3}   \\
&&\hskip-.8in
u_x + w_z = 0, \label{UPP-5} \\
&&\hskip-.8in
\pp_t T - \kappa \Dd T  + u_1 T_x + w_1 T_z  + u (T_2)_x + w (T_2)_z  = 0.  \label{UPP-4}
\end{eqnarray}
By taking the inner product of equation (\ref{UPP-1}) with $u$, (\ref{UPP-2}) with  $v$,
(\ref{UPP-3}) with  $w$, and (\ref{UPP-4}) with  $T$,
in $L^2(\mathbb{T}^2)$, and by integration by parts, thanks to (\ref{BC-1}), (\ref{EQ-4}), and (\ref{UPP-5}), we get
\begin{eqnarray*}
	&&\hskip-.58in \frac{1}{2} \frac{d (\|u\|_{L^2}^{2}+\|v\|_{L^2}^{2}+\|T\|_{L^2}^{2})}{dt} + \ee_1 (\|u\|_{L^2}^{2} +\|v\|_{L^2}^{2}) + \nu (\|u_z\|_{L^2}^{2} +\|v_z\|_{L^2}^{2})
	+ \ee_2  \|w\|_{L^2}^{2} + \kappa  \, \| \nabla T\|_{L^2}^2 \\
	&&\hskip-.55in
	\leq  \left| \int_{\mathbb{T}^2} \left[ u (u_2)_x+w(u_2)_z \right]\; u \; dxdz \right|+\left|\int_{\mathbb{T}^2} \left[ u (v_2)_x+w(v_2)_z
	\right]\; v \; dxdz \right|   \\
	&&\hskip-.365in  +\left|\int_{\mathbb{T}^2} wT\; dxdz\right| +\left|\int_{\mathbb{T}^2} \left[ u (T_2)_x+w(T_2)_z
	\right]\; T \; dxdz\right| =: \rom{1}+\rom{2}+\rom{3}+\rom{4}.
\end{eqnarray*}
By integration by parts, using H\"older inequality and Young's inequality, thanks to (\ref{BC-1}), (\ref{EQ-4}), and (\ref{UPP-5}), we have
\begin{eqnarray*}
	&&\hskip.08in
	\rom{1} = \left| \int_{\mathbb{T}^2} \left[ u (u_2)_x+w(u_2)_z \right]\; u \; dxdz \right|
	\leq \frac{\ee_2}{8}\|w\|_{L^2}^2 + (\|(u_2)_{x}\|_{L^\infty}+\|(u_2)_{z}\|_{L^\infty}^2)\|u\|_{L^2}^2,
\end{eqnarray*}
\begin{eqnarray*}
	&&\hskip.08in
	\rom{2} = \left|\int_{\mathbb{T}^2} \left[ u (v_2)_x+w(v_2)_z
	\right]\; v \; dxdz \right| \leq \frac{\ee_2}{8}\|w\|_{L^2}^2 + (\|(v_2)_{x}\|_{L^\infty}+\|(v_2)_{z}\|_{L^\infty}^2)(\|u\|_{L^2}^2 + \|v\|_{L^2}^2), \label{UV}
\end{eqnarray*}
\begin{eqnarray*}
	&&\hskip.08in
	\rom{3} = \left|\int_{\mathbb{T}^2} wT\; dxdz\right| \leq \frac{\ee_2}{8} \|w\|_{L^2}^2 + C \|T\|_{L^2}^2,
\end{eqnarray*}
\begin{eqnarray*}
	&&\hskip-.38in
\rom{4} = \left|\int_{\mathbb{T}^2} \left[ u (T_2)_x+w(T_2)_z
\right]\; T \; dxdz\right|
\leq \frac{\ee_2}{8}\|w\|_{L^2}^2 + (\|(T_2)_{x}\|_{L^\infty}+\|(T_2)_{z}\|_{L^\infty}^2)(\|u\|_{L^2}^2 + \|T\|_{L^2}^2).
\end{eqnarray*}	
From the estimates above, we obtain
\begin{eqnarray*}
	&&\hskip-.68in \frac{d (\|u\|_{L^2}^{2}+\|v\|_{L^2}^{2}+\|T\|_{L^2}^{2})}{dt} + \ee_1 (\|u\|_{L^2}^{2} +\|v\|_{L^2}^{2}) + \nu (\|u_z\|_{L^2}^{2} +\|v_z\|_{L^2}^{2})
	+ \ee_2  \|w\|_{L^2}^{2} + \kappa  \, \| \nabla T\|_{L^2}^2 \\
	&&\hskip-.55in
	\leq   C K \left(\|u\|_{L^2}^{2}+\|v\|_{L^2}^{2}+\|T\|_{L^2}^2 \right),
\end{eqnarray*}
where
\begin{eqnarray*}
	&&\hskip-.68in
	K=1+ \|\nabla u_2\|_{L^\infty}^2 + \|\nabla v_2\|_{L^\infty}^2 + \|\nabla T_2\|_{L^\infty}^2.
\end{eqnarray*}	
Thanks to (\ref{L-2x3}), we obtain $K\in L^1(0,\mathcal{T})$. Therefore, by Gronwall inequality, we obtain
\begin{eqnarray*}
	&&\hskip-.68in
	\|u(t)\|_{L^2}^{2}+\|v(t)\|_{L^2}^{2}+\|T(t)\|_{L^2}^{2}
	\leq
	(\|u(t=0)\|_{L^2}^2 +\|v(t=0)\|_{L^2}^2 +\|T(t=0)\|_{L^2}^2)
	e^{ C \int_0^t K(s)ds},
\end{eqnarray*}
The above inequality proves the continuous dependence of the
solutions on the initial data, and in particular, when
$u(t=0)=v(t=0)=T(t=0)=0$, we have $u(t)=v(t)=T(t)=0,$ for all $t\ge 0$.
Therefore, the strong solution is unique.

\subsection{The Special Case: $\mathbf{f_0=0, v\equiv 0}$ and $\mathbf{T\equiv 0}$}
In this section, we assume that $f_0=0$, $v\equiv 0$ and $T\equiv 0$. In this case, system (\ref{EQO-1})--(\ref{EQO-5}) will be reduced to:
\begin{eqnarray}
&&\hskip-.8in
\pp_t u  -\nu  u_{zz} + u u_x +w u_z +\ee_1 u + p_x  = 0,  \label{PP-1}  \\
&&\hskip-.8in
\ee_2 w + p_z  =0,   \label{PP-2}   \\
&&\hskip-.8in u_x+w_z=0,    \label{PP-3}
\end{eqnarray}
\begin{remark}
	There are two reasons why we consider this special case. Firstly, notice that when $\ee_1=\ee_2=0$, system (\ref{PP-1})--(\ref{PP-3}) is exactly the hydrostatic Navier-Stokes equations (\ref{EQ-1-4})--(\ref{EQ-3-4}). So we can regard system (\ref{PP-1})--(\ref{PP-3}) as the hydrostatic Navier-Stokes equations with damping. Secondly, as we will see later, we can show the local regularity of strong solution to system (\ref{PP-1})--(\ref{PP-3}) for initial conditions with less regularity.
	The reason why we need to assume more regularity for initial data to system (\ref{EQ-1})--(\ref{EQ-5}) is that we need to bound terms which contain $v_{xx}$. For $u_{xx}$, we can use incompressible condition $u_{xx}=-w_{xz}$ to avoid such an issue. Therefore, in the case when we do not have the evolution equation for $v$, we can require less for the initial data.
\end{remark}
As in section 3.2, our domain is $\mathbb{T}^2$, and the boundary and initial condition are
\begin{eqnarray}
&&\hskip-.8in
u,\;  w \; \text{and} \; p \; \text{are periodic in} \; x \;  \text{and} \; z \; \text{with period 1},  \label{DBC-1} \\
&&\hskip-.8in
u \; \text{and} \; p \; \text{are even in} \; z, \; \text{and} \; w \;  \text{is odd in} \; z, \label{DBC-2} \\
&&\hskip-.8in
u|_{t=0}=u_0. \label{DIC-1}
\end{eqnarray}
Using an analogue argument to that in section 3.1, system (\ref{PP-1})--(\ref{PP-3}) subjects to (\ref{DBC-1})--(\ref{DIC-1}) is equivalent to the following:
\begin{eqnarray}
&&\hskip-.8in
u_t -\nu  u_{zz} + u\, u_x + w u_z +\ee_1 u + p_x = 0,  \label{PP-u}
\end{eqnarray}
with $w,p_x,p_z$ defined by:
\begin{eqnarray}
&&\hskip-.8in
w(x,z):=-\int_0^z u_x(x,s)ds,  \label{PPw}  \\
&&\hskip-.8in
p_{x}(x,z):= \ee_2 \int_0^z\int_0^s u_{xx}(x,\xi)d\xi ds
+\int_0^1 \Big[-\ee_2 \int_0^{z'}\int_0^s u_{xx}(x,\xi)d\xi ds -2uu_x(x,z')\Big]dz', \label{PPpx}  \\
&&\hskip-.8in
p_z(x,z):= \ee_2 \int_0^z u_x(x,s)ds, \label{PPpz}
\end{eqnarray}
subject to the following symmetry boundary condition and initial condition:
\begin{eqnarray}
&&\hskip-.8in
u \; \text{is periodic in} \; x \; \text{and} \; z  \; \text{with period 1, and is even in } \; z; \label{PPBC} \\
&&\hskip-.8in
u|_{t=0}=u_0.\label{PPIC}
\end{eqnarray}
By virtue of (\ref{PPw})--(\ref{PPpz}) and (\ref{PPBC}), we obtain that $w, p$ also satisfy the symmetry conditions:
\begin{eqnarray}
&&\hskip-.8in
w \; \text{and} \; p \; \text{are periodic in} \; x \; \text{and} \; z \; \text{with period 1}; \label{PPwpBC-1} \\
&&\hskip-.8in
p \; \text{is even in} \; z, \; \text{and} \; w \; \text{is odd in} \;z. \label{PPwpBC-2}
\end{eqnarray}
By virtue of (\ref{PPw}) and (\ref{PPpz}), and by differentiating (\ref{PPw}) with respect to $z$, we have
\begin{eqnarray}
&&\hskip-.8in
\ee_2 w + p_z  =0,    \label{PP-3a}  \\
&&\hskip-.8in
u_x + w_z =0.   \label{PP-4a}
\end{eqnarray}
In this section, we are interested in system (\ref{PP-u}) with (\ref{PPw})--(\ref{PPpz}) in the unit two-dimensional flat torus $\mathbb{T}^2$, subjects to (\ref{PPBC})--(\ref{PPIC}). First, we give the definition of strong solution to system (\ref{PP-u}) with (\ref{PPw})--(\ref{PPpz}).
\begin{definition}
	Suppose that $u_0\in H^1(\mathbb{T}^2)$ satisfies the symmetry conditions (\ref{PPBC}), with the compatibility condition $\int_0^1 \partial_x  u_{0} dz =0$. Moreover, suppose that $ \partial^2_{xz} u_{0}\in L^2(\mathbb{T}^2)$. Given time $\mathcal{T}>0$, we say u is a strong solution to system (\ref{PP-u}) with (\ref{PPw})--(\ref{PPpz}), subjects to (\ref{PPBC})--(\ref{PPIC}), on the time interval $[0,\mathcal{T}]$, if \\
	
	(i) u satisfies the symmetry condition (\ref{PPBC});\\
	
	(ii) u has the regularities
	\begin{eqnarray*}
		&&\hskip-.8in
		u\in L^{\infty}(0,\mathcal{T};H^1)\cap L^{2}(0,\mathcal{T};H^2)\cap C([0,\mathcal{T}], L^2)\cap L^{\infty}(0,\mathcal{T};L^\infty),  \;\; u_z \in L^{2}(0,\mathcal{T};L^{\infty})\\
		&&\hskip-.8in
		u_{xz}\in L^\infty(0,\mathcal{T};L^2), \;\;  u_{xzz}\in L^2(0,\mathcal{T};L^2), \;\;	\partial_t u\in L^2(0,\mathcal{T};L^2);
	\end{eqnarray*} 	
	
	(iii) u satisfies system (\ref{PP-u}) in the following sense:
	\begin{eqnarray*}
		&&\hskip-.68in
		\pp_t u  -\nu  u_{zz} + u u_x +w u_z +\ee_1 u   + p_x  = 0,  \;\;  \text{in} \; \; L^2(0,\mathcal{T}; L^2),
	\end{eqnarray*}
	with $w,p_x,p_z$ defined by (\ref{PPw})--(\ref{PPpz}), and fulfill the initial condition (\ref{PPIC}).
\end{definition}
We have the following result concerning the locally existence and uniqueness of strong solutions to system (\ref{PP-u}) with (\ref{PPw})--(\ref{PPpz}), subjects to (\ref{PPBC})--(\ref{PPIC}), on $\mathbb{T}^2 \times (0,\mathcal{T})$, for some positive time $\mathcal{T}$.
\begin{theorem}
	Suppose that $u_0 \in H^1(\mathbb{T}^2)$ satisfies the symmetry conditions (\ref{PPBC}), with the compatibility condition $\int_0^1 \partial_x u_{0} dz = 0$. Moreover, suppose that $\partial^2_{xz} u_{0}\in L^2(\mathbb{T}^2)$. Then there exists some $\mathcal{T}>0$ such that there is a unique strong solution u of system
	(\ref{PP-u}) with (\ref{PPw})--(\ref{PPpz}), subjects to (\ref{PPBC})--(\ref{PPIC}), on the interval $[0,\mathcal{T}]$. Moreover, the unique strong solution u depends continuously on the initial data.
\end{theorem}
\begin{proof}
For sake of simplicity, we will only do\textit{ a priori} estimates formally here, and we can use Galerkin method, as remarked in section 3.2, to prove the result rigorously. We denote by $\| \cdot \|:= \| \cdot \|_{L^2(\mathbb{T}^2)}$.
By taking the inner product of equation  (\ref{PP-u}) with $u$, $-u_{zz}$, and
equation  (\ref{PP-3a}) with $w$, $-w_{zz}$, in $L^2(\mathbb{T}^2)$,  we get
\begin{eqnarray*}
	&&\hskip-.68in \frac{1}{2} \frac{d} {dt} \big(\|u\|^{2} + \|u_z\|^{2}\big)
	+ \nu\, \big(\|u_z\|^2 + \|u_{zz}\|^2\big) + \ee_1 \big(\|u\|^2 + \|u_z\|^2 \big)+ \ee_2  \big(\|w\|^{2} + \|w_z\|^{2}\big)  \\
	&&\hskip-.78in
	=  -\int_{\mathbb{T}^2} \left(u u_x+wu_z
	\right)\; (u-u_{zz}) \; dxdz -\int_{\mathbb{T}^2} \Big(p_x\left(u-u_{zz}\right) + p_z \left(w-w_{zz}\right)\Big)\; dxdz.
\end{eqnarray*}
By integration by parts, thanks to (\ref{PPBC}), (\ref{PPwpBC-1}) and (\ref{PP-4a}), we have
\begin{eqnarray*}
	&&\hskip-.68in  -\int_{\mathbb{T}^2} \left(u u_x+wu_z
	\right)\; (u-u_{zz}) \; dxdz -\int_{\mathbb{T}^2} \Big(p_x\left(u-u_{zz}\right) + p_z \left(w-w_{zz}\right)\Big)\; dxdz = 0.
\end{eqnarray*}
Thanks to Gronwall inequality, we obtain
\begin{eqnarray*}
	&&\hskip-.28in \|u(t)\|^2 + \|u_z(t)\|^2 + 2\int_0^t \Big[\nu \big(\|u_z(s)\|^2 + \|u_{zz}(s)\|^2 \big) +  \ee_2 \big(\|w(s)\|^2 + \|w_z(s)\|^2 \big) \Big]ds
	 \leq \|u(0)\|^2 + \|u_z(0)\|^2.
\end{eqnarray*}
From the estimates above, we obtain
\begin{eqnarray}
&&\hskip-.68in
u, \; u_z  \; \text{bounded in} \; L^{\infty}(0,\mathcal{T};L^2), \nonumber\\
&&\hskip-.68in
w, \;u_{zz}, \; w_z=-u_x   \; \text{bounded in} \; L^{2}(0,\mathcal{T};L^2), \label{DB-1}
\end{eqnarray}
for arbitrary $\mathcal{T} >0$. By taking the inner product of equation  (\ref{PP-u}) with $-u_{xx}, u_{xxzz}$ and
equation  (\ref{PP-3a}) with $-w_{xx}$, $w_{xxzz}$ in $L^2(\mathbb{T}^2)$, integrating by parts, thanks to (\ref{PPBC}) and (\ref{PPwpBC-1}) we get
\begin{eqnarray}
&&\hskip-.68in \frac{1}{2} \frac{d }{dt}(\|u_x\|^{2} + \|u_{xz}\|^{2})
+ \nu\, (\|u_{xz}\|^2 + \|u_{xzz}\|^2) + \ee_1  (\|u_x\|^{2} +  \|u_{xz}\|^{2}) + \ee_2  (\|w_x\|^{2} +  \|w_{xz}\|^{2}) \nonumber\\
&&\hskip-.65in  = \int_{\mathbb{T}^2} (uu_x + wu_z)
\;  (u_{xx} - u_{xxzz})  \; dxdz  + \int_{\mathbb{T}^2} \Big(p_x (u_{xx}-u_{xxzz}) + p_z( w_{xx}-w_{xxzz})\Big)  dxdz. \label{D}
\end{eqnarray}
By integration by parts, thanks to (\ref{PPBC}), (\ref{PPwpBC-1}) and (\ref{PP-4a}), we have
\begin{eqnarray*}
	&&\hskip-.68in  \int_{\mathbb{T}^2} \Big(p_x (u_{xx}-u_{xxzz}) + p_z( w_{xx}-w_{xxzz})\Big)  dxdz = 0.
\end{eqnarray*}
Therefore, we have
\begin{eqnarray*}
	&&\hskip-.68in \frac{1}{2} \frac{d }{dt}(\|u_x\|^{2} + \|u_{xz}\|^{2})
	+ \nu\, (\|u_{xz}\|^2 + \|u_{xzz}\|^2) + \ee_1  (\|u_x\|^{2} +  \|u_{xz}\|^{2}) + \ee_2  (\|w_x\|^{2} +  \|w_{xz}\|^{2}) \nonumber\\
	&&\hskip-.65in  \leq |\int_{\mathbb{T}^2} (uu_x + wu_z)
	\;  (u_{xx} - u_{xxzz})  \; dxdz| =: |I_1+I_2|.
\end{eqnarray*}
Let us denote by
$$ Y:=1+\|u_x\|^{2} + \|u_{xz}\|^{2}, \;\; F:=\|u_{xz}\|^2 + \|u_{xzz}\|^2,$$
$$ G:= \|w_x\|^{2}  + \|w_{xz}\|^{2}, \; \; K :=1+ \|u\|^{2} + \|u_{z}\|^{2} + \|u_{zz}\|^{2}. $$
By integration by parts and Lemma 1, using Young's inequality, thanks to (\ref{ESw}), (\ref{ESwx}), (\ref{PPBC}), (\ref{PPwpBC-1}) and (\ref{PP-4a}), we have
\begin{eqnarray*}
	&&\hskip-.48in
	|I_1| = \left|\int_{\mathbb{T}^2} \left( u u_{x}+w u_{z}
	\right) u_{xx} dxdz \right|
	\leq C \|u\|_{L^2}^{\frac{1}{2}} (\|u\|_{L^2}^{\frac{1}{2}} + \|u_x\|_{L^2}^{\frac{1}{2}}) \|u_x\|_{L^2}^{\frac{1}{2}} (\|u_x\|_{L^2}^{\frac{1}{2}} + \|u_{xz}\|_{L^2}^{\frac{1}{2}})\|w_{xz}\|_{L^2} \\
	&&\hskip -.25in + C \|u_z\|_{L^2}^{\frac{1}{2}} (\|u_z\|_{L^2}^{\frac{1}{2}} + \|u_{xz}\|_{L^2}^{\frac{1}{2}}) \|w\|_{L^2}^{\frac{1}{2}} (\|w\|_{L^2}^{\frac{1}{2}} + \|w_{z}\|_{L^2}^{\frac{1}{2}}) \|w_{xz}\|_{L^2} \leq \frac{\ee_2}{4}G +  CKY^2, \\
	&&\hskip-.48in	
	|I_2| = \left|\int_{\mathbb{T}^2} \left[ u u_{x}+w u_{z}
	\right]\; u_{xxzz} \; dxdz\right|=
	\left|  \int_{\mathbb{T}^2} \left( u_xu_{xz}+w_x u_{zz}
	\right)\; u_{xz} \; dxdz \right|  \\
	&&\hskip-.25in
	\leq C \Big[  \|u_x\|_{L^2}^{\frac{1}{2}} (\|u_x\|_{L^2}^{\frac{1}{2}} + \|w_{xz}\|_{L^2}^{\frac{1}{2}}) \|u_{xz}\|_{L^2}^{\frac{3}{2}} (\|u_{xz}\|_{L^2}^{\frac{1}{2}} + \|u_{xzz}\|_{L^2}^{\frac{1}{2}})\\
	&&\hskip-.05in
	+ \|u_{xz} \|_{L^2}   \|w_{x}\|_{L^2}^{\frac{1}{2}} ( \|w_{x}\|_{L^2}^{\frac{1}{2}}+ \|w_{xz}\|_{L^2}^{\frac{1}{2}}) \|u_{zz}\|_{L^2}^{\frac{1}{2}} (\|u_{zz}\|_{L^2}^{\frac{1}{2}}+ \|u_{xzz}\|_{L^2}^{\frac{1}{2}})\Big]  \leq \frac{\ee_2}{4}G + \frac{\nu}{2}F + CKY^2.
\end{eqnarray*}
From the estimates above and by (\ref{DB-1}), we have
\begin{eqnarray}
&&\hskip-.68in \frac{dY}{dt} + \nu F + \ee_2 G \leq C K Y^2, \; \; \; \; \text{with}  \;\; K\in L^1(0,\mathcal{T}) \;\;\text{for arbitrary} \; \;\mathcal{T} >0.\label{PP-5}
\end{eqnarray}
Therefore, we have $\frac{dY}{dt} \leq C K Y^2$, and this implies that
$$Y(t)\leq \frac{Y(0)}{1-Y(0)C\int_0^t Kds}.$$
Let $\mathcal{T}$ be such that $$\int_0^\mathcal{T} Kds = \frac{1}{2Y(0)C}.$$
From above, we will have $Y(t)\leq 2Y(0)$ on $[0,\mathcal{T}]$.
Plugging it in (\ref{PP-5}), we have
\begin{eqnarray*}
	&&\hskip-.68in \frac{dY}{dt} + \nu F + \ee_2 G \leq 4C K Y(0)^2,  \;\text{for}\; t\in[0,\mathcal{T}].
\end{eqnarray*}
Integrating above from 0 to t for any time $t\in[0,\mathcal{T}]$, we obtain
\begin{eqnarray*}
	&&\hskip-.68in Y(t) + \int_0^t \left(\nu F(s) + \ee_2 G(s) \right) \; ds \leq Y(0)+ 4C Y(0)^2 \int_0^t K(s) ds.
\end{eqnarray*}
From the estimates above, by virtue of (\ref{PP-4a}), (\ref{DB-1}) and (\ref{ESw}), we obtain
\begin{eqnarray}
&&\hskip-.68in
u  \in L^{\infty}(0,\mathcal{T};H^1)\cap L^{2}(0,\mathcal{T};H^2), \;\; u_{xz}\in L^{\infty}(0,\mathcal{T};L^2), \;\; u_{xzz} \in L^2(0,\mathcal{T};L^2), \label{DB-3} \\
&&\hskip-.68in
w, w_z, w_{zz}\in L^{\infty}(0,\mathcal{T};L^2), \;\; w_x, w_{xz}\in L^{2}(0,\mathcal{T};L^2) \label{DB-4}.
\end{eqnarray}
Using Galerkin method, as remarked in section 3.2, we can obtain local existence of strong solution to system (\ref{PP-u}) with (\ref{PPw})--(\ref{PPpz}), subjects to (\ref{PPBC})--(\ref{PPIC}).
Next,  we show the continuous dependence of solutions on the initial data and
the uniqueness of the strong solutions. Let $(u_1, w_1, p_1)$ and
$(u_2, w_2, p_2)$ be two strong solutions of system
(\ref{PP-u}) with (\ref{PPw})--(\ref{PPpz}), and initial data $(u_0)_1$ and
$(u_0)_2$, respectively. Denote by $u=u_1-u_2,  w=w_1-w_2, p=p_1 -p_2.$
It is clear that
\begin{eqnarray}
&&\hskip-.8in
\pp_t u   + u_1 u_x + w_1 u_z  + u (u_2)_x + w (u_2)_z  + \ee_1 u - \nu u_{zz} + p_x = 0,  \label{uPP-1}  \\
&&\hskip-.8in
\ee_2 w + p_z  =0.   \label{uPP-2}
\end{eqnarray}
By taking the inner product of equation (\ref{uPP-1}) with $u$, (\ref{uPP-2}) with  $w$
in $L^2(\mathbb{T}^2)$, we have
\begin{eqnarray*}
	&&\hskip-.28in \frac{1}{2} \frac{d \|u\|^{2}}{dt} + \ee_1 \|u\|^2
	+ \ee_2  \|w\|^{2} + \nu \|u_z\|^{2} = \int_{\mathbb{T}^2} u(u_1 u_x + w_1 u_z + u (u_2)_x + w (u_2)_z) + (p_xu + p_z w) \; dxdz  .
\end{eqnarray*}
By integration by parts, thanks to (\ref{PPBC}), (\ref{PPwpBC-1}) and (\ref{PP-4a}), we have
\begin{eqnarray*}
	&&\hskip-.8in \int_{\mathbb{T}^2} u(u_1 u_x + w_1 u_z ) + (p_xu + p_z w) \; dxdz  = 0 .
\end{eqnarray*}
Therefore, we have
\begin{eqnarray*}
	&&\hskip-.68in \frac{1}{2} \frac{d \|u\|^{2}}{dt} + \ee_1 \|u\|^2
	+ \ee_2  \|w\|^{2} + \nu \|u_z\|^{2} \leq |\int_{\mathbb{T}^2} u\Big( u (u_2)_x + w (u_2)_z\Big) \;dxdz| =: |I_1+I_2|  .
\end{eqnarray*}
From (\ref{DB-3}) and (\ref{DB-4}), and by lemma 2, we obtain that $w_2, (u_2)_z \in L^{2}(0,\mathcal{T};L^\infty)$. Therefore, using Young's inequality and H\"older inequality, we have
\begin{eqnarray*}
	&&\hskip-.68in |I_1| = |\int_{\mathbb{T}^2} u^2 (u_2)_x dxdz | = | \int_{\mathbb{T}^2} u^2 (w_2)_z dxdz | = | 2 \int_{\mathbb{T}^2} u u_z w_2 \; dxdz | \\
	&&\hskip-.38in
	 \leq \int_{\mathbb{T}^2} \left( \frac{\nu}{2} |u_z|^2 + C|uw_2|^2  \right) dxdz \leq C\|w_2\|_{L^\infty}^2 \|u\|^2 + \frac{\nu}{2} \|u_z\|^2,
\end{eqnarray*}
\begin{eqnarray*}
	&&\hskip-.68in |I_2| = | \int_{\mathbb{T}^2} u w (u_2)_z dxdz | \leq \int_{\mathbb{T}^2} \left( \frac{\ee_2}{2} |w|^2 + C|u(u_2)_z|^2 \right) dxdz  \leq C\|(u_2)_z\|_{L^\infty}^2 \|u\|^2 + \frac{\ee_2}{2} \|w\|^2.
\end{eqnarray*}
From the estimates above, we obtain
\begin{eqnarray*}
	&&\hskip-.68in \frac{d}{dt}\|u\|^2 + \ee_1 \|u\|^2 + \ee_2 \|w\|^2 + \nu \|u_z\|^2 \leq C(\|w_2\|_{L^\infty}^2 + \|(u_2)_z\|_{L^\infty}^2)\|u\|^2.
\end{eqnarray*}
Thanks to Gronwall inequality, we have
\begin{eqnarray*}
	&&\hskip-.68in \|u(t)\|^2 \leq \|u(0)\|^2 \exp \left(C\int_0^t \left(\|w_2(s)\|_{L^\infty}^2 + \|(u_2)_z(s)\|_{L^\infty}^2 \right)ds \right).
\end{eqnarray*}
The above inequality proves the continuous dependence of the
solutions on the initial data, and in particular, when
$u(t=0)=0$, we have $u(t)=0,$ for all $t\in [0, \mathcal{T}]$.
Therefore, the strong solution is unique.
\end{proof}
\subsection{Global Well-posedness with Small Initial Data}
In this section, we will show the following result concerning the global existence and uniqueness of strong solutions to system (\ref{EQ-1})--(\ref{EQ-5}) with (\ref{EQw})--(\ref{EQpz}), subjects to boundary and initial conditions (\ref{BC-1})--(\ref{IC-1}), provided that the initial data is small enough.
\begin{theorem} \label{T-MAIN2}
	Suppose that $u_0, v_0, T_0, \partial_x u_{0}, \partial_x  v_{0}, \partial_x T_{0}\in H^1(\mathbb{T}^2)$ satisfy the symmetry conditions (\ref{BC-1}) and (\ref{BC-2}), with the compatibility condition $\int_0^1 \partial_x u_{0} dz = 0$. Moreover, suppose that $$\|u_0\|_{H^1}+ \|v_0\|_{H^1} + C_0 \|T_0\|_{H^1} + \|\partial_x u_{0}\|_{H^1} + \|\partial_x v_{0}\|_{H^1} + C_0 \|\partial_x T_{0}\|_{H^1} <<1$$ is small enough, for some $C_0>0$ determined in (\ref{C0}). Then for any time $\mathcal{T}>0$, there exists a unique strong solution $(u, v, T)$ of system
	(\ref{EQ-1})--(\ref{EQ-5}) with (\ref{EQw})--(\ref{EQpz}), subjects to (\ref{BC-1})--(\ref{IC-1}), on the interval $[0,\mathcal{T}]$. Moreover, the unique strong solution $(u, v, T)$ depends continuously on the initial data.
\end{theorem}
\begin{proof}
From Theorem \ref{T-MAIN}, we know there exists time $\mathcal{T}^*>0$ such that there is a unique strong solution $(u,v,T)$ of system (\ref{EQ-1})--(\ref{EQ-5}) with (\ref{EQw})--(\ref{EQpz}), subjects to (\ref{BC-1})--(\ref{IC-1}), on the interval $[0,\mathcal{T}^*]$. Assume the maximal time $\mathcal{T}$ for existence of solution is finite, then it is necessary to have $$\limsup\limits_{t\rightarrow \mathcal{T}^{-}} (\|u(t)\|_{H^1}+ \|v(t)\|_{H^1} + \|T(t)\|_{H^1} +\|u_{x}(t)\|_{H^1}+ \|v_{x}(t)\|_{H^1} + \|T_{x}(t)\|_{H^1} )=\infty. $$ We will prove this is not true for any finite time $\mathcal{T}>0$, and therefore $\mathcal{T}=\infty$. First, notice that since $T$ is an odd function with respect to $z$ variable, we have
\begin{eqnarray}
\int_{\mathbb{T}^2} T \;dxdz \equiv 0. \label{poincare}
\end{eqnarray}
By taking the $L^2$-inner product of equation  (\ref{EQ-1}) with $u, -\Dd u, \Dd u_{xx}$, equation  (\ref{EQ-2}) with $v, -\Dd v, \Dd v_{xx}$,
equation  (\ref{EQ-3}) with $w, -\Dd w, \Dd w_{xx}$ and
equation  (\ref{EQ-5}) with $C_0 T, -C_0 \Dd T, C_0 \Dd T_{xx}$, in $L^2(\mathbb{T}^2)$, by integration by parts, thanks to (\ref{BC-1}), (\ref{wpBC-1}) and (\ref{EQ-4}), we have
\begin{eqnarray*}
&&\hskip-.28in \frac{1}{2} \frac{d}{dt} \Big(\|u\|_{L^2}^{2} + \|\nabla u\|_{L^2}^{2} +  \|v\|_{L^2}^{2} + \|\nabla v\|_{L^2}^{2} + \|\nabla u_{x}\|_{L^2}^{2} + \|\nabla v_{x}\|_{L^2}^{2} +C_0\|T\|_{L^2}^{2} + C_0\|\nabla T\|_{L^2}^{2}  + C_0\|\nabla T_x\|_{L^2}^{2}\Big)\nonumber \\
&&\hskip-.08in
+ \nu \Big(\|u_z\|_{L^2}^{2} + \|v_z\|_{L^2}^{2} + \|\nabla u_z\|_{L^2}^{2} + \|\nabla v_z\|_{L^2}^{2}+ \|\nabla u_{xz}\|_{L^2}^{2} + \|\nabla v_{xz}\|_{L^2}^{2}\Big)
\nonumber \\
&&\hskip-.08in + \ee_1 \Big(\|u\|_{L^2}^{2} + \|\nabla u\|_{L^2}^{2} +  \|v\|_{L^2}^{2} + \|\nabla v\|_{L^2}^{2} + \|\nabla u_{x}\|_{L^2}^{2} + \|\nabla v_{x}\|_{L^2}^{2}\Big)
\nonumber \\
&&\hskip-.08in
+ \ee_2 \Big(\|w\|_{L^2}^{2} + \|\nabla w\|_{L^2}^{2}+ \|\nabla w_{x}\|_{L^2}^{2}\Big) + C_0 \kappa \Big( \|\nabla T\|_{L^2}^{2} + \|\Dd T\|_{L^2}^{2} + \|\Dd T_x\|_{L^2}^{2} \Big) \nonumber \\
&&\hskip-.28in
= \int_{\mathbb{T}^2} \left[ u u_{x}+w u_{z} - f_0 v + p_x
\right] \left(-u+ \Dd u - \Dd u_{xx}\right) + \left[ u v_{x}+w v_{z} + f_0 u
\right] \left(-v+ \Dd v - \Dd v_{xx}\right)\nonumber \\
&&\hskip-.08in
+ \left(  p_{z}+ T\right) \left(-w + \Dd w - \Dd w_{xx}\right) + C_0 \left(uT_x+wT_z\right)(-T+ \Dd T -\Dd T_{xx})\; dxdz\\
&&\hskip-.28in
= \int_{\mathbb{T}^2} \left( u u_{x}+w u_{z}
\right)\left(u_{xx} - u_{xxxx}- u_{xxzz}\right) + \left( u v_{x}+w v_{z}
\right) \left(\Dd v - v_{xxxx} - v_{xxzz}\right)
\\
&&\hskip-.08in + T \left(-w + \Dd w - \Dd w_{xx} \right) + C_0\left(uT_x+wT_z\right)(\Dd T - \Dd T_{xx})\; dxdz
=: I_1+I_2+I_3+I_4 +I_5+I_6+I_7+I_8.
\end{eqnarray*}
We denote by:
$$ Y:=\|u\|_{L^2}^{2} + \|\nabla u\|_{L^2}^{2} +  \|v\|_{L^2}^{2} + \|\nabla v\|_{L^2}^{2} + \|\nabla u_{x}\|_{L^2}^{2} + \|\nabla v_{x}\|_{L^2}^{2} +C_0\|T\|_{L^2}^{2} + C_0\|\nabla T\|_{L^2}^{2}  + C_0\|\nabla T_x\|_{L^2}^{2},$$
$$ F:=\|u_z\|_{L^2}^{2} + \|v_z\|_{L^2}^{2} + \|\nabla u_z\|_{L^2}^{2} + \|\nabla v_z\|_{L^2}^{2}+ \|\nabla u_{xz}\|_{L^2}^{2} + \|\nabla v_{xz}\|_{L^2}^{2},$$
$$ G:= \|w\|_{L^2}^{2} + \|\nabla w\|_{L^2}^{2}+ \|\nabla w_{x}\|_{L^2}^{2},$$
$$ H:= \|\nabla T\|_{L^2}^{2} + \|\Dd T\|_{L^2}^{2} + \|\Dd T_x\|_{L^2}^{2}, $$
$$ K:= \|u\|_{L^2}^{2} + \|\nabla u\|_{L^2}^{2} +  \|v\|_{L^2}^{2} + \|\nabla v\|_{L^2}^{2} + \|\nabla u_{x}\|_{L^2}^{2} + \|\nabla v_{x}\|_{L^2}^{2}. $$
By integration by parts, using Poincar\'e inequality, Young's inequality and Lemma 1, thanks to (\ref{BC-1}), (\ref{wpBC-1}), (\ref{EQ-4}), (\ref{ESw}), (\ref{ESwx}) and (\ref{poincare}), we have
\begin{eqnarray*}
	&&\hskip-.45in  |I_1| = \left|\int_{\mathbb{T}^2} \left( u u_{x}+w u_{z}
	\right) w_{xz} dxdz \right| \\
	&&\hskip-.20in
	\leq C \|u\|_{L^2}^{\frac{1}{2}} (\|u\|_{L^2}^{\frac{1}{2}} + \|u_z\|_{L^2}^{\frac{1}{2}}) \|u_x\|_{L^2}^{\frac{1}{2}} (\|u_x\|_{L^2}^{\frac{1}{2}} + \|u_{xx}\|_{L^2}^{\frac{1}{2}})\|w_{xz}\|_{L^2} \\
	&&\hskip-.15in + C \|u_z\|_{L^2}^{\frac{1}{2}} (\|u_z\|_{L^2}^{\frac{1}{2}} + \|u_{xz}\|_{L^2}^{\frac{1}{2}}) \|w\|_{L^2}^{\frac{1}{2}} (\|w\|_{L^2}^{\frac{1}{2}} + \|w_{z}\|_{L^2}^{\frac{1}{2}}) \|w_{xz}\|_{L^2}\\
	&&\hskip-.20in
	\leq C(\|w\|_{L^2}^2 + \|w_z\|_{L^2}^2 + \|w_{xz}\|_{L^2}^2)(\|u\|_{L^2}+ \|u_z\|_{L^2} + \|u_{xz}\|_{L^2}) \leq CGY^{1/2},
\end{eqnarray*}
\begin{eqnarray*}
	&&\hskip-.45in  |I_2| = \left|\int_{\mathbb{T}^2} \left[ u u_{x}+w u_{z}
	\right]\; u_{xxxx} \; dxdz\right|=
	\left|  \int_{\mathbb{T}^2} \left[ 3w_{z} w_{xz}+w_{xx} u_{z} + 2w_{x} u_{xz}
	\right]\; u_{xx} \; dxdz \right|  \\
	&&\hskip-.20in
	\leq    C \Big[  \|w_z\|_{L^2}^{\frac{1}{2}} (\|w_z\|_{L^2}^{\frac{1}{2}} + \|w_{xz}\|_{L^2}^{\frac{1}{2}}) \|w_{xz}\|_{L^2} \|u_{xx}\|_{L^2}^{\frac{1}{2}}(\|u_{xx}\|_{L^2}^{\frac{1}{2}} + \|u_{xxz}\|_{L^2}^{\frac{1}{2}})\\
	&&\hskip-.15in
	+ \|w_{xx} \|_{L^2}   \|u_{z}\|_{L^2}^{\frac{1}{2}} ( \|u_{z}\|_{L^2}^{\frac{1}{2}}+ \|u_{xz}\|_{L^2}^{\frac{1}{2}}) \|u_{xx}\|_{L^2}^{\frac{1}{2}} (\|u_{xx}\|_{L^2}^{\frac{1}{2}}+ \|u_{xxz}\|_{L^2}^{\frac{1}{2}})  \\
	&&\hskip-.15in
	+  \|w_{x} \|_{L^2}  \|u_{xz}\|_{L^2}^{\frac{1}{2}}(\|u_{xz}\|_{L^2}^{\frac{1}{2}} + \|u_{xxz}\|_{L^2}^{\frac{1}{2}} ) \|u_{xx}\|_{L^2}^{\frac{1}{2}} (\|u_{xx}\|_{L^2}^{\frac{1}{2}} +  \|u_{xxz}\|_{L^2}^{\frac{1}{2}}) \Big] \\
	&&\hskip-.20in
	\leq C( \|\nabla w_{x}\|_{L^2}^2 + \|\nabla u\|_{L^2}^2 + \|\nabla u_x\|_{L^2}^2 +  \|u_{xxz}\|_{L^2}^2 )(\|\nabla u\|_{L^2} + \|\nabla u_{x}\|_{L^2})\leq C(F+G+K)Y^{1/2},\\
	&&\hskip-.45in  |I_3| = \left|\int_{\mathbb{T}^2} \left[ u u_{x}+w u_{z}
	\right]\; u_{xxzz} \; dxdz\right|=
	\left|  \int_{\mathbb{T}^2} \left( u_xu_{xz}+w_x u_{zz}
	\right)\; u_{xz} \; dxdz \right|  \\
	&&\hskip-.20in
	\leq    C \Big[  \|u_x\|_{L^2}^{\frac{1}{2}} (\|u_x\|_{L^2}^{\frac{1}{2}} + \|u_{xz}\|_{L^2}^{\frac{1}{2}}) \|u_{xz}\|_{L^2}^{\frac{3}{2}} (\|u_{xz}\|_{L^2}^{\frac{1}{2}} + \|u_{xxz}\|_{L^2}^{\frac{1}{2}})\\
	&&\hskip-.15in
	+ \|u_{xz} \|_{L^2}   \|w_{x}\|_{L^2}^{\frac{1}{2}} ( \|w_{x}\|_{L^2}^{\frac{1}{2}}+ \|w_{xz}\|_{L^2}^{\frac{1}{2}}) \|u_{zz}\|_{L^2}^{\frac{1}{2}} (\|u_{zz}\|_{L^2}^{\frac{1}{2}}+ \|u_{xzz}\|_{L^2}^{\frac{1}{2}}) \Big]
	\leq C(F+G+K)Y^{1/2},\\
	&&\hskip-.45in  |I_4| = \left|\int_{\mathbb{T}^2} \left[ u v_{x}+w v_{z}
	\right]\; \Dd v \; dxdz\right|	\\
	&&\hskip-.20in
	\leq    C (\|u_{}\|_{L^2}+\|u_{z}\|_{L^2})(\|v_{x}\|_{L^2}+\|v_{xx}\|_{L^2})(\|v_{xx}\|_{L^2}+ \|v_{zz}\|_{L^2}) \\
	&&\hskip-.15in
	+ C(\|w_{}\|_{L^2}+\|w_{z}\|_{L^2})(\|v_{z}\|_{L^2} + \|v_{xz}\|_{L^2})(\|v_{xx}\|_{L^2}+ \|v_{zz}\|_{L^2}) \leq C(K+F)Y^{1/2},\\
	&&\hskip-.45in  |I_5| = |\int_{\mathbb{T}^2} \left[ u v_{x}+w v_{z}
	\right]\; v_{xxxx} dxdz | =
	|\int_{\mathbb{T}^2} \left[ u_{xx} v_{x} + w_{xx} v_{z} +2 u_x v_{xx} + 2w_{x} v_{xz} \right]\; v_{xx}  dxdz| \\
	&&\hskip-.20in
	\leq C \Big[ \|v_{xx}\|_{L^2}\|v_{x}\|_{L^2}^{\frac{1}{2}} (\|v_{x}\|_{L^2}^{\frac{1}{2}} + \|v_{xx}\|_{L^2}^{\frac{1}{2}}) \|u_{xx}\|^{\frac{1}{2}}_{L^2} (\|u_{xx}\|^{\frac{1}{2}}_{L^2} + \|u_{xxz}\|_{L^2}^{\frac{1}{2}})\\
	&&\hskip-.15in
	+ \|w_{xx}\|_{L^2}\|v_{z}\|_{L^2}^{\frac{1}{2}} (\|v_{z}\|_{L^2}^{\frac{1}{2}} + \|v_{xz}\|_{L^2}^{\frac{1}{2}}) \|v_{xx}\|_{L^2}^{\frac{1}{2}} ( \|v_{xx}\|_{L^2}^{\frac{1}{2}} + \|v_{xxz}\|_{L^2}^{\frac{1}{2}} )  \\
	&&\hskip-.15in
	+\|v_{xx}\|_{L^2}^{\frac{3}{2}}(\|v_{xx}\|_{L^2}^{\frac{1}{2}} + \|v_{xxz}\|_{L^2}^{\frac{1}{2}} )\|u_{x}\|_{L^2}^{\frac{1}{2}} ( \|u_{x}\|_{L^2}^{\frac{1}{2}} + \|u_{xx}\|_{L^2}^{\frac{1}{2}}) \\
	&&\hskip-.15in
	+ \|w_x\|_{L^2}^{\frac{1}{2}} (\|w_x\|_{L^2}^{\frac{1}{2}} + \|w_{xz}\|_{L^2}^{\frac{1}{2}})  \|v_{xz}\|_{L^2}^{\frac{1}{2}} (\|v_{xz}\|_{L^2}^{\frac{1}{2}} + \|v_{xxz}\|_{L^2}^{\frac{1}{2}}) \|v_{xx}\|_{L^2} \Big]  \leq C (K+F+G)Y^{1/2},\\
	&&\hskip-.45in  |I_6| = |\int_{\mathbb{T}^2} \left[ u v_{x}+w v_{z}
	\right]\; v_{xxzz} dxdz | =
	|\int_{\mathbb{T}^2} \left[ u_{xz} v_{x} + v_{xx} u_{z} - v_z u_{xx} + w_{x} v_{xz} \right]\; v_{xz}  dxdz| \\
	&&\hskip-.20in
	\leq C \Big[ \|v_{xz}\|_{L^2}\|v_{x}\|_{L^2}^{\frac{1}{2}} (\|v_{x}\|_{L^2}^{\frac{1}{2}} + \|v_{xz}\|_{L^2}^{\frac{1}{2}}) \|u_{xz}\|^{\frac{1}{2}}_{L^2} (\|u_{xz}\|^{\frac{1}{2}}_{L^2} + \|u_{xxz}\|_{L^2}^{\frac{1}{2}})\\
	&&\hskip-.15in
	+ \|v_{xz}\|_{L^2}\|u_{z}\|_{L^2}^{\frac{1}{2}} (\|u_{z}\|_{L^2}^{\frac{1}{2}} + \|u_{xz}\|_{L^2}^{\frac{1}{2}}) \|v_{xx}\|_{L^2}^{\frac{1}{2}} ( \|v_{xx}\|_{L^2}^{\frac{1}{2}} + \|v_{xxz}\|_{L^2}^{\frac{1}{2}} )  \\
	&&\hskip-.15in
	+\|v_{xz}\|_{L^2}\|v_{z}\|_{L^2}^{\frac{1}{2}} ( \|v_{z}\|_{L^2}^{\frac{1}{2}} + \|v_{xz}\|_{L^2}^{\frac{1}{2}})\|u_{xx}\|_{L^2}^{\frac{1}{2}}(\|u_{xx}\|_{L^2}^{\frac{1}{2}} + \|u_{xxz}\|_{L^2}^{\frac{1}{2}} ) \\
	&&\hskip-.15in
	+ \|v_{xz}\|_{L^2}\|w_x\|_{L^2}^{\frac{1}{2}} (\|w_x\|_{L^2}^{\frac{1}{2}} + \|w_{xz}\|_{L^2}^{\frac{1}{2}})  \|v_{xz}\|_{L^2}^{\frac{1}{2}} (\|v_{xz}\|_{L^2}^{\frac{1}{2}} + \|v_{xxz}\|_{L^2}^{\frac{1}{2}})  \Big] \leq C(K+F+G)Y^{1/2},\\
	&&\hskip-.45in  |I_7| = |\int_{\mathbb{T}^2} T \left(-w + \Dd w - \Dd w_{xx} \right) dxdz |
	\\
	&&\hskip-.20in
	\leq \|T\|_{L^2}\|w\|_{L^2}+ \|\nabla T\|_{L^2} \|\nabla w\|_{L^2} +\|\nabla T_x\|_{L^2} \|\nabla w_x\|_{L^2} \\
	&&\hskip-.20in
	\leq \frac{\ee_2}{2}G + \frac{1}{2\ee_2}(\|T\|_{L^2}^2 + \|\nabla T\|_{L^2}^2+\|\nabla T_x\|_{L^2}^2 ) \\
	&&\hskip-.20in
	\leq \frac{\ee_2}{2}G + \frac{1}{2\ee_2}(C_p \|\nabla T\|_{L^2}^2 + \|\nabla T\|_{L^2}^2+\|\nabla T_x\|_{L^2}^2 )\leq \frac{\ee_2}{2}G + \frac{1}{2\ee_2}(C_p+1)H,
	\end{eqnarray*}
where, thanks to (\ref{poincare}), we apply Poincar\'e inequality to obtain the last inequality.
\begin{eqnarray*}
	&&\hskip-.45in  |I_8| = C_0|\int_{\mathbb{T}^2} \left[ u T_{x}+w T_{z}
	\right]\; (\Dd T - \Dd T_{xx}) dxdz | \\
	&&\hskip-.20in
	\leq |\int_{\mathbb{T}^2} \left[ u T_{x}+w T_{z}
	\right]\; \Dd T  dxdz | + |\int_{\mathbb{T}^2} \left[ u_x T_{x} + uT_{xx} + wT_{xz}+w_x T_{z}
	\right]\;  \Dd T_{x} dxdz |\\
	&&\hskip-.20in
	\leq C_0 C \Big[ \|u\|_{L^2}^{\frac{1}{2}} (\|u\|_{L^2}^{\frac{1}{2}} + \|u_{z}\|_{L^2}^{\frac{1}{2}}) \|T_{x}\|^{\frac{1}{2}}_{L^2} (\|T_{x}\|^{\frac{1}{2}}_{L^2} + \|T_{xx}\|_{L^2}^{\frac{1}{2}}) \\
	&&\hskip.25in
	+ \|w\|_{L^2}^{\frac{1}{2}} (\|w\|_{L^2}^{\frac{1}{2}} + \|w_{z}\|_{L^2}^{\frac{1}{2}}) \|T_{z}\|^{\frac{1}{2}}_{L^2} (\|T_{z}\|^{\frac{1}{2}}_{L^2} + \|T_{xz}\|_{L^2}^{\frac{1}{2}}) \Big] \|\Dd T\|_{L^2}\\
	&&\hskip-.15in
	+ C_0 C\Big[ \|u_x\|_{L^2}^{\frac{1}{2}} (\|u_x\|_{L^2}^{\frac{1}{2}} + \|u_{xz}\|_{L^2}^{\frac{1}{2}}) \|T_{x}\|^{\frac{1}{2}}_{L^2} (\|T_{x}\|^{\frac{1}{2}}_{L^2} + \|T_{xx}\|_{L^2}^{\frac{1}{2}}) \\
	&&\hskip.25in
	+ \|u\|_{L^2}^{\frac{1}{2}} (\|u\|_{L^2}^{\frac{1}{2}} + \|u_{z}\|_{L^2}^{\frac{1}{2}}) \|T_{xx}\|^{\frac{1}{2}}_{L^2} (\|T_{xx}\|^{\frac{1}{2}}_{L^2} + \|T_{xxx}\|_{L^2}^{\frac{1}{2}}) \\
	&&\hskip.25in
	+ \|w\|_{L^2}^{\frac{1}{2}} (\|w\|_{L^2}^{\frac{1}{2}} + \|w_{x}\|_{L^2}^{\frac{1}{2}}) \|T_{xz}\|^{\frac{1}{2}}_{L^2} (\|T_{xz}\|^{\frac{1}{2}}_{L^2} + \|T_{xzz}\|_{L^2}^{\frac{1}{2}})\\
	&&\hskip.25in
	+ \|w_x\|_{L^2}^{\frac{1}{2}} (\|w_x\|_{L^2}^{\frac{1}{2}} + \|w_{xz}\|_{L^2}^{\frac{1}{2}}) \|T_{z}\|^{\frac{1}{2}}_{L^2} (\|T_{z}\|^{\frac{1}{2}}_{L^2} + \|T_{xz}\|_{L^2}^{\frac{1}{2}}) \Big] \|\Dd T_x\|_{L^2}
	\leq C_0 CHY^{1/2}.
\end{eqnarray*}
 From the estimates above, we obtain
 \begin{eqnarray*}
 	|I_1 + I_2 + I_3 + I_4 +I_5 + I_6 + I_7 + I_8| \leq C(K+G+F+C_0 H)Y^{\frac{1}{2}}+ \frac{\ee_2}{2}G + \frac{1}{2\ee_2}(C_p+1)H.
 \end{eqnarray*}
 Therefore, we obtain
 \begin{eqnarray}
 &&\hskip-.68in \frac{1}{2}\frac{dY}{dt}
 + \nu(1-\frac{C}{\nu}Y^{\frac{1}{2}}) F + \ee_1(1-\frac{C}{\ee_1}Y^{\frac{1}{2}})  K + \ee_2(\frac{1}{2}-\frac{C}{\ee_2}Y^{\frac{1}{2}})G + C_0 \kappa (1-\frac{C}{\kappa}Y^{\frac{1}{2}}-\frac{C_p+1}{2\ee_2 \kappa C_0 }) H \leq 0.
 \end{eqnarray}
 Choose
 \begin{eqnarray}
 &&\hskip-.68in C_0=\frac{C_p+1}{\ee_2 \kappa}. \label{C0}
 \end{eqnarray}
 Observe that if $Y_0 < \min(\frac{\nu^2}{C^2}, \frac{\ee_1^2}{C^2}, \frac{\ee_2^2}{4C^2}, \frac{\kappa^2}{4C^2})$, there exists $t^*>0$ such that $\frac{dY}{dt} \leq 0$ on $[0,t^*]$, and hence $Y(t) \leq Y_0 $ for $t\in[0, t^*]$, and in particular, $Y(t^*)< \min(\frac{\nu^2}{C^2}, \frac{\ee_1^2}{C^2}, \frac{\ee_2^2}{4C^2}, \frac{\kappa^2}{4C^2})$. Thus we can repeat this procedure to arbitrary time $t>0$ to get $Y(t) \leq Y_0 < \min(\frac{\nu^2}{C^2}, \frac{\ee_1^2}{C^2}, \frac{\ee_2^2}{4C^2}, \frac{\kappa^2}{4C^2})$ for all time. This implies the required bound for the global in time existence of strong solution.
\end{proof}

\section{Global Well-posedness of system (\ref{EQO1-1})--(\ref{EQO1-5})}
In this section we study system (\ref{EQO1-1})--(\ref{EQO1-5}) with $\nu=0$. Similar as in section 3.1, our domain is $\mathbb{T}^2$, and the boundary conditions and initial condition are
\begin{eqnarray}
&&\hskip-.8in
u, \;v, \; w,\;  p \;\text{and} \; T \;  \text{are periodic in} \; x \; \text{and} \; z \; \text{with period 1},  \label{ABC-1} \\
&&\hskip-.8in
u, \; v \; \text{and} \; p \; \text{are even in} \; z,\;  \text{and} \; w, \; T \; \text{are odd in} \; z, \label{ABC-2} \\
&&\hskip-.8in
(u, v, T)|_{t=0}=(u_0, v_0, T_0). \label{AIC-1}
\end{eqnarray}
Using analogue argument as in section 3.1, system (\ref{EQO1-1})--(\ref{EQO1-5}) subjects to (\ref{ABC-1})--(\ref{AIC-1}) is equivalent to the following:
\begin{eqnarray}
&&\hskip-.8in
(u-\alpha^2  u_{zz})_t  + u\, u_x + w u_z +\ee_1 u - f_0 v + p_x = 0,  \label{AEQ-1}  \\
&&\hskip-.8in
(v-\alpha^2  v_{zz})_t  + u\, v_x + w v_z +\ee_1 v + f_0 u = 0,  \label{AEQ-2}  \\
&&\hskip-.8in
T_t - \kappa \Dd T   + u \, T_x + w \, T_z =  0, \label{AEQ-5}
\end{eqnarray}
with $w,p_x,p_z$ defined by (\ref{EQw})--(\ref{EQpz}), subject to the symmetry boundary conditions and initial conditions (\ref{BC-1})--(\ref{IC-1}). We also have (\ref{EQ-3}) and (\ref{EQ-4}), for which we repeat here:
\begin{eqnarray}
&&\hskip-.8in
\ee_2 w + p_z  +  T =0,    \label{AEQ-3}  \\
&&\hskip-.8in
u_x + w_z =0.   \label{AEQ-4}
\end{eqnarray}
  In this section, we are interested in system (\ref{AEQ-1})--(\ref{AEQ-5}) with (\ref{EQw})--(\ref{EQpz}), subjects to (\ref{BC-1})--(\ref{IC-1}). We will show the global regularity of strong solution to system (\ref{AEQ-1})--(\ref{AEQ-5}) with (\ref{EQw})--(\ref{EQpz}), subject to (\ref{BC-1})--(\ref{IC-1}). First, we give the definition of strong solution to system (\ref{AEQ-1})--(\ref{AEQ-5}) with (\ref{EQw})--(\ref{EQpz}), subject to (\ref{BC-1})--(\ref{IC-1}).
\begin{definition}
	Suppose that $u_0, v_0 \in H^2(\mathbb{T}^2)$ and $T_0 \in H^1(\mathbb{T}^2)$ satisfy the symmetry conditions (\ref{BC-1}) and (\ref{BC-2}), with the compatibility condition $\int_0^1 \partial_x u_{0} dz = 0$. Moreover, suppose that $\partial^3_{xxz} u_{0}, \partial^3_{xxz} v_{0} \in L^2(\mathbb{T}^2)$. Given time $\mathcal{T}>0$, we say $(u,v,T)$ is a strong solution to the system (\ref{AEQ-1})--(\ref{AEQ-5}) with (\ref{EQw})--(\ref{EQpz}), subjects to (\ref{BC-1})--(\ref{IC-1}), on the time interval $[0,\mathcal{T}]$, if \\
	
	(i) u, v and T satisfy the symmetry conditions (\ref{BC-1}) and (\ref{BC-2});\\
	
	(ii) u, v and T have the regularities
	\begin{eqnarray*}
		&&\hskip-.28in
		u\in L^\infty(0,\mathcal{T};H^2)\cap C([0,\mathcal{T}]; H^1), \;\; u_{xxz}\in L^\infty(0,\mathcal{T};L^2), \;\; 	\partial_t u\in L^\infty(0,\mathcal{T};L^2)\cap L^2(0,\mathcal{T};H^1) , \\
		&&\hskip-.28in
		v\in L^\infty(0,\mathcal{T};H^2)\cap C([0,\mathcal{T}]; H^1), \;\; v_{xxz}\in L^\infty(0,\mathcal{T};L^2), \;\; 	\partial_t v\in L^\infty(0,\mathcal{T}; H^1), \\
		&&\hskip-.28in
		T\in L^2(0,\mathcal{T};H^2)\cap L^\infty(0,\mathcal{T};H^1)\cap C([0,\mathcal{T}]; L^2), \;\; 	\partial_t T \in L^2(0,\mathcal{T};L^2);
	\end{eqnarray*}
	
	(iii) u, v and T satisfy system (\ref{EQ-1})--(\ref{EQ-5}) in the following sense:
	\begin{eqnarray*}
		&&\hskip-.68in
		\partial_t(u-\alpha^2 u_{zz})+uu_x + wu_z+\ee_1 u -f_0 v + p_x = 0  \;\;  \text{in} \; \; L^\infty(0,\mathcal{T}; L^2)\cap L^2(0,\mathcal{T};H^1);\\
		&&\hskip-.68in
		\partial_t(v-\alpha^2 v_{zz})+uv_x + wv_z+\ee_1 v +f_0 u = 0  \;\;  \text{in} \; \; L^\infty(0,\mathcal{T}; H^1);\\
		&&\hskip-.68in
		\partial_t T -\kappa \Dd T +uT_x + wT_z= 0  \;\;  \text{in} \; \; L^2(0,\mathcal{T}; L^2),
	\end{eqnarray*}
	with $w,p_x,p_z$ defined by (\ref{EQw})--(\ref{EQpz}), and fulfill the initial condition (\ref{IC-1}).
\end{definition}
We have the following result concerning the existence and uniqueness of strong solutions to system (\ref{AEQ-1})--(\ref{AEQ-5}) with (\ref{EQw})--(\ref{EQpz}), subjects to (\ref{BC-1})--(\ref{IC-1}), on $\mathbb{T}^2 \times (0,\mathcal{T})$, for any positive time $\mathcal{T}$.

\begin{theorem} \label{T-MAIN3}
	Suppose that $u_0, v_0 \in H^2(\mathbb{T}^2)$ and $T_0 \in H^1(\mathbb{T}^2)$ satisfy the symmetry conditions (\ref{BC-1}) and (\ref{BC-2}), with the compatibility condition $\int_0^1 \partial_x u_{0} dz = 0$. Moreover, suppose that $\partial^3_{xxz} u_{0}, \partial^3_{xxz} v_{0} \in L^2(\mathbb{T}^2)$. Let time $\mathcal{T}>0$.
	Then there exists a unique strong solution $(u, v, T)$ of system
	(\ref{AEQ-1})--(\ref{AEQ-5}) with (\ref{EQw})--(\ref{EQpz}), subjects to (\ref{BC-1})--(\ref{IC-1}), on the interval $[0,\mathcal{T}]$. Moreover, the unique strong solution $(u, v, T)$ depends continuously on the initial data.
\end{theorem}

\begin{proof}
 We will only do\textit{ a priori} estimates formally here, and we can use Galerkin method as in section 3.2 to prove the result rigorously. By taking the $L^2$-inner product of equation  (\ref{AEQ-1}) with $u$, equation  (\ref{AEQ-2}) with $v$,
 equation  (\ref{AEQ-3}) with $w$ and
 equation  (\ref{AEQ-5}) with $T$, in $L^2(\mathbb{T}^2)$, by integration by parts, thanks to (\ref{BC-1}), we get

 \begin{eqnarray*}
 	&&\hskip-.48in \frac{1}{2} \frac{d}{dt} \Big(\|u\|_{L^2}^{2}+\|v\|_{L^2}^{2}+\|T\|_{L^2}^{2}+ \alpha^2  \, \| u_{z}\|_{L^2}^2
 		+ \alpha^2  \, \| v_{z}\|_{L^2}^2
 		\Big) + \ee_1 \|u\|_{L^2}^{2} + \ee_1 \|v\|_{L^2}^{2}
 	+ \ee_2  \|w\|_{L^2}^{2} + \kappa  \, \| \nabla T\|_{L^2}^2 \\
 	&&\hskip-.45in
 	=  -\int_{\mathbb{T}^2} \left[ u u_{x}+w u_{z}
 	\right]\; u \; dxdz -\int_{\mathbb{T}^2} \left[ u v_{x}+w v_{z}
 	\right]\; v \; dxdz \\
 	&&\hskip-.25in -\int_{\mathbb{T}^2} \left[ u p_{x}+w p_{z}+w T
 	\right]\; dxdz -\int_{\mathbb{T}^2} \left[ u T_{x}+w T_{z}
 	\right]\; T \; dxdz.
 \end{eqnarray*}
 By integration by parts, thanks to (\ref{BC-1}), (\ref{wpBC-1}) and (\ref{AEQ-4}), we have
 \begin{eqnarray*}
 	&&\hskip-.45in   -\int_{\mathbb{T}^2} \left[ u u_{x}+w u_{z}
 	\right]\; u \; dxdz -\int_{\mathbb{T}^2} \left[ u v_{x}+w v_{z}
 	\right]\; v \; dxdz  \\
 	&&\hskip-.45in  -\int_{\mathbb{T}^2} \left[ u p_{x}+w p_{z}
 	\right]\; dxdz-\int_{\mathbb{T}^2} \left[ u T_{x}+w T_{z}
 	\right]\; T \; dxdz =0.
 \end{eqnarray*}
 By Cauchy--Schwarz inequality and Young's inequality, we have
 \begin{eqnarray*}
 	&&\hskip-.45in   -\int_{\mathbb{T}^2} wT dxdz \leq \|w\|_{L^2} \|T\|_{L^2} \leq \frac{\ee_2}{2} \|w\|_{L^2}^2 + C \|T\|_{L^2}^2.
 \end{eqnarray*}
 As a result of the above, we have
 \begin{eqnarray*}
 	&&\hskip-.48in
 	\frac{d}{dt} \Big(\|u\|_{L^2}^{2}+\|v\|_{L^2}^{2}+\|T\|_{L^2}^{2}+ \alpha^2  \, \| u_{z}\|_{L^2}^2
 		+ \alpha^2  \, \| v_{z}\|_{L^2}^2
 		\Big) + 2\ee_1 \|u\|_{L^2}^{2} + 2\ee_1 \|v\|_{L^2}^{2}
 	+ \ee_2  \|w\|_{L^2}^{2} + 2\kappa  \, \| \nabla T\|_{L^2}^2  \\
 	&&\hskip-.45in \leq C \|T\|_{L^2}^2 \leq C (\|u\|_{L^2}^{2}+\|v\|_{L^2}^{2}+\|T\|_{L^2}^{2}+ \alpha^2  \, \| u_{z}\|_{L^2}^2
 	+ \alpha^2  \, \| v_{z}\|_{L^2}^2).
 \end{eqnarray*}
 Thanks to Gronwall inequality, we obtain
 \begin{eqnarray*}
 	&&\hskip-.68in
 	(\|u\|_{L^2}^{2}+\|v\|_{L^2}^{2}+\|T\|_{L^2}^{2}+ \alpha^2  \, \| u_{z}\|_{L^2}^2
 	+ \alpha^2  \, \| v_{z}\|_{L^2}^2
 	)(t)
 	\\
 	&&\hskip-.48in
 	+
 	\int_0^t \left( 2\ee_1 \|u(s)\|_{L^2}^{2} + 2\ee_1 \|v(s)\|_{L^2}^{2} + \ee_2  \|w(s)\|_{L^2}^{2} + 2\kappa  \, \|\nabla T(s)\|_{L^2}^2 \right) \; ds  \nonumber \\
 	&&\hskip-.68in
 	\leq (\|u_{0}\|_{L^2}^{2}+\|v_{0}\|_{L^2}^{2}+\|T_{0}\|_{L^2}^{2}+ \alpha^2  \, \| \partial_z u_{0}\|_{L^2}^2
 	+ \alpha^2  \, \| \partial_{z} v_{0}\|_{L^2}^2) e^{Ct}.
 \end{eqnarray*}
 Consequently, we have
 \begin{eqnarray}
 u, v, u_z, v_z \in L^\infty(0,\mathcal{T};L^2), \nonumber \\
 w \in L^2(0,\mathcal{T};L^2), \nonumber \\
 T \in L^\infty(0,\mathcal{T};L^2)\cap L^2(0,\mathcal{T};H^1),  \label{L-2}
 \end{eqnarray}
for arbitrary $\mathcal{T}>0$. By taking the $L^2$-inner product of equation  (\ref{AEQ-1}) with $-u_{zz}$ and
equation  (\ref{AEQ-3}) with $-w_{zz}$  in $L^2(\mathbb{T}^2)$, by integration by parts, thanks to (\ref{BC-1}) and (\ref{wpBC-1}), we get
 \begin{eqnarray*}
 	&&\hskip-.28in \frac{1}{2} \frac{d (\|u_{z}\|_{L^2}^{2}+ \alpha^2  \, \| u_{zz}\|_{L^2}^2)}{dt} + \ee_1 \|u_z\|_{L^2}^{2}  +
 	\ee_2  \|w_{z}\|_{L^2}^{2} \\
 	&&\hskip-.28in
 	= \int_{\mathbb{T}^2} \left[ u u_{x}+w u_{z} - f_0 v
 	\right] u_{zz} + \left[ u_{zz} p_{x}+w_{zz} p_{z}+w_{zz} T
 	\right]\; dxdz.
 \end{eqnarray*}
 By integration by parts, thanks to (\ref{BC-1}), (\ref{wpBC-1}) and (\ref{AEQ-4}), we have
 \begin{eqnarray*}
 	&&\hskip-.45in   \int_{\mathbb{T}^2} \left[ u u_{x}+w u_{z}
 	\right] u_{zz} dxdz + \int_{\mathbb{T}^2} \left[ u_{zz} p_{x}+w_{zz} p_{z}
 	\right]\; dxdz \\
 	&&\hskip-.45in = -\int_{\mathbb{T}^2} [u_z u_x + u u_{xz} + w_z u_z + w u_{zz}]u_z dxdz - \int_{\mathbb{T}^2} p(u_x + w_z)_{zz} dxdz \\
 	&&\hskip-.45in = -\frac{1}{2}\int_{\mathbb{T}^2}u_z^2(u_x+w_z)dxdz - \int_{\mathbb{T}^2} p(u_x + w_z)_{zz} dxdz = 0.
 \end{eqnarray*}
 By integration by parts, using Cauchy--Schwarz inequality and Young's inequality, thanks to (\ref{BC-1}) and (\ref{wpBC-1}), we have
 \begin{eqnarray*}
 	&&\hskip-.45in   \int_{\mathbb{T}^2} \left(-f_0 v u_{zz} + w_{zz} T\right) dxdz = \int_{\mathbb{T}^2} \left(f_0 v_z u_{z} - w_{z} T_z\right) dxdz \leq f_0\|v_z\|_{L^2} \|u_z\|_{L^2} + \|w_z\|_{L^2} \|T_z\|_{L^2} \\
 	&&\hskip 1.1 in \leq \frac{\ee_2}{2} \|w_z\|_{L^2}^2 + C \|T_z\|_{L^2}^2 + C\|v_z\|_{L^2} (1+\|u_{z}\|^2_{L^2}+ \alpha^2   \| u_{zz}\|_{L^2}^2).
 \end{eqnarray*}
 As a result of the above, we have
 \begin{eqnarray*}
 	&&\hskip-.28in \frac{d (1+ \|u_{z}\|_{L^2}^{2}+ \alpha^2   \| u_{zz}\|_{L^2}^2)}{dt} + \ee_1 \|u_z\|_{L^2}^{2} +
 	\ee_2  \|w_{z}\|_{L^2}^{2}
 	\leq C \|v_{z}\|_{L^2}(1+\|u_{z}\|^2_{L^2}+ \alpha^2   \| u_{zz}\|_{L^2}^2)  + C \|T_{z}\|^2_{L^2}.
 \end{eqnarray*}
 Thanks to Gronwall inequality, we obtain
 \begin{eqnarray*}
 	&&\hskip-.68in
 	\|u_{z}(t)\|_{L^2}^{2}+ \alpha^2  \, \| u_{zz}(t)\|_{L^2}^2 +
 	\int_0^t \left(2\ee_1 \|u_z(s)\|_{L^2}^{2} + \ee_2 \|w_{z}(s)\|_{L^2}^{2}\right)  \; ds  \nonumber \\
 	&&\hskip-.68in
 	\leq C\left(1+ \int_0^t \|T_{z}(s)\|^2_{L^2}ds + \|\partial_z u_{0}\|_{L^2}^{2}+ \alpha^2  \, \| \partial^2_{zz} u_{0}\|_{L^2}^2\right) \exp\left(C \int_0^t \|v_z(s)\|_{L^2} ds\right).
 \end{eqnarray*}
 By virtue of (\ref{L-2}) and the above, we have
 \begin{eqnarray}
 u_{zz} \in L^\infty(0,\mathcal{T};L^2), \; \;  w_{z}=-u_{x} \in L^2(0,\mathcal{T};L^2), 	\label{L-2z}
 \end{eqnarray}
for arbitrary $\mathcal{T}>0$. By taking the $L^2$-inner product of equation  (\ref{AEQ-1}) with $-u_{xx}$ and
equation  (\ref{AEQ-3}) with $-w_{xx}$, in $L^2(\mathbb{T}^2)$, by integration by parts, thanks to (\ref{BC-1}) and (\ref{wpBC-1}), we get
 \begin{eqnarray*}
 	&&\hskip-.28in \frac{1}{2} \frac{d (\|u_{x}\|_{L^2}^{2}+ \alpha^2 \, \|u_{xz}\|_{L^2}^2 )}{dt}
 	+ \ee_1 \|u_x\|_{L^2}^{2}  + \ee_2  \|w_{x}\|_{L^2}^{2} \\
 	&&\hskip-.28in
 	= \int_{\mathbb{T}^2} \left[ u u_{x}+w u_{z} - f_0 v
 	\right] u_{xx} + \left[ u_{xx} p_{x}+w_{xx} p_{z}+w_{xx} T
 	\right]\; dxdz.
 \end{eqnarray*}
 By integration by parts, thanks to (\ref{BC-1}), (\ref{wpBC-1}) and (\ref{AEQ-4}), we have
 \begin{eqnarray*}
 	&&\hskip-.45in  \int_{\mathbb{T}^2} \left[ u_{xx} p_{x}+w_{xx} p_{z}
 	\right]\; dxdz =0.
 \end{eqnarray*}
 By integration by parts, using Cauchy--Schwarz inequality and Young's inequality, thanks to (\ref{BC-1}), (\ref{wpBC-1}) and (\ref{AEQ-4}), we have
 \begin{eqnarray*}
 	&&\hskip-.45in   -\int_{\mathbb{T}^2} f_0vu_{xx} dxdz = \int_{\mathbb{T}^2} f_0 v w_{xz} dxdz = -\int_{\mathbb{T}^2} f_0 v_z w_x dxdz \leq C \|v_{z}\|_{L^2}^2 + \frac{\ee_2}{6} \|w_{x}\|_{L^2}^{2}, \\
 	&&\hskip .05in \int_{\mathbb{T}^2} T w_{xx}\; dxdz = -\int_{\mathbb{T}^2} T_{x} w_{x}\; dxdz \leq C \|T_{x}\|_{L^2}^2 + \frac{\ee_2}{6} \|w_{x}\|_{L^2}^{2}.
 \end{eqnarray*}
 By integration by parts, using Young's inequality and Lemma 1, thanks to (\ref{BC-1}), (\ref{wpBC-1}) and (\ref{AEQ-4}), we have
 \begin{eqnarray*}
 	&&\hskip-.065in   \int_{\mathbb{T}^2} \left[ u u_{x}+w u_{z}
 	\right]\; u_{xx} \; dxdz = - \int_{\mathbb{T}^2} \left[ (u_{x})^3 +w_{x} u_{z} u_{x}
 	\right] \; dxdz \\
 	&&\hskip-.065in
 	=  - \int_{\mathbb{T}^2} \left[ -w_z (u_x)^2 +w_{x} u_{z} u_{x}
 	\right] \; dxdz = - \int_{\mathbb{T}^2} \left[ 2 w u_{x} u_{xz} + w_{x} u_{z} u_{x}
 	\right] \; dxdz   \\
 	&&\hskip-.065in  \leq    C \big[  \|w\|_{L^2}^{\frac{1}{2}} (\|w\|_{L^2}^{\frac{1}{2}}+\|w_{x}\|_{L^2}^{\frac{1}{2}}) \|u_{x}\|_{L^2}^{\frac{1}{2}} (\|u_{x}\|_{L^2}^{\frac{1}{2}}+\|u_{xz}\|_{L^2}^{\frac{1}{2}})\|u_{xz}\|_{L^2} \\
 	&&\hskip.15in
 	+ \|w_{x} \|_{L^2}  \|u_{z}\|_{L^2}^{\frac{1}{2}} (\|u_{z}\|_{L^2}^{\frac{1}{2}}+\|u_{xz}\|_{L^2}^{\frac{1}{2}}) \|u_{x}\|_{L^2}^{\frac{1}{2}} (\|u_{x}\|_{L^2}^{\frac{1}{2}}+\|u_{xz}\|_{L^2}^{\frac{1}{2}})
 	\big]   \\
 	&&\hskip-.065in
 	\leq    C \left( 1+ \|w\|_{L^2}^2 + \|u_{x}\|_{L^2}^2
 	+  \|u_{z}\|_{L^2}^2  \right) \left(1+ \|u_{x}\|_{L^2}^2+ \alpha^2\|u_{xz}\|_{L^2}^2 \right)
 	+ \frac{\ee_2}{6} \|w_{x}\|_{L^2}^{2}.
 \end{eqnarray*}
 From the estimates above,  we have

 \begin{eqnarray*}
 	&&\hskip-.68in
 	\frac{d (1+\|u_{x}\|_{L^2}^{2}+ \alpha^2 \, \|u_{xz}\|_{L^2}^2 )}{dt}
 	+ \ee_1 \|u_x\|_{L^2}^{2}  + \ee_2  \|w_{x}\|_{L^2}^{2} \\
 	&&\hskip-.65in \leq    C \left( 1+ \|w\|_{L^2}^2 + \|u_{x}\|_{L^2}^2
 	+  \|u_{z}\|_{L^2}^2 \right) \left(1+ \|u_{x}\|_{L^2}^2+ \alpha^2\|u_{xz}\|_{L^2}^2 \right) + C(\|v_{z}\|_{L^2}^2 + \|T_{x}\|_{L^2}^2).
 \end{eqnarray*}
 By Gronwall inequality, we obtain
 \begin{eqnarray*}
 	&&\hskip-.68in \|u_{x}(t)\|_{L^2}^{2}+ \alpha^2 \, \|u_{xz}(t)\|_{L^2}^2 +
 	\int_0^t  \left( 2\ee_1 \|u_x(s)\|_{L^2}^{2}  + \ee_2  \|w_{x}(s)\|_{L^2}^{2}\right)  \; ds   \nonumber  \\
 	&&\hskip-.68in
 	\leq C\left(1+\int_0^t \left(\|v_{z}(s)\|_{L^2}^2 + \|T_{x}(s)\|_{L^2}^2\; \right) ds + \|\partial_{x} u_{0}\|_{L^2}^{2}+ \alpha^2 \, \|\partial^2_{xz} u_{0}\|_{L^2}^2 \right) \nonumber  \\
 	&&\hskip-.38in
 	\times \exp\left(C \int_0^t \left(1+ \|w(s)\|_{L^2}^2 + \|u_{x}(s)\|_{L^2}^2
 	+  \|u_{z}(s)\|_{L^2}^2  \;  \right) ds\right).
 \end{eqnarray*}
 By virtue of (\ref{L-2}), (\ref{L-2z}) and the above, we have
 \begin{eqnarray}
 u, u_{z} \in L^{\infty}(0,\mathcal{T};H^1), \;\; w \in L^2(0,\mathcal{T};H^1), 	\label{LL-2x}
 \end{eqnarray}
for arbitrary $\mathcal{T}>0$. By virtue of (\ref{LL-2x}), (\ref{ESw}) and (\ref{ESwx}), we have
 \begin{eqnarray}
 w\in L^{\infty}(0,\mathcal{T};L^2)\cap L^2(0,\mathcal{T};H^1), 	\label{LL-2x2}
 \end{eqnarray}
for arbitrary $\mathcal{T}>0$. By taking the $L^2$-inner product of equation  (\ref{AEQ-2}) with $-\Dd v$ in $L^2(\mathbb{T}^2)$, and by integration by parts, thanks to (\ref{BC-1}), we have
 \begin{eqnarray*}
 	&&\hskip-.68in \frac{1}{2} \frac{d (\|\nabla v\|_{L^2}^{2}+ \alpha^2  \, \| \nabla v_{z}\|_{L^2}^2)}{dt} + \ee_1  \|\nabla v\|_{L^2}^{2} \\
 	&&\hskip-.68in
 	=   \int_{\mathbb{T}^2} \Big[(u v_{x}+w v_{z}) \; (v_{xx} + v_{zz}
 	)+ f_0 u \Dd v \Big] dxdz  =: \rom{1}+\rom{2}+\rom{3}+\rom{4}+\rom{5}. \\
 \end{eqnarray*}
 By integration by parts, using Cauchy--Schwarz inequality and Lemma 1, thanks to (\ref{BC-1}), (\ref{wpBC-1}) and (\ref{AEQ-4}), we have
 \begin{eqnarray*}
 	&&\hskip-.68in
 	|\rom{1}| = |\int_{\mathbb{T}^2} u v_{x} v_{xx} \; dxdz|= |\int_{\mathbb{T}^2} \frac{1}{2} u_x v_{x}^2 \; dxdz| = |\int_{\mathbb{T}^2} \frac{1}{2} w_z v_{x}^2 \; dxdz| = |\int_{\mathbb{T}^2} w v_{x} v_{xz} \; dxdz| \\
 	&&\hskip-.48in
 	\leq C \|w\|_{L^2}^{\frac{1}{2}} (\|w\|_{L^2}^{\frac{1}{2}}+\|w_{x}\|_{L^2}^{\frac{1}{2}}) \|v_{x}\|_{L^2}^{\frac{1}{2}}  (\|v_{x}\|_{L^2}^{\frac{1}{2}}+\|v_{xz}\|_{L^2}^{\frac{1}{2}})\|v_{xz}\|_{L^2}, \\
 	&&\hskip-.68in
 	|\rom{2}| = | \int_{\mathbb{T}^2} u v_x v_{zz} \; dxdz| \leq C \|u\|_{L^2}^{\frac{1}{2}} (\|u\|_{L^2}^{\frac{1}{2}}+\|u_{x}\|_{L^2}^{\frac{1}{2}}) \|v_{x}\|_{L^2}^{\frac{1}{2}}  (\|v_{x}\|_{L^2}^{\frac{1}{2}}+\|v_{xz}\|_{L^2}^{\frac{1}{2}})\|v_{zz}\|_{L^2}, \\
 	&&\hskip-.68in
 	|\rom{3}| = | \int_{\mathbb{T}^2} w v_z v_{xx} \; dxdz| =  | \int_{\mathbb{T}^2} v_{x}(w_{x}v_{z} + wv_{xz}) \; dxdz| \\
 	&&\hskip-.38in
 	\leq C\|w_{x}\|_{L^2} \|v_{x}\|_{L^2}^{\frac{1}{2}} (\|v_{x}\|_{L^2}^{\frac{1}{2}}+\|v_{xz}\|_{L^2}^{\frac{1}{2}} )\|v_{z}\|_{L^2}^{\frac{1}{2}} (\|v_{z}\|_{L^2}^{\frac{1}{2}}+ \|v_{xz}\|_{L^2}^{\frac{1}{2}}) \\
 	&&\hskip-.18in
 	+ C \|w\|_{L^2}^{\frac{1}{2}} (\|w\|_{L^2}^{\frac{1}{2}}+\|w_{x}\|_{L^2}^{\frac{1}{2}}) \|v_{x}\|_{L^2}^{\frac{1}{2}} (\|v_{x}\|_{L^2}^{\frac{1}{2}}+ \|v_{xz}\|_{L^2}^{\frac{1}{2}})\|v_{xz}\|_{L^2}, \\
 	&&\hskip-.68in
 	|\rom{4}|  = | \int_{\mathbb{T}^2} w v_z v_{zz} \; dxdz| \leq C \|w\|_{L^2}^{\frac{1}{2}} (\|w\|_{L^2}^{\frac{1}{2}}+\|w_{x}\|_{L^2}^{\frac{1}{2}}) \|v_{z}\|_{L^2}^{\frac{1}{2}} (\|v_{z}\|_{L^2}^{\frac{1}{2}}+ \|v_{zz}\|_{L^2}^{\frac{1}{2}})\|v_{zz}\|_{L^2}, \\
 	&&\hskip-.68in
 	|\rom{5}|  = |\int_{\mathbb{T}^2} f_0 u \Dd v \; dxdz| = |\int_{\mathbb{T}^2} f_0 \nabla u \nabla v \; dxdz| \leq C \|\nabla u\|_{L^2} \|\nabla v\|_{L^2}.
 \end{eqnarray*}
 As a result of the above and by Young's inequality, we conclude
 \begin{eqnarray*}
 	&&\hskip-.68in
 	\frac{d (1+\|\nabla v\|_{L^2}^{2}+ \alpha^2  \, \| \nabla v_{z}\|_{L^2}^2)}{dt} + 2\ee_1  \|\nabla v\|_{L^2}^{2}  \\
 	&&\hskip-.68in
 	\leq C \left(1+ \|u\|_{L^2}^2 + \|w\|_{L^2}^2 + \|\nabla u\|_{L^2}^2 + \|w_{x}\|_{L^2}^2 \right) \left(1+\|\nabla v\|_{L^2}^{2}+ \alpha^2  \, \| \nabla v_{z}\|_{L^2}^2 \right).
 \end{eqnarray*}
 Thanks to Gronwall inequality, we obtain
 \begin{eqnarray*}
 	&&\hskip-.68in
 	\|\nabla v(t)\|_{L^2}^{2}+ \alpha^2  \, \| \nabla v_{z}(t)\|_{L^2}^2  + \int_0^t 2\ee_1  \|\nabla v(s)\|_{L^2}^{2} ds \\
 	&&\hskip-.68in
 	\leq \left(1+\|\nabla v_0\|_{L^2}^{2}+ \alpha^2  \, \| \nabla \partial_{z} v_{0}\|_{L^2}^2\right) \exp\left(C \int_0^t \left(1+\|u\|_{L^2}^2 + \|w\|_{L^2}^2 + \|\nabla u\|_{L^2}^2 + \|w_{x}\|_{L^2}^2 \right)(s)\; ds\right).
 \end{eqnarray*}
 By virtue of (\ref{L-2}), (\ref{L-2z}), (\ref{LL-2x}) and the above, we have
 \begin{eqnarray}
 v, v_z \in L^{\infty}(0,\mathcal{T};H^1), \label{vH-1}
 \end{eqnarray}
for arbitrary $\mathcal{T}>0$. By taking the $L^2$-inner product of equation  (\ref{AEQ-5}) with $-\Dd T$ in $L^2(\mathbb{T}^2)$, and by integration by parts, thanks to (\ref{BC-1}), we have
 \begin{eqnarray*}
 	&&\hskip-.68in \frac{1}{2} \frac{d \|\nabla T\|_{L^2}^{2}}{dt} +  \kappa  \, \| \Dd T\|_{L^2}^2  =\int_{\mathbb{T}^2}(u T_{x}+w T_{z}) \Dd T.
 \end{eqnarray*}
 By Lemma 1 and Young's inequality, thanks to (\ref{AEQ-4}), we have
 \begin{eqnarray*}
 	&&\hskip-.68in
 	\int_{\mathbb{T}^2}(u T_{x}+w T_{z}) \Dd T\\
 	&&\hskip-.68in
 	\leq C \Big( \| u\|_{L^2}^{\frac{1}{2}} (\| u\|_{L^2}^{\frac{1}{2}} + \| u_{z}\|_{L^2}^{\frac{1}{2}}) \| T_{x}\|_{L^2}^{\frac{1}{2}}  (\| T_{x}\|_{L^2}^{\frac{1}{2}} + \| T_{xx}\|_{L^2}^{\frac{1}{2}})\| \Dd T\|_{L^2} \\
 	&&\hskip-.48in + \| w\|_{L^2}^{\frac{1}{2}} (\|w\|_{L^2}^{\frac{1}{2}} + \| w_{z}\|_{L^2}^{\frac{1}{2}}) \| T_{z}\|_{L^2}^{\frac{1}{2}}  (\| T_{z}\|_{L^2}^{\frac{1}{2}} + \| T_{xz}\|_{L^2}^{\frac{1}{2}})\| \Dd T\|_{L^2}\Big)	\\
 	&&\hskip-.68in
 	\leq \frac{\kappa}{2}\| \Dd T\|_{L^2}^2 + C \left(1+\|u\|_{L^2}^4 + \|u_{z}\|_{L^2}^4 + \|w\|_{L^2}^4 + \|w_{z}\|_{L^2}^4  \right) \|\nabla T\|_{L^2}^{2},\\
 	&&\hskip-.68in
 	=\frac{\kappa}{2}\| \Dd T\|_{L^2}^2 + C \left(1+\|u\|_{L^2}^4 + \|u_{z}\|_{L^2}^4 + \|w\|_{L^2}^4 + \|u_{x}\|_{L^2}^4  \right) \|\nabla T\|_{L^2}^{2}.
 \end{eqnarray*}
 As a result of the above we conclude
 \begin{eqnarray*}
 	&&\hskip-.68in
 	\frac{d \|\nabla T\|_{L^2}^{2}}{dt} +  \kappa  \, \| \Dd T\|_{L^2}^2
 	\leq C \left(1+\|u\|_{L^2}^4 + \|u_{z}\|_{L^2}^4 + \|w\|_{L^2}^4 + \|u_{x}\|_{L^2}^4  \right) \|\nabla T\|_{L^2}^{2}.
 \end{eqnarray*}
 Thanks to Gronwall inequality, we obtain
 \begin{eqnarray*}
 	&&\hskip-.68in
 	\|\nabla T(t)\|_{L^2}^{2}+ \kappa   \,\int_0^t \| \Dd T(s)\|_{L^2}^2 \; ds \\
 	&&\hskip-.68in
 	\leq C\|\nabla T_0\|_{L^2}^{2} \exp\left(C \int_0^t \left(1+ \|u\|_{L^2}^4 + \|u_{z}\|_{L^2}^4 + \|w\|_{L^2}^4 + \|u_{x}\|_{L^2}^4 \right)(s) ds\right).
 \end{eqnarray*}
 By virtue of (\ref{L-2}), (\ref{L-2z}), (\ref{LL-2x}), (\ref{LL-2x2}) and the above,
 we have
 \begin{eqnarray}
 T \in L^{\infty}(0,\mathcal{T};H^1)\cap L^{2}(0,\mathcal{T};H^2) \label{TH-2}.
 \end{eqnarray}
By taking the $L^2$-inner product of equation  (\ref{AEQ-1}) with $u_{xxxx}$, equation (\ref{AEQ-2}) with $v_{xxxx}$, and equation  (\ref{AEQ-3}) with $w_{xxxx}$ in $L^2(\mathbb{T}^2)$, and by integration by parts, thanks to (\ref{BC-1}) and (\ref{wpBC-1}), we get
 \begin{eqnarray*}
 	&&\hskip-.68in \frac{1}{2} \frac{d (\|u_{xx}\|_{L^2}^{2}+ \|v_{xx}\|_{L^2}^{2}+ \alpha^2  \, \| u_{xxz}\|_{L^2}^2 +  \alpha^2  \, \| v_{xxz}\|_{L^2}^2)}{dt} + \ee_1  \|u_{xx}\|_{L^2}^{2} + \ee_1  \|v_{xx}\|_{L^2}^{2}
 	+ \ee_2  \|w_{xx}\|_{L^2}^{2} \\
 	&&\hskip-.65in
 	=  -\int_{\mathbb{T}^2} \left[ u u_{x}+w u_{z} - f_0 v
 	\right]\; u_{xxxx} \; dxdz -\int_{\mathbb{T}^2} \left[ u v_{x}+w v_{z} + f_0 u
 	\right]\; v_{xxxx} \; dxdz \\
 	&&\hskip-.25in -\int_{\mathbb{T}^2} \left[ u_{xxxx} p_{x}+w_{xxxx} p_{z}+w_{xxxx} T
 	\right]\; dxdz.
 \end{eqnarray*}
 By integration by parts, thanks to (\ref{BC-1}), (\ref{wpBC-1}) and (\ref{AEQ-4}), , we have
 \begin{eqnarray*}
 	&&\hskip-.45in   \int_{\mathbb{T}^2} \left(-f_0 v u_{xxxx} + f_0 u v_{xxxx}\right)  dxdz -\int_{\mathbb{T}^2} \left(u_{xxxx} p_{x}+w_{xxxx} p_{z} \right) dxdz=0.
 \end{eqnarray*}
 By integration by parts, using Young's inequality and Lemma 1, thanks to (\ref{BC-1}), (\ref{wpBC-1}) and (\ref{AEQ-4}), we have
 \begin{eqnarray*}
 	&&\hskip-.68in
 	\left|\int_{\mathbb{T}^2} \left[ u u_{x}+w u_{z}
 	\right]\; u_{xxxx} \; dxdz\right|=
 	\left|  \int_{\mathbb{T}^2} \left[ 5w_{} u_{xxz}-\frac{1}{2}w_{xx} u_{z} - 2w_{x} u_{xz}
 	\right]\; u_{xx} \; dxdz \right|  \\
 	&&\hskip-.65in
 	\leq    C \Big[  \|w\|_{L^2}^{\frac{1}{2}} (\|w\|_{L^2}^{\frac{1}{2}} + \|w_{x}\|_{L^2}^{\frac{1}{2}}) \|u_{xx}\|_{L^2}^{\frac{1}{2}} (\|u_{xx}\|_{L^2}^{\frac{1}{2}} + \|u_{xxz}\|_{L^2}^{\frac{1}{2}})\|u_{xxz}\|_{L^2}\\
 	&&\hskip-.45in
 	+ \|w_{xx} \|_{L^2}   \|u_{z}\|_{L^2}^{\frac{1}{2}} ( \|u_{z}\|_{L^2}^{\frac{1}{2}}+ \|u_{xz}\|_{L^2}^{\frac{1}{2}}) \|u_{xx}\|_{L^2}^{\frac{1}{2}} (\|u_{xx}\|_{L^2}^{\frac{1}{2}}+ \|u_{xxz}\|_{L^2}^{\frac{1}{2}})  \\
 	&&\hskip-.45in
 	+  \|w_{x} \|_{L^2}  \|u_{xz}\|_{L^2}^{\frac{1}{2}}(\|u_{xz}\|_{L^2}^{\frac{1}{2}} + \|u_{xxz}\|_{L^2}^{\frac{1}{2}} ) \|u_{xx}\|_{L^2}^{\frac{1}{2}} (\|u_{xx}\|_{L^2}^{\frac{1}{2}} +  \|u_{xxz}\|_{L^2}^{\frac{1}{2}}) \Big] \\
 	&&\hskip-.65in
 	\leq \frac{\ee_2}{6} \|w_{xx}\|_{L^2}^2 +  C\left(1+  \|w_{}\|^2_{L^2} + \|w_{x}\|^2_{L^2} + \|u_{z}\|_{L^2}^2 + \|u_{xz}\|_{L^2}^2  \right)\left(1+ \|u_{xx}\|_{L^2}^{2}  + \alpha^2  \, \| u_{xxz}\|_{L^2}^2 \right),
 \end{eqnarray*}
 and
 \begin{eqnarray*}
 	&&\hskip-.38in
 	|\int_{\mathbb{T}^2} \left[ u v_{x}+w v_{z}
 	\right]\; v_{xxxx} dxdz | =
 	|\int_{\mathbb{T}^2} \left[ u_{xx} v_{x} + w_{xx} v_{z} - 4w v_{xxz} + 2w_{x} v_{xz} \right]\; v_{xx}  dxdz| \\
 	&&\hskip-.38in
 	\leq |\int_{\mathbb{T}^2} \left[ w_{x} v_{xz} + w_{xx} v_{z} - 4w v_{xxz} + 2w_{x} v_{xz} \right]\; v_{xx}  dxdz|  +  |\int_{\mathbb{T}^2} w_x v_x v_{xxz}  \; dxdz| \\
 	&&\hskip-.38in
 	\leq C \Big[ \|w_{x}\|_{L^2}\|v_{xx}\|_{L^2}^{\frac{1}{2}} (\|v_{xx}\|_{L^2}^{\frac{1}{2}} + \|v_{xxz}\|_{L^2}^{\frac{1}{2}}) \|v_{xz}\|^{\frac{1}{2}}_{L^2} (\|v_{xz}\|^{\frac{1}{2}}_{L^2} + \|v_{xxz}\|_{L^2}^{\frac{1}{2}})\\
 	&&\hskip-.18in
 	+ \|w_{xx}\|_{L^2}\|v_{z}\|_{L^2}^{\frac{1}{2}} (\|v_{z}\|_{L^2}^{\frac{1}{2}} + \|v_{xz}\|_{L^2}^{\frac{1}{2}}) \|v_{xx}\|_{L^2}^{\frac{1}{2}} ( \|v_{xx}\|_{L^2}^{\frac{1}{2}} + \|v_{xxz}\|_{L^2}^{\frac{1}{2}} )  \\
 	&&\hskip-.18in
 	+ \|w\|_{L^2}^{\frac{1}{2}}( \|w\|_{L^2}^{\frac{1}{2}} + \|w_x\|_{L^2}^{\frac{1}{2}}) \|v_{xx}\|_{L^2}^{\frac{1}{2}} (\|v_{xx}\|_{L^2}^{\frac{1}{2}} + \|v_{xxz}\|_{L^2}^{\frac{1}{2}}) \|v_{xxz}\|_{L^2}\\
 	&&\hskip-.18in
 	+\|w_{x}\|_{L^2}\|v_{xz}\|_{L^2}^{\frac{1}{2}} (\|v_{xz}\|_{L^2}^{\frac{1}{2}} + \|v_{xxz}\|_{L^2}^{\frac{1}{2}} )\|v_{xx}\|_{L^2}^{\frac{1}{2}} ( \|v_{xx}\|_{L^2}^{\frac{1}{2}} + \|v_{xxz}\|_{L^2}^{\frac{1}{2}}) \\
 	&&\hskip-.18in
 	+ \|w_x\|_{L^2}^{\frac{1}{2}} (\|w_x\|_{L^2}^{\frac{1}{2}} + \|w_{xz}\|_{L^2}^{\frac{1}{2}})  \|v_x\|_{L^2}^{\frac{1}{2}} (\|v_x\|_{L^2}^{\frac{1}{2}} + \|v_{xx}\|_{L^2}^{\frac{1}{2}}) \|v_{xxz}\|_{L^2} \Big] \\
 	&&\hskip-.38in
 	\leq \frac{\ee_2}{6} \|w_{xx}\|_{L^2}^2 +  C\left(1+  \|w_{}\|^2_{L^2} + \|w_{x}\|^2_{L^2} + \|v_{x}\|_{L^2}^2 +  \|v_{z}\|_{L^2}^2 + \|v_{xz}\|_{L^2}^2  \right)\left(1+ \|v_{xx}\|_{L^2}^{2}  + \alpha^2  \, \| v_{xxz}\|_{L^2}^2 \right).
 \end{eqnarray*}
 By integration by parts, using Cauchy-Schwarz inequality and Young's inequality, thanks to  (\ref{BC-1}) and (\ref{wpBC-1}), we have
 \begin{eqnarray*}
 	&&\hskip-.68in
 	|\int_{\mathbb{T}^2} T w_{xxxx} dxdz | = |\int_{\mathbb{T}^2} T_{xx} w_{xx} dxdz | \leq \|T_{xx}\|_{L^2} \|w_{xx}\|_{L^2} \leq \frac{\ee_2}{6} \|w_{xx}\|_{L^2}^2 +  C\|T_{xx}\|_{L^2}^2.
 \end{eqnarray*}
 As a result of the above, we conclude
 \begin{eqnarray*}
 	&&\hskip-.68in
 	\frac{d (1+\|u_{xx}\|_{L^2}^{2} + \|v_{xx}\|_{L^2}^{2}  + \alpha^2  \, \| u_{xxz}\|_{L^2}^2 + \alpha^2  \, \| v_{xxz}\|_{L^2}^2)}{dt} + 2\ee_1  \|u_{xx}\|_{L^2}^{2} + 2\ee_1  \|v_{xx}\|_{L^2}^{2}
 	+ \ee_2  \|w_{xx}\|_{L^2}^{2}  \\
 	&&\hskip-.65in
 	\leq   C \left( 1+  \|w_{}\|^2_{L^2} + \|w_{x}\|^2_{L^2} + \|u_{z}\|_{L^2}^2 + \|u_{xz}\|_{L^2}^2 + \|v_{x}\|_{L^2}^2 + \|v_{z}\|_{L^2}^2 + \|v_{xz}\|_{L^2}^2  + \|T_{xx}\|_{L^2}^2 \right) \\
 	&&\hskip.25in
 	\times \left(1+\|u_{xx}\|_{L^2}^{2} + \|v_{xx}\|_{L^2}^{2}  + \alpha^2  \, \| u_{xxz}\|_{L^2}^2 + \alpha^2  \, \| v_{xxz}\|_{L^2}^2
 	\right).
 \end{eqnarray*}
 By Gronwall inequality, we obtain
 \begin{eqnarray*}
 	&&\hskip-.18in\|u_{xx}(t)\|_{L^2}^{2}+ \|v_{xx}(t)\|_{L^2}^{2} + \alpha^2  \, \| u_{xxz}(t)\|_{L^2}^2 + \alpha^2  \, \| v_{xxz}(t)\|_{L^2}^2 \\
 	&&\hskip.18in
 	+ \int_0^t \left(2 \ee_1  \|u_{xx}(s)\|_{L^2}^{2} + 2\ee_1  \|v_{xx}(s)\|_{L^2}^{2} +  \ee_2  \|w_{xx}(s)\|_{L^2}^{2} \right) \; ds   \nonumber  \\
 	&&\hskip-.28in
 	\leq C(1+\|\partial^2_{xx} u_{0}\|_{L^2}^{2}+ \|\partial^2_{xx} v_{0}\|_{L^2}^{2} +  \alpha^2  \, \| \partial^3_{xxz} u_{0}\|_{L^2}^2 + \alpha^2  \, \| \partial^3_{xxz} v_{0}\|_{L^2}^2 ) \exp \Big(C \int_0^t ( 1+ \|w_{}(s)\|^2_{L^2}\\
 	&&\hskip-.08in  + \|w_{x}(s)\|^2_{L^2} + \|u_{z}(s)\|_{L^2}^2 + \|u_{xz}(s)\|_{L^2}^2 + \|\nabla v(s)\|_{L^2}^2 + \|v_{xz}(s)\|_{L^2}^2 +
 	\|T_{xx}(s)\|_{L^2}^2 ) ds \Big)
 \end{eqnarray*}
 From (\ref{L-2}), (\ref{L-2z}), (\ref{LL-2x}), (\ref{LL-2x2}), (\ref{vH-1}), (\ref{TH-2}) and the above, we have
 \begin{eqnarray}
 u,v \in L^{\infty}(0,\mathcal{T};H^2), \;\; u_{xxz}, v_{xxz} \in L^{\infty}(0,\mathcal{T};L^2), \;\; w \in L^{2}(0,\mathcal{T};H^2), \label{uvxx}
 \end{eqnarray}
 for arbitrary $\mathcal{T}>0$. By virtue of (\ref{uvxx}), thanks (\ref{ESw}) and (\ref{ESwx}), we have
 \begin{eqnarray}
 w \in L^{\infty}(0,\mathcal{T};H^1)\cap L^{2}(0,\mathcal{T};H^2)  \label{wb},
 \end{eqnarray}
for arbitrary $\mathcal{T}>0$. Using standard Galerkin method as in section 3.2, we can establish the existence result, and we omit the details here. Next,  we show the continuous dependence on the initial data and the the uniqueness of the strong solutions. Let $(u_1, v_1, w_1, p_1, T_1)$ and
$(u_2, v_2, w_2, p_2, T_2)$ be two strong solutions of the system
(\ref{AEQ-1})--(\ref{AEQ-5}) with (\ref{EQw})--(\ref{EQpz}), and initial data $((u_0)_1, (v_0)_1, (T_0)_1)$ and
$((u_0)_2, (v_0)_2, (T_0)_2)$, respectively. Denote by $u=u_1-u_2, v=v_1-v_2, w=w_1-w_2, p=p_1 -p_2, T=T_1-T_2.$ Thanks to (\ref{AEQ-3}) and (\ref{AEQ-4}),
it is clear that
\begin{eqnarray}
&&\hskip-.8in
\pp_t \left( u - \alpha^2 u_{zz} \right)  + u_1 u_x + w_1 u_z  + u (u_2)_x + w (u_2)_z + \ee_1 u -f_0 v + p_x = 0,  \label{AUPP-1}  \\
&&\hskip-.8in
\pp_t \left( v - \alpha^2 v_{zz} \right)  + u_1 v_x + w_1 v_z  + u (v_2)_x + w (v_2)_z + \ee_1 v +f_0 u = 0,  \label{AUPP-2}  \\
&&\hskip-.8in
\ee_2 w + p_z +T =0,   \label{AUPP-3}   \\
&&\hskip-.8in
u_x + w_z = 0, \label{AUPP-5} \\
&&\hskip-.8in
\pp_t T - \kappa \Dd T  + u_1 T_x + w_1 T_z  + u (T_2)_x + w (T_2)_z  = 0.  \label{AUPP-4}
\end{eqnarray}
By taking the inner product of equation (\ref{AUPP-1}) with $u$, (\ref{AUPP-2}) with  $v$,
(\ref{AUPP-3}) with  $w$, and (\ref{AUPP-4}) with  $T$
in $L^2(\mathbb{T}^2)$, and by integration by parts, thanks to (\ref{BC-1}), (\ref{AEQ-4}) and (\ref{AUPP-5}), we get
\begin{eqnarray*}
	&&\hskip-.58in \frac{1}{2} \frac{d (\|u\|_{L^2}^{2}+\|v\|_{L^2}^{2}+\|T\|_{L^2}^{2}+ \alpha^2  \, \| u_z\|_{L^2}^2
		+ \alpha^2  \, \| v_z\|_{L^2}^2 )}{dt} + \ee_1 \|u\|_{L^2}^{2} + \ee_1 \|v\|_{L^2}^{2}
	+ \ee_2  \|w\|_{L^2}^{2} + \kappa  \, \| \nabla T\|_{L^2}^2 \\
	&&\hskip-.55in
	\leq  \left| \int_{\mathbb{T}^2} \left[ u (u_2)_x+w(u_2)_z \right]\; u \; dxdz \right|+\left|\int_{\mathbb{T}^2} \left[ u (v_2)_x+w(v_2)_z
	\right]\; v \; dxdz \right|   \\
	&&\hskip-.365in  +\left|\int_{\mathbb{T}^2} wT\; dxdz\right| +\left|\int_{\mathbb{T}^2} \left[ u (T_2)_x+w(T_2)_z
	\right]\; T \; dxdz\right| =: \rom{1}+\rom{2}+\rom{3}+\rom{4}.
\end{eqnarray*}
By integration by parts, using H\"older inequality and Young's inequality, thanks to (\ref{BC-1}), (\ref{AEQ-4}) and (\ref{AUPP-5}) and Lemma 1,
\begin{eqnarray}
	&&\hskip-.48in
	\rom{1} = \left| \int_{\mathbb{T}^2} \left[ u (u_2)_x+w(u_2)_z \right]\; u \; dxdz \right| \leq C \int_{\mathbb{T}^2}|w u_z u_2| + | w (u_2)_z u| \;dxdz, \nonumber\\
	&&\hskip-.28in
	\leq \frac{\ee_2}{8}\|w\|_{L^2}^2 + C \|u_2\|_{L^\infty}^2 \|u_z\|_{L^2}^2 + C\|(u_2)_z\|_{L^2}(\|(u_2)_z\|_{L^2} + \|(u_2)_{xz}\|_{L^2}) (\|u\|_{L^2}^2 + \alpha^2 \|u_z\|_{L^2}^2),  \nonumber \\
&&\hskip-.48in
\rom{2} = \left|\int_{\mathbb{T}^2} \left[ u (v_2)_x+w(v_2)_z
\right]\; v \; dxdz \right| \leq \frac{\ee_2}{8} \|w\|_{L^2}^2 + C\|(v_2)_x\|_{L^\infty}  (\|u\|_{L^2}^2 + \|v\|_{L^2}^2) \nonumber  \\
&&\hskip.08in
+ C\|(v_2)_z\|_{L^2} (\|(v_2)_z\|_{L^2} + \|(v_2)_{xz}\|_{L^2}) (\|v\|_{L^2}^2 + \alpha^2 \|v_z\|_{L^2}^2), \label{AUV}\\
&&\hskip-.48in
	\rom{3} = \left|\int_{\mathbb{T}^2} wT\; dxdz\right| \leq \frac{\ee_2}{8} \|w\|_{L^2}^2 + C \|T\|_{L^2}^2, \nonumber \\
&&\hskip-.48in
	\rom{4} = \left|\int_{\mathbb{T}^2} \left[ u (T_2)_x+w(T_2)_z
	\right]\; T \; dxdz\right| \nonumber\\
	&&\hskip-.28in
	\leq  C \|T\|_{L^2} \|(T_2)_{x}\|_{L^2}^{\frac{1}{2}}  (\|(T_2)_{x}\|_{L^2}^{\frac{1}{2}} + \|(T_2)_{xx}\|_{L^2}^{\frac{1}{2}}) (\|u\|_{L^2} + \|u_{z}\|_{L^2}) \nonumber \\
	&&\hskip-.08in
	+ \frac{\ee_2}{8} \|w\|_{L^2}^2 + \frac{\kappa}{2}\|T_z\|_{L^2}^2 +  C( 1+ \|(T_2)_{z}\|_{L^2}^2)  (\|(T_2)_{z}\|_{L^2}^2 + \|(T_2)_{xz}\|_{L^2}^2) \|T\|_{L^2}^2 \nonumber\\
	&&\hskip-.28in
	\leq \frac{\ee_2}{8} \|w\|_{L^2}^2 + \frac{\kappa}{2}\|\nabla T\|_{L^2}^2 + C (1+ \|(T_2)_{x}\|_{L^2}+  \|(T_2)_{xx}\|_{L^2} + \|(T_2)_{z}\|_{L^2}^4 + \|(T_2)_{z}\|_{L^2}^2\|(T_2)_{xz}\|_{L^2}^2) \nonumber\\
	&&\hskip1.68in
	\times (\|u\|_{L^2}^2 + \|T\|_{L^2}^2 + \alpha^2  \, \| u_z\|_{L^2}^2).
\end{eqnarray}	
From the estimates above, we obtain
\begin{eqnarray*}
	&&\hskip-.68in \frac{d (\|u\|_{L^2}^{2}+\|v\|_{L^2}^{2}+\|T\|_{L^2}^{2}+ \alpha^2  \, \| u_z\|_{L^2}^2
		+ \alpha^2  \, \| v_z\|_{L^2}^2 )}{dt} + \ee_1 \|u\|_{L^2}^{2} + \ee_1 \|v\|_{L^2}^{2}
	+ \ee_2  \|w\|_{L^2}^{2} + \kappa  \, \| \nabla T\|_{L^2}^2 \\
	&&\hskip-.65in
	\leq   C K \left(\|u\|_{L^2}^{2}+\|v\|_{L^2}^{2}+\|T\|_{L^2}^2 + \alpha^2  \, \| u_z\|_{L^2}^2
	+ \alpha^2  \, \| v_z\|_{L^2}^2\right),
\end{eqnarray*}
where
\begin{eqnarray*}
	&&\hskip-.68in
	K=1+ \|u_2\|_{L^\infty}^2 + \|(v_2)_x\|_{L^\infty} + \|(u_2)_z\|_{L^2}^2 + \|(u_2)_{xz}\|_{L^2}^2 + \|(v_2)_z\|_{L^2}^2 + \|(v_2)_{xz}\|_{L^2}^2 \\
	&&\hskip-.28in
	+ \|(T_2)_{x}\|_{L^2} + \|(T_2)_{xx}\|_{L^2} + \|(T_2)_{z}\|_{L^2}^4 + \|(T_2)_{z}\|_{L^2}^2\|(T_2)_{xz}\|_{L^2}^2.
\end{eqnarray*}	
Using lemma 2, thanks to (\ref{TH-2}) and (\ref{uvxx}), we obtain $K\in L^1(0,\mathcal{T})$. Therefore, by Gronwall inequality, we obtain
\begin{eqnarray*}
	&&\hskip-.68in
	\|u(t)\|_{L^2}^{2}+\|v(t)\|_{L^2}^{2}+\|T(t)\|_{L^2}^{2}+ \alpha^2  \, \| u_z(t)\|_{L^2}^2
	+ \alpha^2  \, \| v_z(t)\|_{L^2}^2   \\
	&&\hskip-.65in
	\leq
	(\|u(t=0)\|_{L^2}^2 + \alpha^2 \|u_z(t=0)\|_{L^2}^2+\|v(t=0)\|_{L^2}^2 + \alpha^2 \|v_z(t=0)\|_{L^2}^2+\|T(t=0)\|_{L^2}^2)
	e^{ C \int_0^t K(s)ds},
\end{eqnarray*}
The above inequality proves the continuous dependence of the
solutions on the initial data, and in particular, when
$u(t=0)=v(t=0)=T(t=0)=0$, we have $u(t)=v(t)=T(t)=0,$ for all $t\geq 0$.
Therefore, the strong solution is unique.
\end{proof}
In the case when $f_0=0$ and $v\equiv 0$, our system (\ref{EQO1-1})--(\ref{EQO1-5}) will be reduced to the following system:
\begin{eqnarray}
&&\hskip-.8in
(u-\alpha^2  u_{zz})_t  + u\, u_x + w u_z + p_x = 0,  \label{SEQO-1}  \\
&&\hskip-.8in
\ee_2 w + p_z  +  T =0,    \label{SEQO-2}  \\
&&\hskip-.8in
u_x + w_z =0,   \label{SEQO-3} \\
&&\hskip-.8in
T_t - \kappa \Dd T   + u \, T_x + w \, T_z =  0, \label{SEQO-4}
\end{eqnarray}
in $\mathbb{T}^2$. We impose similar boundary and initial conditions for this system:
\begin{eqnarray}
&&\hskip-.8in
u, \;w,\; p\;  \text{and} \; T\;  \text{are periodic in} \; x \; \text{and} \; z \; \text{with period 1},  \label{SBCO-1} \\
&&\hskip-.8in
u, \; p\;  \text{are even in} \; z, \; \text{and} \; w, \; T \; \text{are odd in} \; z, \label{SBCO-2} \\
&&\hskip-.8in
(u, T)|_{t=0}=(u_0, T_0). \label{SICO-1}
\end{eqnarray}
Using analogue argument as in section 3.1, the system (\ref{SEQO-1})--(\ref{SEQO-4}) subjects to (\ref{SBCO-1})--(\ref{SICO-1}) is equivalent to the following:
\begin{eqnarray}
&&\hskip-.8in
(u-\alpha^2  u_{zz})_t  + u\, u_x + w u_z +\ee_1 u + p_x = 0,  \label{SEQ-1}  \\
&&\hskip-.8in
T_t - \kappa \Dd T   + u \, T_x + w \, T_z =  0, \label{SEQ-2}
\end{eqnarray}
with $w,p_x,p_z$ defined by:
\begin{eqnarray}
&&\hskip-.8in
w(x,z):=-\int_0^z u_x(x,s)ds,  \label{SEQw}  \\
&&\hskip-.8in
p_{x}(x,z):= \ee_2 \int_0^z\int_0^s u_{xx}(x,\xi)d\xi ds-\int_0^zT_x(x,s)ds \nonumber \\
&&\hskip-.2in
+\int_0^1 \Big[\int_0^{z'} T_x(x,s)ds-\ee_2 \int_0^{z'}\int_0^s u_{xx}(x,\xi)d\xi ds -2uu_x(x,z')\Big]dz', \label{SEQpx}  \\
&&\hskip-.8in
p_z(x,z):= -T(x,z) + \ee_2 \int_0^z u_x(x,s)ds. \label{SEQpz}
\end{eqnarray}
We are interested in the system (\ref{SEQ-1})--(\ref{SEQ-2}) with (\ref{SEQw})--(\ref{SEQpz}) in the unit two dimensional torus $\mathbb{T}^2$, subject to the following symmetry boundary conditions and initial conditions:
\begin{eqnarray}
&&\hskip-.8in
u \; \text{and} \; T\; \text{are periodic in} \; x \; \text{and} \; z\;  \text{with period 1}; \label{SBC-1} \\
&&\hskip-.8in
u \; \text{is even in} \; z, \; \text{and} \; T\;  \text{is odd in}\; z; \label{SBC-2}  \\
&&\hskip-.8in
(u, T)|_{t=0}=(u_0, T_0).\label{SIC-1}
\end{eqnarray}
We have the global well-posedness for system (\ref{SEQ-1})--(\ref{SEQ-2}) with (\ref{SEQw})--(\ref{SEQpz}), for initial condition with less regularity. i.e., for $u_0, \partial_z u_{0}\in H^1$ and $T_0\in{H^1}$. Let us give the definition of strong solution first.
\begin{definition}
	Suppose that $u_0 \in H^1(\mathbb{T}^2)$ and $T_0 \in H^1(\mathbb{T}^2)$ satisfy the symmetry conditions (\ref{SBC-1}) and (\ref{SBC-2}), with the compatibility condition $\int_0^1 \partial_x u_{0} dz = 0$. Moreover, suppose that $\partial_z u_{0}\in H^1$. Given time $\mathcal{T}>0$, we say $(u,T)$ is a strong solution to the system (\ref{SEQ-1})--(\ref{SEQ-2}) with (\ref{SEQw})--(\ref{SEQpz}), subjecto to (\ref{SBC-1})--(\ref{SIC-1}), on the time interval $[0,\mathcal{T}]$, if \\
	
	(i) u and T satisfy the symmetry conditions (\ref{SBC-1}) and (\ref{SBC-2});\\
	
	(ii) u and T have the regularities
	\begin{eqnarray*}
		&&\hskip-.8in
		u\in L^\infty(0,\mathcal{T};H^1)\cap C([0,\mathcal{T}]; L^2), \;\; u_{z}\in L^\infty(0,\mathcal{T};H^1), \;\; 	\partial_t u\in L^2(0,\mathcal{T};L^2), \\
		&&\hskip-.8in
		T\in L^2(0,\mathcal{T};H^2)\cap L^\infty(0,\mathcal{T};H^1)\cap C([0,\mathcal{T}]; L^2), \;\; 	\partial_t T \in L^2(0,\mathcal{T};L^2);
	\end{eqnarray*}
	
	(iii) u, T satisfy system (\ref{SEQ-1})--(\ref{SEQ-2}) in the following sense:
	\begin{eqnarray*}
		&&\hskip-.8in
		(u-\alpha^2  u_{zz})_t  + u\, u_x + w u_z +\ee_1 u + p_x = 0  \;\;  \text{in} \; \; L^2(0,\mathcal{T}; L^2);\\
		&&\hskip-.8in
		T_t - \kappa \Dd T   + u \, T_x + w \, T_z =  0  \;\;  \text{in} \; \; L^2(0,\mathcal{T}; L^2),
	\end{eqnarray*}
	with $w,p_x,p_z$ defined by (\ref{SEQw})--(\ref{SEQpz}), and fulfill the initial condition (\ref{SIC-1}).
\end{definition}
Based on theorem \ref{T-MAIN3}, we have the following theorem on the existence and uniqueness of strong solutions to system (\ref{SEQ-1})--(\ref{SEQ-2}) with (\ref{SEQw})--(\ref{SEQpz}), subject to (\ref{SBC-1})--(\ref{SIC-1}), on $\mathbb{T}^2 \times (0,\mathcal{T})$, for any positive time $\mathcal{T}$. The proof is similar as theorem \ref{T-MAIN3}, and we omit the details here.

\begin{theorem}
	Suppose that $u_0 \in H^1(\mathbb{T}^2)$ and $T_0 \in H^1(\mathbb{T}^2)$ satisfy the symmetry conditions (\ref{SBC-1}) and (\ref{SBC-2}), with the compatibility condition $\int_0^1 \partial_x u_{0} dz = 0$. Moreover, suppose that $\partial_z u_{0}\in H^1$. Given time $\mathcal{T}>0$.
	Then there exists a unique strong solution $(u, T)$ of the system (\ref{SEQ-1})--(\ref{SEQ-2}) with (\ref{SEQw})--(\ref{SEQpz}), subject to (\ref{SBC-1})--(\ref{SIC-1}), on the interval $[0,\mathcal{T}]$. Moreover, the unique strong solution $(u, T)$ depends continuously on the initial data. Same result holds when $T\equiv 0$.
\end{theorem}
\begin{remark}
	The reason why we need to assume more regularity for the initial data to system (\ref{AEQ-1})--(\ref{AEQ-5}) is that we need a bound for $\|(v_2)_x\|_{L^\infty}$ appears in (\ref{AUV}). If we do not have the evolution equation in $v$, we can require less for the initial data.
\end{remark}

\section{Convergence}
In this section, we will prove the convergence of the strong solution of the following system
\begin{eqnarray}
&&\hskip-.8in
(u^\alpha-\alpha^2  u^\alpha_{zz})_t -\nu u^\alpha_{zz} + u^\alpha u^\alpha_x + w^\alpha u^\alpha_z + \ee_1 u^\alpha + p^\alpha_x = 0,  \label{C-1}  \\
&&\hskip-.8in
\ee_2 w^\alpha + p^\alpha_z =0 ,   \label{C-2}  \\
&&\hskip-.8in
u^\alpha_x + w^\alpha_z =0 ,  \label{C-3}
\end{eqnarray}
subjects to the following symmetric boundary conditions and initial condition
\begin{eqnarray}
&&\hskip-.8in
u^\alpha, \; w^\alpha \;\;\text{and}\;\; p^\alpha \;\;\text{are periodic in} \; x \; \text{and} \; z \; \text{with period 1}; \label{CBC-1} \\
&&\hskip-.8in
u^\alpha, \;p^\alpha \;\; \text{are even in} \;z, \;\; \text{and} \;\; w^\alpha \;\; \text{is odd in}\; z; \label{CBC-2} \\
&&\hskip-.8in
u^\alpha|_{t=0}=u^\alpha_0,\label{CIC-1}
\end{eqnarray}
to the strong solution of system (\ref{PP-1})--(\ref{PP-3}) subjects to (\ref{DBC-1})--(\ref{DIC-1}), as $\alpha\rightarrow 0$.
\begin{remark}
	The global well-posedness of system (\ref{C-1})--(\ref{C-3}) subjects to (\ref{CBC-1})--(\ref{CIC-1}) can be easily obtained as in section 4. Moreover, as indicated in the last part of section 4, we only need to assume that $u^\alpha_0$ and $u^\alpha_{0z} \in H^1(\mathbb{T}^2)$ since we do not have the evolution equation in $v^\alpha$.
\end{remark}
\begin{theorem}	\label{T-C}
	Suppose that $u_0, \{u_0^\alpha\}_{0< \alpha \leq 1}\subset H^1(\mathbb{T}^2)$ satisfy the symmetry conditions (\ref{DBC-1})--(\ref{DBC-2}) and (\ref{CBC-1})--(\ref{CBC-2}), with the compatibility conditions $\int_0^1 \partial_x u_{0} dz = 0$ and $\int_0^1 \partial_x u_{0}^\alpha dz = 0$, for $\forall \; 0<\alpha\leq 1$, and suppose that $\partial^2_{xz} u_{0}\in L^2(\mathbb{T}^2), \{\partial_z u_{0}^\alpha\}_{0< \alpha \leq 1}\subset H^1(\mathbb{T}^2)$. Moreover, suppose there exists some constant $M>0$ such that the following uniform bound for initial data holds:
	\begin{eqnarray}
	\sup_{0< \alpha \leq 1} \;\Big( \|u_0^\alpha\|_{L^2} + \|\partial_z u_0^\alpha\|_{L^2} + \alpha \|\partial^2_{zz} u_0^\alpha\|_{L^2} \Big) \leq M. \label{UN}
	\end{eqnarray}
	Let $\mathcal{T}>0$ be such that $u$ is the strong solution of system (\ref{PP-1})--(\ref{PP-3}) on $[0,\mathcal{T}]$ with initial data $u_0$. Let $u^\alpha$ be the strong solution to system (\ref{C-1})--(\ref{C-3}) on $[0,\mathcal{T}]$ with initial data $u_0^\alpha$. If $u_0^\alpha \rightarrow u_0$ in $L^2$, as $\alpha \rightarrow 0$, then $u^\alpha \rightarrow u$ in $L^\infty(0,\mathcal{T}; L^2)$, and $u_z^\alpha \rightarrow u_z$ in $L^2(0,\mathcal{T}; L^2)$, as $\alpha \rightarrow 0$.
\end{theorem}
\begin{proof}
	Let us first derive the uniform bounds of some norms of the strong solution $u^\alpha$. By taking the $L^2$-inner product of equation  (\ref{C-1}) with $u^\alpha, -u^\alpha_{zz}$, and
	equation  (\ref{C-2}) with $w^\alpha, -w^\alpha_{zz}$, in $L^2(\mathbb{T}^2)$, and by integration by parts, thanks to (\ref{CBC-1}), we get
	\begin{eqnarray*}
		&&\hskip-.48in \frac{1}{2} \frac{d}{dt} \Big(\|u^\alpha\|_{L^2}^{2}+ (\alpha^2+1)  \, \| u^\alpha_{z}\|_{L^2}^2 + \alpha^2\; \| u^\alpha_{zz}\|_{L^2}^2 \Big) + \ee_1 \Big( \|u^\alpha\|_{L^2}^{2} + \|u^\alpha_z\|_{L^2}^{2} \Big)\\
		&&\hskip-.25in
		+ \ee_2  \Big(\|w^\alpha\|_{L^2}^{2} + \|w^\alpha_z\|_{L^2}^{2} \Big) + \nu \Big(\|u^\alpha_z\|_{L^2}^{2} + \|u^\alpha_{zz}\|_{L^2}^{2}\Big) \\
		&&\hskip-.45in
		=  -\int_{\mathbb{T}^2} \left( u^\alpha u^\alpha_{x}+w^\alpha u^\alpha_{z}
		\right)\; \left(u^\alpha - u^\alpha_{zz} \right) \; dxdz  -\int_{\mathbb{T}^2}\Big( p^\alpha_{x}\left( u^\alpha-u^\alpha_{zz}
		\right) + p^\alpha_{z}\left( w^\alpha-w^\alpha_{zz}
		\right) \Big)dxdz.
	\end{eqnarray*}
	By integration by parts, thanks to (\ref{C-3}) and (\ref{CBC-1}), we have
	\begin{eqnarray*}
		&&\hskip-.45in   -\int_{\mathbb{T}^2} \left( u^\alpha u^\alpha_{x}+w^\alpha u^\alpha_{z}
		\right)\; \left(u^\alpha - u^\alpha_{zz} \right) \; dxdz  -\int_{\mathbb{T}^2}\Big( p^\alpha_{x}\left( u^\alpha-u^\alpha_{zz}
		\right) + p^\alpha_{z}\left( w^\alpha-w^\alpha_{zz}
		\right) \Big)dxdz =0.
	\end{eqnarray*}
	As a result of the above, we have
	\begin{eqnarray*}
		&&\hskip-.48in
		\frac{d}{dt} \Big(\|u^\alpha\|_{L^2}^{2}+ (\alpha^2+1)  \, \| u^\alpha_{z}\|_{L^2}^2 + \alpha^2\; \| u^\alpha_{zz}\|_{L^2}^2 \Big) + \nu \Big(\|u^\alpha_z\|_{L^2}^{2} + \|u^\alpha_{zz}\|_{L^2}^{2}\Big) + \ee_2  \Big(\|w^\alpha\|_{L^2}^{2} + \|w^\alpha_z\|_{L^2}^{2} \Big) \leq 0.
	\end{eqnarray*}
	Thanks to Gronwall inequality, we obtain
	\begin{eqnarray*}
		&&\hskip-.68in
		\|u^\alpha(t)\|_{L^2}^{2}+ \, \| u^\alpha_{z}(t)\|_{L^2}^2+
		\int_0^t \Big[\nu \left( \|u^\alpha_z(s)\|_{L^2}^{2} + \|u^\alpha_{zz}(s)\|_{L^2}^{2}\right) + \ee_2  \left(\|w^\alpha(s)\|_{L^2}^{2} + \|w^\alpha_z(s)\|_{L^2}^{2} \right)\Big] \; ds  \\
		&&\hskip-.38in
		\leq  \|u^\alpha_0\|_{L^2}^{2}+ (1+\alpha^2) \| \partial_z u^\alpha_{0}\|_{L^2}^2 + \alpha^2 \| \partial^2_{zz} u^\alpha_{0}\|_{L^2}^2,
	\end{eqnarray*}
	for $t\in [0,\mathcal{T}]$. Thanks to the uniform bound for initial data (\ref{UN}),  we have
	\begin{eqnarray}
	&&\hskip-.8in
	\sup_{0< \alpha \leq 1} \;\Big(\|u^\alpha\|_{L^\infty(0,\mathcal{T};L^2)} + \|u^\alpha_z\|_{L^\infty(0,\mathcal{T};L^2)} + \nu \|u^\alpha_{zz}\|_{L^2(0,\mathcal{T};L^2)} \nonumber \\
	&&\hskip-.18in
	  + \ee_2 \|w^\alpha\|_{L^2(0,\mathcal{T};L^2)} + \ee_2 \|w^\alpha_z\|_{L^2(0,\mathcal{T};L^2)} \Big) \leq C(M), \label{BDalpha}
	\end{eqnarray}
	where $C(M)$ is a constant depending on $M$, but not on $\alpha$.
	Now subtracting (\ref{PP-1})--(\ref{PP-2}) from (\ref{C-1})--(\ref{C-2}), we obtain
	\begin{eqnarray}
	&&\hskip-.8in
	\partial_t [(u^\alpha-u)-\alpha^2  (u^\alpha_{zz} - u_{zz})] -\nu (u^\alpha_{zz} - u_{zz})+ \ee_1(u^\alpha - u) + (p^\alpha_x - p_x) \nonumber\\
	&&\hskip-.8in = (u-u^\alpha)u_x + (u_x - u^\alpha_x)u^\alpha + (w-w^\alpha)u_z + (u_z-u^\alpha_z)w^\alpha  - \alpha^2 \partial_t u_{zz},  \label{CD-1}  \\
	\nonumber\\
	&&\hskip-.8in
	\ee_2 (w^\alpha - w) + (p^\alpha_z- p_z) =0 .   \label{CD-2}
	\end{eqnarray}
	For simplicity, we denote $\|\cdot \|:= \|\cdot \|_{L^2}$ from now on. By taking the inner product of equation  (\ref{CD-1}) with $u^\alpha - u$ and
	equation (\ref{CD-2}) with $w^\alpha -w$, by integration by parts, and using (\ref{PP-3}) and (\ref{C-3}), we get
	\begin{eqnarray*}
		&&\hskip-.68in \frac{1}{2} \frac{d}{dt}(\|u^\alpha - u\|^2 + \alpha^2 \|u^\alpha_z - u_z\|^2)
		+ \nu\, \|u_z^\alpha - u_z\|^2 + + \ee_1  \|u^\alpha - u\|^{2}  + \ee_2  \|w^\alpha - w\|^{2}  \\
		&&\hskip-.68in =  \int_{\mathbb{T}^2} \Big[(u-u^\alpha)^2w_z + (u^\alpha-u)(u_x-u^\alpha_x)u^\alpha + (w-w^\alpha)(u^\alpha-u)u_z + (u_z-u^\alpha_z)(u^\alpha-u)w^\alpha \\
		&&\hskip-.68in + (p_x-p^\alpha_x)(u^\alpha-u) + (p_z-p^\alpha_z)(w^\alpha-w) + \alpha^2 \partial_t u_{zz}(u-u^\alpha)\Big] dxdz \\
		&&\hskip-.68in =: I_1+I_2+I_3+I_4+I_5+I_6+I_7.
	\end{eqnarray*}
	By integration by parts, using H\"older inequality and Young's inequality, thanks to (\ref{PP-3}), (\ref{DBC-1}), (\ref{C-3}) and (\ref{CBC-1}), we have
	\begin{eqnarray*}
		&&\hskip-.28in
		I_1 =
		 \int_{\mathbb{T}^2} (u-u^\alpha)^2w_z  dxdz = -2 \int_{\mathbb{T}^2} (u^\alpha-u)_z (u^\alpha-u) w  dxdz
		\leq \frac{\nu}{2} \|u^\alpha_z - u_z\|^2 + C \|w\|_\infty^2 \|u^\alpha - u\|^2,\\
		&&\hskip-.28in
		I_2 + I_4 = \int_{\mathbb{T}^2}\Big[ (u^\alpha-u)(u_x-u^\alpha_x)u^\alpha+ (u_z-u^\alpha_z)(u^\alpha-u)w^\alpha\Big] dxdz \\
		&&\hskip.18in
		  = \frac{1}{2}\int_{\mathbb{T}^2} (u^\alpha-u)^2 (u^\alpha_x + w^\alpha_z) dxdz = 0,\\
		&&\hskip-.28in
		I_3 = \int_{\mathbb{T}^2} (w-w^\alpha)(u^\alpha-u)u_z  dxdz \leq \frac{\ee_2}{2}\|w^\alpha - w\|^2 + C\|u_z\|_\infty^2 \|u^\alpha - u\|^2,\\
        &&\hskip-.28in
		I_5 + I_6 = \int_{\mathbb{T}^2} \Big[(p_x-p^\alpha_x)(u^\alpha-u) + (p_z-p^\alpha_z)(w^\alpha-w)\Big] dxdz \\
		&&\hskip.18in
		 = \int_{\mathbb{T}^2} (p-p^\alpha)[(u-u^\alpha)_x + (w-w^\alpha)_z] dxdz = 0,\\
		&&\hskip-.28in
		I_7 = \alpha^2 \int_{\mathbb{T}^2} \partial_t u_{zz} (u- u^\alpha ) dxdz = \alpha^2 \int_{\mathbb{T}^2} u_t (u -  u^\alpha )_{zz} dxdz \leq C\alpha^2 \|u_t\| (\|u^\alpha_{zz}\| + \|u_{zz}\|).
	\end{eqnarray*}
     From all the estimates above, we obtain
      \begin{eqnarray*}
      	&&\hskip-.68in
      	\frac{d}{dt} (\|u^\alpha - u \|^2 + \alpha^2 \|u_z^\alpha - u_z\|^2 ) + \nu\|u_z^\alpha - u_z\|^2 + \ee_1 \|u^\alpha - u\|^2 + \ee_2 \|w^\alpha - w\|^2 \\
      	&&\hskip-.68in
      	\leq C(\|w\|_\infty^2  + \|u_z\|_\infty^2)(\|u^\alpha - u \|^2 + \alpha^2 \|u_z^\alpha - u_z\|^2) +C\alpha^2 \|u_t\| (\|u^\alpha_{zz}\| + \|u_{zz}\|) .
      	\end{eqnarray*}
     Let us denote by
     $$ F:= \|w\|_{\infty}^2  + \|u_z\|_{\infty}^2, \;\; G:= \|u_t\| (\|u^\alpha_{zz}\| + \|u_{zz}\|).$$
     Therefore, we obtain
     \begin{eqnarray*}
     	&&\hskip-.68in \frac{d}{dt} (\|u^\alpha - u \|^2 + \alpha^2 \|u_z^\alpha - u_z\|^2 ) + \nu\|u_z^\alpha - u_z\|^2 + \ee_1 \|u^\alpha -u \|^2  + \ee_2 \|w^\alpha - w\|^2 \\
     	&&\hskip-.48in \leq CF(\|u^\alpha - u \|^2 + \alpha^2 \|u_z^\alpha - u_z\|^2) + C\alpha^2 G.
     \end{eqnarray*}
     Notice that the constant C appears above may change from line to line, and may depend on $\nu$, $\ee_2$, and $\mathbb{T}^2$, but not on $\alpha$.
     Thanks to Gronwall inequality, we obtain
     \begin{eqnarray}
     	&&\hskip-.68in \|u^\alpha - u \|^2(t) + \alpha^2 \|u_z^\alpha - u_z\|^2(t) + \int_0^t \Big(\nu\|u_z^\alpha - u_z\|^2(s) + \ee_1 \|u^\alpha - u\|^2(s) +\ee_2 \|w^\alpha - w\|^2(s) \Big) ds  \nonumber\\
     	&&\hskip-.68in \leq (\|u^\alpha_0 - u_0\|^2 + \alpha^2 \|\partial_z u_{0}^\alpha - \partial_z u_{0}\|^2)\exp\Big(C \int_0^t F(s)ds\Big) + C\alpha^2 \exp\Big( C \int_0^t F(s)ds\Big) \int_0^t  G(s)ds  \nonumber\\
     	&&\hskip-.68in =: (\|u^\alpha_0 - u_0\|^2 + \alpha^2 \|\partial_z u_{0}^\alpha - \partial_z u_{0}\|^2)\exp\Big(C \int_0^t F(s)ds\Big) + C\alpha^2 H(t). \label{Convergence}
     \end{eqnarray}
      By virture of the regularity of strong solution to system (\ref{PP-1})--(\ref{PP-3}) as stated in Definition 6, and
      the uniform bound (\ref{BDalpha}), using Lemma 2, we have $F,G \in L^1(0,\mathcal{T})$. By virtue of uniform bound (\ref{BDalpha}), we have $\alpha^2 H(t) \rightarrow 0$, as $\alpha \rightarrow 0$. Since $u_0^\alpha \rightarrow u_0$ in $L^2$, and thanks to (\ref{UN}), we have $u^\alpha \rightarrow u$ in $L^\infty(0,\mathcal{T}; L^2)$, $u^\alpha_z \rightarrow u_z$ in $L^2(0,\mathcal{T}; L^2)$, and $w^\alpha \rightarrow w$ in $L^2(0,\mathcal{T}; L^2)$, as $\alpha \rightarrow 0$.

\end{proof}


\section{Blow-up Criterion}
In this section we give a blow-up criterion for system (\ref{PP-1})--(\ref{PP-3}) subjects to (\ref{DBC-1})--(\ref{DIC-1}). The following result follows the idea in \cite{LPTW18}.
\begin{theorem}
	With the same assumptions in Theorem \ref{T-C}, and take $u_0^\alpha = u_0$ for all $\alpha$. Suppose there exists some time $\mathcal{T}^* < \infty$ such that
	\begin{eqnarray}
	\limsup\limits_{\alpha\rightarrow 0^+} \Big(\alpha^2\sup\limits_{t\in[0,\mathcal{T}^*]}   \|u_z^\alpha(t)\|^2\Big) >0, \label{BL-2}
	\end{eqnarray}
	then the solution for system (\ref{PP-1})--(\ref{PP-3}) blows up on $[0,\mathcal{T}^*]$.
\end{theorem}

\begin{proof}
	Assume the solution for system (\ref{PP-1})--(\ref{PP-3}) will not blow up on $[0,\mathcal{T}^*]$, then $u\in L^\infty (0,\mathcal{T}^*; H^1)$ and $\partial_t u\in L^2 (0,\mathcal{T}^*; L^2)$. By taking the inner product of equation (\ref{PP-1}) with u and equation (\ref{PP-2}) with w in $L^2(\mathbb{T}^2)$, by integration by parts and thanks to (\ref{PP-3}) and (\ref{DBC-1}), we have
	\begin{eqnarray}
	&&\hskip-.68in \frac{1}{2}\frac{d}{dt}\|u_t\|^2 + \nu \|u_z\|^2 + \ee_1 \|u\|^2 + \ee_2\|w\|^2=0. \label{B-0}
	\end{eqnarray}
	Integrating (\ref{B-0}) from 0 to $t$ for $t\in[0,\mathcal{T}^*]$, we have
	\begin{eqnarray}
	&&\hskip-.68in \|u(t)\|^2 + 2\int_0^t \Big(\nu \|u_z(s)\|^2 + \ee_1 \|u(s)\|^2 +  \ee_2 \|w(s)\|^2 \Big) ds = \|u_0\|^2. \label{B-1}
	\end{eqnarray}
	On the other hand, using analogue argument for system (\ref{C-1})--(\ref{C-3}), we have
	\begin{eqnarray}
	&&\hskip-.68in \alpha^2 \|u^\alpha_z(t)\|^2 + \|u^\alpha(t)\|^2 + 2\int_0^t \left(\nu \|u^\alpha_z(s)\|^2 + \ee_1 \|u^\alpha(s)\|^2 + \ee_2 \|w^\alpha(s)\|^2\right)ds \nonumber\\
	&&\hskip-.68in
	= \|u^\alpha_0\|^2 + \alpha^2 \|\partial_z u^\alpha_{0}\|^2 = \|u_0\|^2 + \alpha^2 \|\partial_z u_{0}\|^2 .\label{B-2}
	\end{eqnarray}	
	From (\ref{Convergence}) and thanks to the fact that $u_0^\alpha = u_0$, for any $t\in[0,\mathcal{T}^*]$, we have
	\begin{eqnarray}
	\|u^\alpha(t)\| \geq \|u(t)\| - C\alpha H^{1/2}(t) \geq \|u(t)\| - C\alpha H^{1/2}(\mathcal{T}^*), \label{BL2-1}
	\end{eqnarray}
since $H^{1/2}(t)$ is monotonically increasing. By virtue of (\ref{B-1}), we know $\|u_0\|\geq\|u(t)\|\geq \|u(\mathcal{T}^*)\|$ for any $t\in[0,\mathcal{T}^*]$. Therefore, we can take $\alpha < \frac{\|u(\mathcal{T}^*)\|}{CH^{1/2}(\mathcal{T}^*)}$ to guarantee the right hand side of (\ref{BL2-1}) is positive. Take square on (\ref{BL2-1}), we obtain
	\begin{eqnarray}
	&&\hskip-.68in
	\|u^\alpha(t)\|^2 \geq \|u(t)\|^2 - 2\alpha C H^{1/2}(\mathcal{T}^*)\|u(t)\| + C^2\alpha^2H(\mathcal{T}^*)\nonumber \\
	&&\hskip-.12in
	\geq \|u(t)\|^2 - 2\alpha C H^{1/2}(\mathcal{T}^*)\|u_0\| + C^2\alpha^2H(\mathcal{T}^*). \label{BL2-2}
	\end{eqnarray}
Subtracting (\ref{B-2}) from (\ref{B-1}), we have
\begin{eqnarray}
    &&\hskip-.68in
   \|u(t)\|^2 - \|u^\alpha(t)\|^2 = \alpha^2 \|u_z^\alpha(t)\|^2 - \alpha^2 \|\partial_z u_{0}\|^2 \nonumber \\
   &&\hskip-.28in
   + 2\int_0^t \Big(\nu \|u^\alpha_z(s)\|^2 + \ee_1 \|u^\alpha(s)\|^2 + \ee_2 \|w^\alpha(s)\|^2\Big) - \Big(\nu \|u_z(s)\|^2 + \ee_1 \|u(s)\|^2 +  \ee_2 \|w(s)\|^2 \Big)  ds. \label{BL2-3}
\end{eqnarray}
	Combining (\ref{BL2-3}) with (\ref{BL2-2}), we obtain
	\begin{eqnarray}
	&&\hskip-.68in
	\alpha^2 \|u_z^\alpha(t)\|^2 \leq \alpha^2 \|\partial_z u_{0}\|^2 + 2\alpha C H^{1/2}(\mathcal{T}^*)\|u_0\| - C^2\alpha^2H(\mathcal{T}^*)\nonumber \\
	&&\hskip-.28in
	+ 2\int_0^t \left(\nu \|u_z(s)\|^2 + \ee_1 \|u(s)\|^2 +  \ee_2 \|w(s)\|^2 \right) - \left(\nu \|u^\alpha_z(s)\|^2 + \ee_1 \|u^\alpha(s)\|^2 + \ee_2 \|w^\alpha(s)\|^2\right)  ds. \label{BL2-4}
	\end{eqnarray}
By Cauchy--Schwarz inequality and H\"older inequality, thanks to (\ref{DB-3})--(\ref{DB-4}) and
the uniform bound (\ref{BDalpha}), we have the estimate for the last term in (\ref{BL2-4}):
\begin{eqnarray*}
&&\hskip-.68in
2\int_0^t \left(\nu \|u_z(s)\|^2 + \ee_1 \|u(s)\|^2 +  \ee_2 \|w(s)\|^2 \right) - \left(\nu \|u^\alpha_z(s)\|^2 + \ee_1 \|u^\alpha(s)\|^2 + \ee_2 \|w^\alpha(s)\|^2\right)  ds \\
&&\hskip-.68in
=2\int_0^t \Big[ \nu \big(u_z-u^\alpha_z, u_z+ u^\alpha_z \big) +\ee_1 \big(u-u^\alpha, u+ u^\alpha \big)  + \ee_2  \big(w-w^\alpha, w+ w^\alpha \big) \Big] ds \\
&&\hskip-.68in
\leq 2 \int_0^t \Big[ \nu \|u_z-u_z^\alpha\| \|u_z+u_z^\alpha\| + \ee_1 \|u-u^\alpha\| \|u+u^\alpha\| + \ee_2 \|w-w^\alpha\| \|w+w^\alpha\| \Big] ds \\
&&\hskip-.68in
\leq C \left(\|u_z-u_z^\alpha\|_{L^2(0,\mathcal{T}^*; L^2)} + \|u-u^\alpha\|_{L^2(0,\mathcal{T}^*; L^2)} +  \|w-w^\alpha\|_{L^2(0,\mathcal{T}^*; L^2)} \right).
\end{eqnarray*}
Plugging this back into (\ref{BL2-4}), we have
	\begin{eqnarray}
	&&\hskip-.68in
	\alpha^2 \|u_z^\alpha(t)\|^2 \leq \alpha^2 \|\partial_z u_{0}\|^2 + 2\alpha C H^{1/2}(\mathcal{T}^*)\|u_0\| - C^2\alpha^2H(\mathcal{T}^*)\nonumber \\
	&&\hskip.08in
	+ C \left(\|u_z-u_z^\alpha\|_{L^2(0,\mathcal{T}^*; L^2)} + \|u-u^\alpha\|_{L^2(0,\mathcal{T}^*; L^2)} +  \|w-w^\alpha\|_{L^2(0,\mathcal{T}^*; L^2)} \right). \label{BL2-5}
	\end{eqnarray}
By virtue of Theorem \ref{T-C}, the right hand side of (\ref{BL2-5}) is independent of $t$, and it converges to 0 as $\alpha \rightarrow 0$. Therefore, by taking $\limsup\limits_{\alpha\rightarrow 0^+} \sup\limits_{t\in[0,\mathcal{T}^*]}$ on both hand sides of (\ref{BL2-5}), we obtain $$\limsup\limits_{\alpha\rightarrow 0^+} \Big(\alpha^2\sup\limits_{t\in[0,\mathcal{T}^*]}   \|u_z^\alpha(t)\|^2\Big) = 0,$$ which contradicts to (\ref{BL-2}).
\end{proof}

\begin{remark}
	By considering the convergence for the whole system, i.e., the convergence of the strong solution of system (\ref{EQO1-1})--(\ref{EQO1-5}) to the corresponding
	solution of system (\ref{EQO-1})--(\ref{EQO-5}), we can establish similar blow-up criterion for system (\ref{EQO-1})--(\ref{EQO-5}).
\end{remark}

\noindent
\section*{Acknowledgments}
The work of E.S.T.\ was supported in part by the Einstein Stiftung/Foundation - Berlin, through the Einstein Visiting Fellow Program, and by the John Simon Guggenheim Memorial Foundation.


\begin{thebibliography}{99}


\bibitem{AR75}  R.A. Adams, {\em Sobolev Spaces},
Academic Press, New York, 1975.

\bibitem{BL72} W. Blumen, {\em Geostrophic adjustment,} Rev. Geophys. Space Phys. {\bf 10} (1972), 485--528.

\bibitem{BR99} Y. Brenier, {\em Homogeneous hydrostatic flows with convex velocity profiles,} Nonlinearity,
{\bf 12}(3):495--512, 1999.

\bibitem{BR03} Y. Brenier, {\em Remarks on the derivation of the hydrostatic Euler equations,} Bull. Sci. Math.,
{\bf 127}(7):585--595, 2003.

\bibitem{BGMR03} D. Bresch, F. Guill\'en-Gonz\'alez, N. Masmoudi and M.A. Rodr\'iguez-Bellido, {\em On the uniqueness of weak solutions of the two-dimensional primitive equations,} Differential Integral Equations, {\bf 16} (2003), 77--94.

\bibitem{BKL04} D. Bresch, A. Kazhikhov and J. Lemoine, {\em On the two-dimensional hydrostatic Navier-Stokes equations,} SIAM J. Math. Anal., {\bf 36} (2004), 796--814.

\bibitem{CGT18} C. Cao, Y. Guo and E. S. Titi, {\em
	 Global strong solutions for the three-dimensional Hasegawa-Mima model with
	 partial dissipation,} J. Math. Phys. {\bf 59} (2018), no. 7, 071503, 12pp.

\bibitem{CINT15} C. Cao, S. Ibrahim, K. Nakanishi and E. S. Titi, {\em
	Finite-time blowup for the
	inviscid primitive equations of oceanic and atmospheric dynamics,} Comm. Math.
Phys., {\bf 337} (2015), 473--482.


\bibitem{CLT14a} C. Cao, J. Li and E. S. Titi, {\em
	Local and global well-posedness of strong solutions to
	the 3D primitive equations with vertical eddy diffusivity,} Arch. Rational Mech.
Anal., {\bf 214} (2014), 35--76.

\bibitem{CLT14b} C. Cao, J. Li and E. S. Titi, {\em
	Global well-posedness of strong solutions to the 3D
	primitive equations with horizontal eddy diffusivity,} AJ. Differential Equations,
{\bf 257} (2014), 4108--4132.

\bibitem{CLT16} C. Cao, J. Li and E. S. Titi, {\em
	Global well-posedness of the 3D primitive equations
	with only horizontal viscosity and diffusivity,} Comm. Pure Appl. Math.,
{\bf 69} (2016), 1492--1531.

\bibitem{CLT17} C. Cao, J. Li and E. S. Titi, {\em
	Strong solutions to the 3D primitive equations
	with only horizontal dissipation: Near $H^1$
	initial data,} J. Funct. Anal. (2017),
http://dx.doi.org/10.1016/j.jfa.2017.01.018.

\bibitem{CLT17b} C. Cao, J. Li and E. S. Titi, {\em
	Global well-posedness of the 3D primitive equations with horizontal viscosity and
	vertical diffusivity,} arXiv:1703.02512v1 [math.AP] (2017)

\bibitem{CW11} C. Cao and J. Wu, {\em
	Global regularity for the 2D anisotropic Boussinesq Equations with vertical dissipation,} Arch. Ration. Mech. Anal., {\bf 208} (2013) 985--1004.

\bibitem{CT03} C. Cao and E. S. Titi, {\em
	Global well-posedness and finite-dimensional global attractor
	for a 3-D planetary geostrophic viscous model,} Comm. Pure Appl. Math.,
{\bf 56} (2003), 198--233.

\bibitem{CT07} C. Cao and E. S. Titi, {\em
	Global well-posedness of the three-dimensional viscous primitive
	equations of large scale ocean and atmosphere dynamics,} Ann. of Math.,
{\bf 166} (2007), 245--267.

\bibitem{CT10} C. Cao and E. S. Titi, {\em
	Regularity "in Large" for the 3D Salmon's Planetary Geostrophic Model of Ocean Dynamics,} available at arXiv:1012.5656.

\bibitem{CT12} C. Cao and E. S. Titi, {\em
	Global well-posedness of the 3D primitive equations with partial
	vertical turbulence mixing heat diffusivity,} Comm. Math. Phys., {\bf 310} (2012),
537--568.

\bibitem{CLT06}  Y. Cao, E. Lunasin and E. S. Titi, {\em
	Global well-posedness of the three-dimensional viscous and
	inviscid simplified Bardina turbulence models,}  Commun. Math. Sci. {\bf 4} (2006), no. 4, 823--848.

\bibitem{CDGG} J.-Y. Chemin, B. Desjardines, I. Gallagher and
E. Grenier, {\em Anisotropy and dispersion in rotating fluids ,}
Nonlinear Partial Differential Equations and their Applications,
Coll\`ege de France Seminar, Studies in Mathematics and its
Applications, {\bf 31} (2002), 171--191.

\bibitem{CF88} P. Constantin and C. Foias, {\em Navier-Stokes
	Equations,}
The University of Chicago Press, 1988.

\bibitem{EE97} W. E and B. Enquist, {\em Blow up of solutions to the unsteady Prandtl equation,}
Comm. Pure Appl.
Math. {\bf 50} (1997), 1287--1293. MR1476316 (99c:35196)

\bibitem{GA94} G.P. Galdi, {\em An Introduction to the Mathematical
	Theory of the Navier-Stokes Equations,} Vol. I \& II,
Springer-Verlag, 1994.

\bibitem{GD10} D. G\'erard-Varet and E. Dormy, {\em On the ill-posedness of the Prandtl equation,} J. Amer. Math. Soc., {\bf 23}(2):591--609, 2010.

\bibitem{GMV18} D. G\'erard-Varet and N. Masmoudi and V. Vicol {\em Well-posedness of the hydrostatic Navier-Stokes equations,} J. Math. Fluid Mech. {\bf 14} (2018apr), no. 2, 355--361,
available at arXiv:1804.04489.

\bibitem{GN12} D. G\'erard-Varet and T. Nguyen, {\em Remarks on the ill-posedness of the Prandtl equation,} Asymptotic Analysis, {\bf 77}:71--88, 2012.

\bibitem{GI76} A.E. Gill, {\em Adjustment under gravity in a rotating channel,}
J. Fluid Mech. {\bf 103} (1976), 275--295.

\bibitem{GI82} A.E. Gill, {\em Atmosphere--Ocean Dynamics,}  Academic Press New
York, 1982.

\bibitem{GR99} E. Grenier, {\em On the derivation of homogeneous hydrostatic equations,}  M2AN Math. Model. Numer. Anal., {\bf 33}(5):965--970, 1999.

\bibitem{GMR01} F. Guill\'en-Gonz\'alez, N. Masmoudi and M.A. Rodr\'iguez-Bellido, {\em Anisotropic estimates and strong solutions of the primitive equations,} Differ. Integral Equ., {\bf 14} (2001), 1381--1408.

\bibitem{HN16}  D. Han-Kwan and T. Nguyen, {\em  Illposedness of the hydrostatic Euler and
	singular Vlasov equations}, Arch. Ration. Mech. Anal., {\bf 221} (2016), no. 3, 1317--1344.

\bibitem{Hieber-Kashiwabara}
M. Hieber and  T. Kashiwabara, {\em Global well-posedness of the three-dimensional primitive equations in $L^p$-space}, Arch. Rational Mech. Anal., {\bf 221} (2016), 1077--1115.


\bibitem{HO93}  A.J. Hermann and W.B. Owens, {\em Energetics of gravitational adjustment
	in mesoscale chimneys}, J. Phys. Oceanogr. {\bf 23} (1993), 346--371.

\bibitem{HO04} J.R. Holton, {\em  An Introduction to Dynamic Meteorology,}  4-th ed.
Elsevier Academic Press, 2004.

\bibitem{K06} G. M. Kobelkov, {\em
	Existence of a solution in the large for the 3D large-scale ocean
	dynamics equaitons,} C. R. Math. Acad. Sci. Paris, {\bf 343} (2006), 283--286.

\bibitem{KP97} A.C. Kuo and L.M. Polvani, {\em Time--dependent fully nonlinear geostrophic adjustment,}
J. Phys. Oceanogr. {\bf 27} (1997), 1614--1634.

\bibitem{KMVW14}  I. Kukavica, N. Masmoudi, V. Vicol and T. Wong, {\em
	 On the local well-posedness of the Prandtl
	 and the hydrostatic Euler equations with multiple monotonicity regions,} SIAM J. Math. Anal.,
	{\bf 46}(6):3865--3890, 2014.

\bibitem{KTVZ11} I. Kukavica, R. Temam, V. Vicol, and M. Ziane, {\em
	Local existence and uniqueness for the
	hydrostatic Euler equations on a bounded domain,} J. Differential Equations, {\bf 250}(3):1719--1746, 2011.

\bibitem{KV13} I. Kukavica and V. Vicol, {\em
	On the local existence of analytic solutions to the Prandtl boundary
	layer equations,} Commun. Math. Sci., {\bf 11}(1):269--292, 2013.


\bibitem{KZ07} I. Kukavica and M. Ziane, {\em
	On the regularity of the primitive equations of the ocean,} C. R. Math. Acad. Sci. Paris, {\bf 345} (2007), 257--260.

\bibitem{KZ072} I. Kukavica and M. Ziane, {\em
	On the regularity of the primitive equations,} Nonlinearity, {\bf 20} (2007), 2739--2753.

\bibitem{LPTW18} A. Larios, M. Petersen, E.S. Titi and B. Wingate, {\em
	A computational investigation of the finite-time blow-up of the 3D
	incompressible Euler equations based on the Voigt regularization,} Theor. Comp. Fluid Dyn. {\bf 32} (1) (2018) 23--34.

\bibitem{LT10} A. Larios and E.S. Titi, {\em
	On the higher-order global regularity of the inviscid Voigt-regularization of three-dimensional hydrodynamic models,}  Discrete Contin. Dyn. Syst. Ser. B {\bf 14}(2/3 \#15), 603 (2010).

\bibitem{LT16} J. Li and E. S. Titi, {\em
	Recent  advances  concerning  certain  class  of  geophysical  ﬂows,} in:  Handbook  of Mathematical Analysis in Mechanics of Viscous  Fluids, 2016, arXiv:1604.01695  [math.AP].

\bibitem{LTW92a} J. L. Lions, R. Temam and S. Wang, {\em
	New formulations of the primitive equations
	of the atmosphere and appliations,} Nonlinearity, {\bf 5} (1992), 237--288.

\bibitem{LTW92b} J. L. Lions, R. Temam and S. Wang, {\em
	On the equations of the large-scale ocean,} Nonlinearity, {\bf 5} (1992), 1007--1053.

\bibitem{LTW95} J. L. Lions, R. Temam and S. Wang, {\em
	Mathematical study of the coupled models of
	atmosphere and ocean (CAO III),} J. Math. Pures Appl., {\bf 74} (1995), 105--163.

\bibitem{MW12}  N. Masmoudi and T. Wong, {\em
	On the $H^s$ theory of hydrostatic Euler equations,} Arch. Ration.
Mech. Anal., {\bf 204}(1):231--271, 2012.

\bibitem{MW15}  N. Masmoudi and T. Wong, {\em
	Local-in-time existence and uniqueness of solutions to the Prandtl
	equations by energy methods,} Comm. Pure Appl. Math., {\bf 68}(10):1683--1741, 2015.

\bibitem{OL66}  O. Oleinik, {\em
	 On the mathematical theory of boundary layer for an unsteady flow of incompressible fluid,} J. Appl. Math. Mech., {\bf 30}:951--974 (1967), 1966.


\bibitem{PJ87} J. Pedlosky, {\em Geophysical Fluid Dynamics,}
Springer-Verlag, New York, 1987.

\bibitem{PZ05} R. Plougonven  and V. Zeitlin,  {\em Lagrangian approach to the geostrophic
	adjustment of frontal anomalies in a stratified fluid,}  Geophys. Astrophys. Fluid Dyn. {\bf 99} (2005), 101--135.

\bibitem{RE09} M. Renardy,  {\em Ill-posedness of the hydrostatic Euler and Navier-Stokes equations,} Arch. Ration.
Mech. Anal., {\bf 194}(3):877--886, 2009.

\bibitem{RO38} C.G. Rossby,  {\em The mutual adjustment of pressure
	and velocity distribution in certain simple
	current systems, II,} J. Mar. Res. {\bf 1}  (1938), 239--263.

 \bibitem{SV97} R. Samelson and G. Vallis, {\em A simple friction and diffusion scheme for planetary geostrophic basin models,} J. Physical Oceanography {\bf 27} (1997), 186--194.

\bibitem{S98} R. Salmon, {\em Lectures on Geophysical Fluid
 	Dynamics,} Oxford University Press, New York, Oxford,
 1998.

\bibitem{SIMON} J. Simon, {\em Compact sets in the space $L^p(0,T;B)$,} Ann. Mat. Pure Appl.,
{\bf 146} (1987), 65--96.

\bibitem{TT83} R. Temam, {\em Navier-Stokes Equation and Nonlinear Functional Analysis,} SIAM, {\bf 68}, Philadelphia, 1983.

\bibitem{TT84} R. Temam, {\em Navier-Stokes Equations, Theory and
Numerical Analysis,} North-Holland, 1984.

\bibitem{TZ03} R. Temam and M. Ziane, {\em Some mathematical problems
in geophysical fluid dynamics}, Handbook of Mathematical Fluid
Dynamics, 2003.

\bibitem{W12} T. K. Wong, {\em Blowup of solutions of the hydrostatic Euler equations}, Proc. Amer. Math. Soc. {\bf 143} (2015), no. 3, 1119--1125.


\end{thebibliography}
\end{document}